\documentclass[11pt]{article}

\usepackage{amsfonts}
\usepackage{amscd}
\usepackage{amssymb}
\usepackage{amsthm}
\usepackage{amsmath, mathdots} 
\usepackage{pdfsync}
\usepackage{enumerate}

\usepackage{etex} 
\usepackage{pictexwd}
\usepackage{tikz}

\theoremstyle{plain}
\newtheorem{thm}{Theorem}[section]

\newtheorem{prop}[thm]{Proposition}
\newtheorem{cor}[thm]{Corollary}
\theoremstyle{definition}

\newtheorem{remark}[thm]{Remark}
\theoremstyle{example}
\newtheorem*{example}{Example}

\theoremstyle{remark}

\numberwithin{equation}{section}

\def\cB{\mathcal{B}}

\def\cH{\mathcal{H}}

\def\cO{\mathcal{O}}

\def\cR{\mathcal{R}}

\def\cW{\mathcal{W}}

\def\CC{\mathbb{C}}

\def\ZZ{\mathbb{Z}}

\def\fg{\mathfrak{g}} 
\def\fh{\mathfrak{h}}

\def\fn{\mathfrak{n}}

\def\fgl{\mathfrak{gl}}  
\def\fsl{\mathfrak{sl}}  
\def\fso{\mathfrak{so}}  
\def\fsp{\mathfrak{sp}}  

\def\ad{\mathrm{ad}}

\def\dim{\mathrm{dim}} 
\def\End{\mathrm{End}} 
\def\ev{\mathrm{ev}} 
\def\Hom{\mathrm{Hom}} 
 
\def\id{\mathrm{id}}

\def\qtr{\mathrm{qtr}}

\def\Res{\mathrm{Res}}

\def\tr{\mathrm{tr}} 
 
\def\wt{\mathrm{wt}}

\def\half{\hbox{$\frac12$}}

\def\<{\langle}
\def\>{\rangle}

\usepackage{color}
\definecolor{dred}{rgb}{.65, 0, 0.15}
\definecolor{zajj}{rgb}{0.4, 0, .6}
\definecolor{gray}{rgb}{0.6, .6, .6}

\makeatletter
\renewcommand{\@makefnmark}{\mbox{\textsuperscript{}}}
\makeatother

\title{Affine and degenerate affine BMW algebras:\\ Actions on tensor space}
\author{
Zajj Daugherty\\
Department of Mathematics \\
Dartmouth College\\
Hanover, NH 03755 USA\\
zajj.b.daugherty@dartmouth.edu \\ 
\and
Arun Ram \\
Department of Mathematics and Statistics \\
University of Melbourne \\
Parkville VIC 3010 Australia \\
aram@unimelb.edu.au \\
\and
Rahbar Virk \\
Department of Mathematics, Campus Box 395 \\
Boulder, Colorado 80309 USA \\
rahbar.virk@colorado.edu\\
}

\date{}

\newcommand{\comment}[1]{}
\usepackage[pdftex,bookmarks]{hyperref}

\begin{document}

\maketitle

\setcounter{tocdepth}{3}
\tableofcontents

\begin{abstract}
The affine and degenerate affine  Birman-Murakami-Wenzl (BMW) algebras arise naturally in the 
context of Schur-Weyl duality for orthogonal and symplectic quantum 
groups and Lie algebras, respectively.  Cyclotomic BMW algebras, affine and  cyclotomic Hecke algebras, and their degenerate versions are quotients.  In this paper we explain how the affine and degenerate affine BMW algebras are tantalizers (tensor power centralizer algebras) by defining actions of the affine braid group and the degenerate affine braid algebra on tensor space and showing that, in important cases, these actions induce actions of the affine and degenerate affine BMW algebras.  We then exploit the connection to quantum groups and Lie algebras to determine universal parameters for the affine and degenerate affine BMW algebras.  Finally, we show that  the universal
parameters are central elements---the higher Casimir elements for orthogonal and symplectic 
enveloping algebras and quantum groups. 
\end{abstract}


\section*{Introduction}
This paper is a continuation of our study of the affine Birman-Murakami-Wenzl (BMW) algebras
$W_k$ and their degenerate versions $\cW_k$.
In \cite{DRV} we defined the algebras $W_k$ and $\cW_k$ and determined their centers.  Each of these algebras  contains a commuting family of ``Jucys-Murphy elements''; following
Nazarov and Beliakova-Blanchet \cite{Naz,BB} we derived generating function formulas for specific
central elements $z_k^{(\ell)}\in \cW_k$ and $Z_k^{(\ell)}\in W_k$ in terms of the
Jucys-Murphy elements.

In this paper we show that the algebras $W_k$ and $\cW_k$ have a natural action on tensor space which provides a Schur-Weyl duality with the quantum group and enveloping algebras of classical type Lie algebras.  In particular, the algebras $W_k$ and $\cW_k$  arise in orthogonal and symplectic type, though we treat all the classical type cases uniformly. In complete analogy with the fact that affine BMW algebra is a quotient of the group algebra $CB_k$ of the affine braid group of type A, the degenerate affine BMW algebra is a quotient of the {\sl degenerate affine braid algebra} $\cB_k$ which we first defined in \cite{DRV}.  This analogy extends to the Schur-Weyl duality.  In Theorem \ref{degaction}, we show that there is a natural action of $\cB_k$ on tensor space which commutes with an \emph{arbitrary} finite-dimensional complex reductive Lie algebra $\fg$.  In Theorem \ref{z1action} we show that,
when $\fg$ is of classical type and the tensor space is constructed from the $n$-dimensional defining representation, the action of $\cB_k$ becomes an action of familiar algebras : the degenerate affine BMW algebra arises when $\fg=\fso_n$ or $\fsp_n$, and the degenerate affine Hecke algebra arises when $\fg=\fgl_n$ or $\fsl_n$.

The affine and degenerate affine  BMW algebras depend on the choice of an infinite number of parameters.
This is analogous to the way that the Iwahori-Hecke algebra depends on one parameter, often 
called $q$.  Unfortunately, the infinite collection of parameters for the BMW algebras is not free; significant work has been done on when a collection is admissible \cite{AMR,WY1, WY2,Go2, Go3, GH1, GH2, GH3, Yu}. In this work, we take a different point of view and produce \emph{universal} parameters for the affine and degenerate affine BMW algebras.  These universal parameters are symmetric functions which satisfy the admissibility conditions. In future work, we hope to show via representation theory that every choice of admissible complex parameters is a specialization of our universal parameters.  

To compute the symmetric functions which arise as universal parameters, we use the Schur-Weyl duality to naturally identify them as elements
of the center of the corresponding symplectic or orthogonal enveloping algebra or
quantum group (which, by the Harish-Chandra isomorphism, is isomorphic to a ring of symmetric
functions).  Specifically, in Theorem \ref{z1action} and Theorem \ref{Z1action}
we execute computations which push the recursive formulas of Nazarov \cite{Naz} and Beliakova-Blanchet \cite{BB} to the other side of the Schur-Weyl duality.  This produces explicit formulas for the Harish-Chandra images of the corresponding central elements $z_V^{(\ell)}$ and $Z_V^{(\ell)}$ of the 
orthogonal and symplectic enveloping algebras and quantum groups. 
These computations are related to the calculations in, for example, \cite{Naz}, \cite[Ch.\ 7]{Mo} and \cite{MR}.

In Section \ref{sec:connections} we show that the central elements $z_V^{(\ell)}$ and 
$Z_V^{(\ell)}$are the
natural higher Casimir elements for orthogonal and symplectic enveloping algebras
and quantum groups. In fact, 
we are able to show that the universal parameters $z_V^{(\ell)}$
of the degenerate affine BMW algebras coincide exactly with the
higher Casimirs for orthogonal and symplectic Lie algebras given by Perelomov-Popov
\cite{PP1, PP2}.  Expositions of the Perelomov-Popov results are also found in \cite[\S127]{Zh} and
\cite[\S7.1]{Mo}.  Another 
computation shows that the universal parameters $Z_V^{(\ell)}$ of
the affine BMW algebras coincide with the central elements in quantum groups defined
by Reshetikhin-Takhtajan-Faddeev central elements defined in \cite[Theorem 14]{RTF}.
To execute our computation we have relied on a remarkable identity of 
Baumann \cite[Theorem 1]{Bau}.

\smallskip\noindent
\textbf{Acknowledgements:}  Significant work on this paper was done while the authors 
were in residence at the Mathematical Sciences Research Institute (MSRI) in 2008. 
We thank MSRI for hospitality, support, and a wonderful working environment.  We thank F.\ Goodman and A.\ Molev for their helpful comments and references.  This research has been partially supported by the National Science Foundation DMS-0353038 and the Australian Research Council DP0986774 and DP120101942.

\section{Actions of general type tantalizers}
\label{sec:general-actions}

Our goal in Section \ref{sec:classical-actions}
 is to provide a tensor space action of the affine Birman-Murakami-Wenzl (BMW) algebra $W_k$ and its degenerate version $\cW_k$ by way of the group algebra of the affine braid group $CB_k$ and the degenerate affine braid algebra $\cB_k$, respectively. The definition of the degenerate affine braid algebra $\cB_k$ makes the Schur-Weyl duality framework completely analogous in both the affine and degenerate affine cases.  

In this section, we define $CB_k$ and $\cB_k$ and show that both act on tensor space of the form $M \otimes V^{\otimes k}$.  In the degenerate affine case this action commutes with complex reductive Lie algebras $\fg$; in the affine case this action commutes with the Drinfeld-Jimbo quantum group $U_h\fg$.  As we will see in Section \ref{sec:classical-actions}, the affine and degenerate affine BMW algebras arise when $\fg$ is orthogonal or symplectic and $V$ is the defining representation; similarly, the degenerate affine Hecke algebras arise when $\fg$ is of type $\fgl_n$ or $\fsl_n$ and  $V$ is the defining representation. In the case when $M$ is the trivial representation and $\fg$ is $\fso_n$, the elements $y_1, \dots, y_k$ in $\cB_k$ become the Jucys-Murphy elements for the Brauer algebras used in \cite{Naz}; in the case that  $\fg = \fgl_n$, these become the classical Jucys-Murphy elements in the group algebra of the symmetric group.

The action of the degenerate affine braid algebra $\cB_k$ and the action of 
the affine braid group $B_k$ on $M\otimes V^{\otimes k}$ each provide \emph{Schur functors}
\begin{equation}\label{Schurfunctor}
\begin{matrix}
F^\lambda_{V} \colon &\{ \hbox{$U$-modules}\} &\longrightarrow &\begin{cases}\{\hbox{$\cB_k$-modules}\}, & \text{if $U = U\fg$,}\\ \{\hbox{$B_k$-modules}\}, & \text{if $U = U_h\fg$.}\end{cases} \\~\\
&M &\longmapsto &\Hom_{U}(M(\lambda), M\otimes V^{\otimes k}) \\
\end{matrix}
\end{equation}
where in each case $\Hom_{U}(M(\lambda), M\otimes V^{\otimes k})$  is the vector space
of highest weight vectors of weight $\lambda$ in $M\otimes V^{\otimes k}$.
These ubiquitous functors transfer representation theoretic information back and forth
either between $U\fg$ and  $\cB_k$ 
or  between  $U_h\fg$ and $CB_k$.

%

\subsection{The degenerate affine braid algebra action}
\label{sec:affine-braid-algebra}


Let $C$ be a commutative ring and let $S_k$ denote the symmetric group on $\{1,\ldots, k\}$.
For $i\in\{1,\ldots, k\}$, 
write $s_i$ for the transposition in $S_k$ that switches $i$ and $i+1$.
The \emph{degenerate affine braid algebra} is the algebra $\cB_k$ over $C$ generated by
\begin{equation}\label{gensA}
t_u\quad (u\in S_k),
\qquad
\kappa_0, \kappa_1,
\qquad\hbox{and}\qquad
y_1,\ldots, y_k,
\end{equation}
with relations
\begin{equation}\label{easyrels}
t_ut_v = t_{uv}, \qquad y_iy_j = y_jy_i,
\quad
\kappa_0\kappa_1=\kappa_1\kappa_0,
\quad
\kappa_0 y_i = y_i \kappa_0,
\quad
\kappa_1 y_i = y_i\kappa_1,
\end{equation}
\begin{equation}\label{kyrels}
\quad
\kappa_0 t_{s_i} = t_{s_i}\kappa_0,
\qquad 
\kappa_1t_{s_1}\kappa_1t_{s_1}
=t_{s_1}\kappa_1t_{s_1}\kappa_1,
\qquad\hbox{and}\qquad
\kappa_1t_{s_j} = t_{s_j}\kappa_1,\ \hbox{for $j\ne 1$,}
\end{equation}
\begin{equation}\label{rel:graded_braid3}
t_{s_i}(y_i + y_{i+1}) = (y_i + y_{i+1})t_{s_i},
\quad\hbox{and}\qquad
y_jt_{s_i} = t_{s_i} y_j \quad \hbox{for  $j \neq i, i+1$,}
\end{equation}
and, for $i=1,\ldots, k-2$,
\begin{equation}
\label{rel:graded_braid5}
t_{s_i} t_{s_{i+1}}\gamma_{i,i+1}t_{s_{i+1}} t_{s_i} 
= \gamma_{i+1,i+2},
\qquad\hbox{where}\quad
\gamma_{i,i+1} =  y_{i+1}-t_{s_i}y_it_{s_i}.
\end{equation}
This presentation highlights the ``Jucys-Murphy'' elements $y_1,\ldots, y_k$ for the degenerate 
affine BMW algebra $\cW_k$ as in \cite{Naz}. However, the algebra $\cB_k$ also admits the following presentation, which highlights its natural action on tensor space (as we will see in Theorem \ref{degaction}).

\begin{thm} \label{presthm} \cite[Theorem 2.1]{DRV}
The degenerate affine braid algebra $\cB_k$ has another presentation by generators
\begin{equation}\label{gensB}
t_u\quad (u\in S_k), \qquad
\kappa_0,\ldots, \kappa_k
\qquad\hbox{and}\qquad
\gamma_{i,j},\ \ \hbox{for $0\le i,j\le k$ with $i\ne j$,}
\end{equation}
and relations
\begin{equation}\label{tildeBrels1}
t_ut_v = t_{uv}, \qquad 
t_w \kappa_i t_{w^{-1}} = \kappa_{w(i)},
\qquad
t_w \gamma_{i,j} t_{w^{-1}} = \gamma_{w(i), w(j)},
\end{equation}
\begin{equation}\label{rel:graded_braid4}
\kappa_i\kappa_j = \kappa_j\kappa_i,
\qquad \kappa_i \gamma_{\ell,m} = \gamma_{\ell,m}\kappa_i,
\end{equation}
\begin{equation}\label{rel:graded_braid6}
\gamma_{i,j}=\gamma_{j,i},\qquad
\gamma_{p,r}\gamma_{\ell,m} = \gamma_{\ell,m}\gamma_{p,r},
\qquad\hbox{and}\qquad
\gamma_{i,j}(\gamma_{i,r}+\gamma_{j,r}) = (\gamma_{i,r}+\gamma_{j,r})\gamma_{i,j},
\end{equation}
for $p\ne \ell$ and $p\ne m$ and $r\ne \ell$ and $r\ne m$
and $i\ne j$, $i\ne r$ and $j\ne r$.
\end{thm}

The conversion between the two presentations is given by the formulas
$\kappa_0=\kappa_0$,
$\kappa_1=\kappa_1$,
$t_w=t_w$,
\begin{equation}\label{gensAtogensB}
y_j=\hbox{$\frac12$}\kappa_j + \sum_{0\le \ell<j} \gamma_{\ell, j},
\end{equation}
and the formulas in \eqref{tildeBrels1}.
Set 
\begin{equation}\label{kappadefn}
c_0 = \kappa_0
\qquad\hbox{and}\qquad
c_j= \kappa_0+ 2(y_1+\ldots+y_j),
\end{equation}
for $j=1,2,\ldots, k$.
Then $c_0,\ldots, c_k$ pairwise commute,
\begin{equation}\label{cjconvert}
y_j = \hbox{$\frac{1}{2}$}(c_j - c_{j-1})
\qquad\hbox{and}\qquad
c_j = \sum_{i=0}^j \kappa_i + 2\sum_{0\le \ell<m\le j} \gamma_{\ell,m}.
\end{equation}

Let $\fg$ be a finite-dimensional complex Lie algebra with a symmetric nondegenerate
$\mathrm{ad}$-invariant bilinear form i.e.,
$$\langle,\rangle\colon \fg\otimes \fg \to \CC,
\qquad\hbox{with}\qquad
\langle [x_1,x_2], x_3\rangle + \langle x_2, [x_1,x_3]\rangle = 0
$$
and $\langle x_1,x_2\rangle = \langle x_2,x_1\rangle$, 
for $x_1,x_2,x_3\in \fg$.
Let $B=\{b_1,\ldots, b_n\}$ be a basis for $\fg$ and let $\{b_1^*,\ldots, b_n^*\}$ 
be the dual basis with respect to $\langle, \rangle$.  The \emph{Casimir} is 
\begin{equation}\label{casdefn}
\kappa 
= b_1b_1^*+\cdots+b_nb_n^* = \sum_{b\in B} bb^*
\qquad\hbox{and}\qquad
\kappa\in Z(U\fg),
\end{equation}
where $ Z(U\fg)$ is the center of  the enveloping algebra $U\fg$ (see \cite[I \S 3 Prop.\ 11]{Bou}). 
Since the coproduct on $\fg$ is defined by $\Delta(x) = x\otimes 1 + 1\otimes x$ for $x\in \fg$, 
\begin{equation}\label{gamma_defn}
\Delta(\kappa)=\kappa\otimes 1 + 1\otimes \kappa + 2\gamma,
\qquad\hbox{where}\quad \gamma = \sum_{b\in B} b\otimes b^*.
\end{equation}

\begin{thm}\label{degaction}
Let $\fg$ be a finite-dimensional complex Lie algebra with a symmetric nondegenerate
$\mathrm{ad}$-invariant bilinear form $\langle , \rangle$, and let $U\fg$ be the universal enveloping
algebra.  Let $C=Z(U\fg)$ be the center of $U\fg$, $\kappa$ be the Casimir in $C$, and
$\gamma = \sum_b b \otimes b^*$ as in \eqref{gamma_defn}.
Let $M$ and $V$ be $\fg$-modules and let $s_1 \cdot (u \otimes v) = v \otimes u$,
for $u,v \in V$.

~

\noindent The degenerate affine braid algebra $\cB_k$ acts on $M \otimes V^{\otimes k}$ via 
$\Phi:  \cB_k  \to \End(M \otimes V^{\otimes k})$
defined by
\begin{equation*}\label{oncomponents}
\begin{array}{l}
\hbox{$\Phi(t_{s_i})=\id_M \otimes \id_V^{\otimes i-1} \otimes s_1 \otimes \id_V^{\otimes k-i-1}$
so that $S_k$ acts by permuting tensor factors of $V$,}\\
\hbox{$\Phi(\kappa_i)$ is $\kappa$ acting in the $i$th factor 
of $V$ in $M\otimes V^{\otimes k}$ and $\Phi(\kappa_0)$ is $\kappa$ acting on $M$,} \\
\hbox{$\Phi(\gamma_{\ell,m})$ is $\gamma$ acting in the $\ell$th 
and $m$th factors of $V$ in $M\otimes V^{\otimes k}$, }\\
\hbox{$\Phi(\gamma_{0,\ell})$ is $\gamma$ acting on $M$ and the $\ell$th factor of $V$ in $M \otimes V^{\otimes k}$,} \\
\hbox{$\Phi(c_i)$ is $\kappa$ acting on $M$ and the first $i$ factors of $V$}, \\
\hbox{$\Phi(z)= z\otimes \id_V^{\otimes k}$ for $z\in C$.}
\end{array}
\end{equation*}
This action of $\cB_k$ commutes with the $U\fg$-action on $M\otimes V^{\otimes k}$.
\end{thm}
\begin{proof}

Since $\kappa\in Z(U\fg)$ the operators $\Phi(\kappa_i)$ 
are in $\End_\fg(M\otimes V^{\otimes k})$.
From the relation $2\gamma = \Delta(\kappa) - \kappa\otimes 1 - 1\otimes \kappa$
it also follows that
$\Phi(\gamma_{\ell,m})\in \End_{\fg}(M\otimes V^{\otimes k})$.

All of the relations in Theorem \ref{presthm} except the last
relations in \eqref{rel:graded_braid6} follow from
consideration of the tensor factors being acted upon.
The last relations in \eqref{rel:graded_braid6} are established by the computation
\begin{align*}
\gamma_{1,2}(\gamma_{1,3}&+\gamma_{2,3}) (v\otimes w\otimes z)
= \gamma_{1,2}\left(\sum_{i} b_iv\otimes w\otimes b_i^*z
+ v\otimes b_iw\otimes b_i^*z\right) \\
&=\gamma_{1,2}\left(\sum_{i} \Delta(b_i)\otimes b_i^*\right)(v\otimes w\otimes z)
=\left(\sum_{i} \Delta(b_i)\otimes b_i^*\right)\gamma_{1,2}(v\otimes w\otimes z) \\
&= (\gamma_{1,3}+\gamma_{2,3})\gamma_{1,2}(v\otimes w\otimes z),
\end{align*}
for $v,w,z\in V$. 
Recursively applying the coproduct formula from \eqref{gamma_defn} connects the
action of $c_i$  with the action of $\kappa_i$ and $\gamma_{\ell,m}$
as in the second formula in \eqref{cjconvert}.
\end{proof}

For a $U\fg$-module $M$ let
$$\begin{matrix} \kappa_M\colon &M &\longrightarrow &M \\
&m &\longmapsto &\kappa m
\end{matrix}
\qquad\qquad\hbox{where $\kappa$ is the Casimir}
$$  
as in \eqref{casdefn}.
If $M$ is a $U\fg$-module generated by a highest weight vector $v_\lambda^+$ of weight $\lambda$ 
then 
\begin{equation}\label{Casvalue}
\kappa_M = \langle \lambda, \lambda+2\rho\rangle \id_M,
\qquad\hbox{where}\quad \rho = \hbox{$\frac12$}\sum_{\alpha\in R^+} \alpha
\end{equation}
is the half-sum of the positive roots (see \cite[VIII \S 2 no.\ 3 Prop.\ 6 and VIII \S 6 no.\ 4 Cor.\ to Prop.\ 7]{Bou}).
By equation \eqref{gamma_defn}, if $M=L(\mu)$ and
$N=L(\nu)$ are finite-dimensional irreducible $U\fg$-modules of highest weights
$\mu$ and $\nu$ respectively, 
then $\gamma$ acts on the
$L(\lambda)$-isotypic component of the decomposition
$L(\mu)\otimes L(\nu) \cong \bigoplus_\lambda L(\lambda)^{\oplus c_{\mu\nu}^\lambda}$
by the constant 
\begin{equation}\label{gamma_value} 
\hbox{$\frac12$}\big(\langle \lambda,\lambda+2\rho\rangle
- \langle \mu,\mu+2\rho\rangle - \langle \nu,\nu+2\rho\rangle\big).
\end{equation}
Pictorially,
$$\Phi(c_j) =\!\!%
\beginpicture
\setcoordinatesystem units <0.75cm,0.25cm> 
\setplotarea x from 0 to 5.5, y from -1 to 4   
\plot 0.2 2 2.8 2 /
\plot 0.2 -1 2.8 -1 /
\linethickness=0.5pt                        
\put{$M$} at 0 -3
\put{$M$} at 0 4
\put{$V$} at 1 -3
\put{$V$} at 1 4
\put{$V$} at 3 -3
\put{$V$} at 3 4
\put{$V$} at 4 -3
\put{$V$} at 4 4
\put{$V$} at 5 -3
\put{$V$} at 5 4
\put{$V$} at 6.75 -3
\put{$V$} at 6.75 4
\put{$\otimes$} at 0.5 -3
\put{$\otimes$} at 0.5 3.9
\put{$\otimes$} at 3.5 -3
\put{$\otimes$} at 3.5 3.9
\put{$\otimes$} at 4.5 -3
\put{$\otimes$} at 4.5 3.9
\put{$\otimes\cdots\otimes$} at 2 -3
\put{$\otimes\cdots\otimes$} at 2 3.9
\put{$\otimes\! \cdot\!\cdot\!\cdot\!\otimes$} at 5.85 -3
\put{$\otimes\! \cdot\!\cdot\!\cdot\!\otimes$} at 5.85 3.9
\put{$\bullet$} at 0 -1 
\put{$\bullet$} at 0 2 
\plot 0 -1 0 2 /
\put{$\leftrightsquigarrow\  \kappa\ \leftrightsquigarrow$} at 1.5 .5
\put{$\bullet$} at 3 -1  
\put{$\bullet$} at 3 2 
\plot 3 -1 3 2 /
\put{$\bullet$} at 4 -1 \put{$\bullet$} at 4 2
\put{$\scriptstyle{j}$}[b] at 3 5
\put{$\scriptstyle{j+1}$}[b] at 4 5
\put{$\bullet$} at 5 -1 \put{$\bullet$} at 5 2  \plot 5 -1 5 2 /
\put{$\cdots$} at 5.95 .5
\put{$\bullet$} at 6.75 -1 \put{$\bullet$} at 6.75 2  \plot 6.75 -1 6.75 2 /
\plot 4 2 4 -1 /
\endpicture 
\quad \text{ and} \quad 
\Phi(t_{s_i})  =\!
		\beginpicture
\setcoordinatesystem units <0.75cm,0.25cm> 
			\setplotarea x from 1 to 7.75, y from -1 to 4   
			\linethickness=0.5pt                        
\put{$M$} at 0 -3
\put{$M$} at 0 4
\put{$V$} at 1 -3
\put{$V$} at 1 4
\put{$V$} at 3 -3
\put{$V$} at 3 4
\put{$V$} at 4 -3
\put{$V$} at 4 4
\put{$V$} at 5 -3
\put{$V$} at 5 4
\put{$V$} at 6 -3
\put{$V$} at 6 4
\put{$V$} at 7.75 -3
\put{$V$} at 7.75 4
\put{$\otimes$} at 0.5 -3
\put{$\otimes$} at 0.5 4
\put{$\otimes$} at 3.5 -3
\put{$\otimes$} at 3.5 4
\put{$\otimes$} at 4.5 -3
\put{$\otimes$} at 4.5 4
\put{$\otimes$} at 5.5 -3
\put{$\otimes$} at 5.5 4
\put{$\otimes\cdots\otimes$} at 2 -3
\put{$\otimes\cdots\otimes$} at 2 4
\put{$\otimes\! \cdot\!\cdot\!\cdot\!\otimes$} at 6.85 -3
\put{$\otimes\! \cdot\!\cdot\!\cdot\!\otimes$} at 6.85 3.9
				\put{$\bullet$} at 0 -1 \put{$\bullet$} at 0 2 \plot 0 -1 0 2 /
				\put{$\bullet$} at 1 -1 \put{$\bullet$} at 1 2 \plot 1 -1 1 2 /
				\put{$\cdots$} at 2 .5
				\put{$\bullet$} at 3 -1  \put{$\bullet$} at 3 2 \plot 3 -1 3 2 /
				\put{$\bullet$} at 4 -1 \put{$\bullet$} at 4 2
				\put{$\bullet$} at 5 -1 \put{$\bullet$} at 5 2
				\put{$\bullet$} at 6 -1 \put{$\bullet$} at 6 2  \plot 6 -1 6 2 /
				\put{$\cdots$} at 6.95 .5
				\put{$\bullet$} at 7.75 -1 \put{$\bullet$} at 7.75 2  \plot 7.75 -1 7.75 2 /
				\put{$\scriptstyle{i}$}[b] at 4 5
				\put{$\scriptstyle{i+1}$}[b] at 5 5
				\plot 4 2 5 -1 /
				\plot 5 2 4 -1 /
		\endpicture.
$$
By \eqref{gamma_value} and \eqref{kappadefn}, 
the eigenvalues of $y_j$ are functions of the eigenvalues of the Casimir.


\subsection{The affine braid group action}
\label{s:affbdgp}

The \emph{affine braid group} $B_k$ is the group given by
generators $T_1, T_2, \ldots, T_{k-1}$ and $X^{\varepsilon_1}$, 
with relations
\begin{align}
T_i T_j &= T_j T_i, &\mbox{if } j \neq i \pm 1 \label{affrel1},\\
T_i T_{i+1} T_i &= T_{i+1} T_i T_{i+1}, &\mbox{for } i=1, 2,\ldots, k-2,\label{affrel2}\\
X^{\varepsilon_1}T_1 X^{\varepsilon_1} T_1 
&= T_1 X^{\varepsilon_1} T_1 X^{\varepsilon_1}, \label{affrel3}\\
X^{\varepsilon_1}T_i &= T_i X^{\varepsilon_1},
&\mbox{for } i=2,3,\ldots, k-1. \label{affrel4}
\end{align}
The generators $T_i$ and $X^{\varepsilon_1}$ can be identified with the diagrams
\begin{equation}
T_i = 
\beginpicture
\setcoordinatesystem units <.5cm,.5cm>         
\setplotarea x from -5 to 3.5, y from -2 to 2    
\put{${}^i$} at 0 1.2      %
\put{${}^{i+1}$} at 1 1.2      %
\put{$\bullet$} at -3 .75      %
\put{$\bullet$} at -2 .75      %
\put{$\bullet$} at -1 .75      %
\put{$\bullet$} at  0 .75      
\put{$\bullet$} at  1 .75      %
\put{$\bullet$} at  2 .75      %
\put{$\bullet$} at  3 .75      %
\put{$\bullet$} at -3 -.75          %
\put{$\bullet$} at -2 -.75          %
\put{$\bullet$} at -1 -.75          %
\put{$\bullet$} at  0 -.75          
\put{$\bullet$} at  1 -.75          %
\put{$\bullet$} at  2 -.75          %
\put{$\bullet$} at  3 -.75          %
\plot -4.5 1.25 -4.5 -1.25 /
\plot -4.25 1.25 -4.25 -1.25 /
\ellipticalarc axes ratio 1:1 360 degrees from -4.5 1.25 center 
at -4.375 1.25
\put{$*$} at -4.375 1.25  
\ellipticalarc axes ratio 1:1 180 degrees from -4.5 -1.25 center 
at -4.375 -1.25 
\plot -3 .75  -3 -.75 /
\plot -2 .75  -2 -.75 /
\plot -1 .75  -1 -.75 /
\plot  2 .75   2 -.75 /
\plot  3 .75   3 -.75 /
\setquadratic
\plot  0 -.75  .05 -.45  .4 -0.1 /
\plot  .6 0.1  .95 0.45  1 .75 /
\plot 0 .75  .05 .45  .5 0  .95 -0.45  1 -.75 /
\endpicture
\qquad\hbox{and}\qquad
X^{\varepsilon_1} = 
~~\beginpicture
\setcoordinatesystem units <.5cm,.5cm>         
\setplotarea x from -5 to 3.5, y from -2 to 2    
\put{$\bullet$} at -3 0.75      %
\put{$\bullet$} at -2 0.75      %
\put{$\bullet$} at -1 0.75      %
\put{$\bullet$} at  0 0.75      
\put{$\bullet$} at  1 0.75      %
\put{$\bullet$} at  2 0.75      %
\put{$\bullet$} at  3 0.75      %
\put{$\bullet$} at -3 -0.75          %
\put{$\bullet$} at -2 -0.75          %
\put{$\bullet$} at -1 -0.75          %
\put{$\bullet$} at  0 -0.75          
\put{$\bullet$} at  1 -0.75          %
\put{$\bullet$} at  2 -0.75          %
\put{$\bullet$} at  3 -0.75          %
\plot -4.5 1.25 -4.5 -0.13 /
\plot -4.5 -0.37   -4.5 -1.25 /
\plot -4.25 1.25 -4.25  -0.13 /
\plot -4.25 -0.37 -4.25 -1.25 /
\ellipticalarc axes ratio 1:1 360 degrees from -4.5 1.25 center 
at -4.375 1.25
\put{$*$} at -4.375 1.25  
\ellipticalarc axes ratio 1:1 180 degrees from -4.5 -1.25 center 
at -4.375 -1.25 
\plot -2 0.75  -2 -0.75 /
\plot -1 0.75  -1 -0.75 /
\plot  0 0.75   0 -0.75 /
\plot  1 0.75   1 -0.75 /
\plot  2 0.75   2 -0.75 /
\plot  3 0.75   3 -0.75 /
\setlinear
\plot -3.3 0.25  -4.1 0.25 /
\ellipticalarc axes ratio 2:1 180 degrees from -4.65 0.25  center 
at -4.65 0 
\plot -4.65 -0.25  -3.3 -0.25 /
\setquadratic
\plot  -3.3 0.25  -3.05 .45  -3 0.75 /
\plot  -3.3 -0.25  -3.05 -0.45  -3 -0.75 /
\endpicture
.
\label{braidfig1}
\end{equation}
For $i=1,\ldots, k$ define
\begin{equation}\label{BraidMurphy}
X^{\varepsilon_i}=T_{i-1}T_{i-2}\cdots T_2T_1 
X^{\varepsilon_1}T_1T_2\cdots T_{i-2}T_{i-1} = 
~~\beginpicture
\setcoordinatesystem units <.5cm,.5cm>         
\setplotarea x from -5 to 3.5, y from -2 to 2    
\put{${}^i$} at 1 1.2 
\put{$\bullet$} at -3 0.75      %
\put{$\bullet$} at -2 0.75      %
\put{$\bullet$} at -1 0.75      %
\put{$\bullet$} at  0 0.75      
\put{$\bullet$} at  1 0.75      %
\put{$\bullet$} at  2 0.75      %
\put{$\bullet$} at  3 0.75      %
\put{$\bullet$} at -3 -0.75          %
\put{$\bullet$} at -2 -0.75          %
\put{$\bullet$} at -1 -0.75          %
\put{$\bullet$} at  0 -0.75          
\put{$\bullet$} at  1 -0.75          %
\put{$\bullet$} at  2 -0.75          %
\put{$\bullet$} at  3 -0.75          %
\plot -4.5 1.25 -4.5 -0.13 /
\plot -4.5 -0.37   -4.5 -1.25 /
\plot -4.25 1.25 -4.25  -0.13 /
\plot -4.25 -0.37 -4.25 -1.25 /
\ellipticalarc axes ratio 1:1 360 degrees from -4.5 1.25 center 
at -4.375 1.25
\put{$*$} at -4.375 1.25  
\ellipticalarc axes ratio 1:1 180 degrees from -4.5 -1.25 center 
at -4.375 -1.25 
\plot -3 0.75  -3 -0.1 /
\plot -2 0.75  -2 -0.1 /
\plot -1 0.75  -1 -0.1 /
\plot  0 0.75   0 -0.1 /
\plot -3 -.35  -3 -0.75 /
\plot -2 -.35   -2 -0.75 /
\plot -1 -.35   -1 -0.75 /
\plot  0 -.35    0 -0.75 /
\plot  2 0.75   2 -0.75 /
\plot  3 0.75   3 -0.75 /
\setlinear
\plot -3.2 0.25  -4.1 0.25 /
\plot -2.8 0.25  -2.2 0.25 /
\plot -1.8 0.25  -1.2 0.25 /
\plot -.8 0.25  -.2 0.25 /
\plot  .2 0.25  .5 0.25 /
\plot -3.3 -.25  .5 -.25 /
\ellipticalarc axes ratio 2:1 180 degrees from -4.65 0.25  center 
at -4.65 0 
\plot -4.65 -0.25  -3.3 -0.25 /
\setquadratic
\plot  .5 0.25  .9 .45  1 0.75 /
\plot  .5  -0.25  .9 -0.45 1 -0.75 /
\endpicture.
\end{equation}
The pictorial computation
$$
X^{\varepsilon_j}X^{\varepsilon_i}=
\beginpicture
\setcoordinatesystem units <.5cm,.5cm>         
\setplotarea x from -5.5 to 3.5, y from -2 to 2    
\put{\scriptsize $j$} at 1 2
\put{\scriptsize $i$} at -1 2 
\put{$\bullet$} at -3 1.5      %
\put{$\bullet$} at -2 1.5      %
\put{$\bullet$} at -1 1.5      %
\put{$\bullet$} at  0 1.5      
\put{$\bullet$} at  1 1.5      %
\put{$\bullet$} at  2 1.5      %
\put{$\bullet$} at  3 1.5      %
\put{$\bullet$} at -3 -1.6          %
\put{$\bullet$} at -2 -1.6          %
\put{$\bullet$} at -1 -1.6          %
\put{$\bullet$} at  0 -1.6          
\put{$\bullet$} at  1 -1.6          %
\put{$\bullet$} at  2 -1.6          %
\put{$\bullet$} at  3 -1.6          %
\plot -4.5 2 -4.5 0.45 /
\plot -4.25 2 -4.25  0.45 /
\plot -4.5 0.15 -4.5 -0.87 /
\plot -4.25 0.15 -4.25  -0.87 /
\plot -4.5 -1.18   -4.5 -2 /
\plot -4.25 -1.18 -4.25 -2 /
\ellipticalarc axes ratio 1:1 360 degrees from -4.5 2 center 
at -4.375 2
\put{$*$} at -4.375 2  
\ellipticalarc axes ratio 1:1 180 degrees from -4.5 -2 center 
at -4.375 -2 
\plot -3 1.5  -3 0.5 /
\plot -2 1.5  -2 0.5 /
\plot -1 1.5  -1 0.5 /
\plot  0 1.5   0 0.5 /
\plot -3 .1  -3 -.85 /
\plot -2 .1  -2 -.85 /
\plot -1 .1  -1 0 /
\plot  0 .1   0 -1.5 /
\plot  1 0   1 -1.5 /
\plot -3 -1.15  -3 -1.5 /
\plot -2 -1.15   -2 -1.5 /
\plot  2 1.5   2 -1.5 /
\plot  3 1.5   3 -1.5 /
\setlinear
\plot -3.2 1  -4.1 1 /
\plot -2.8 1  -2.2 1 /
\plot -1.8 1  -1.2 1 /
\plot -.8 1  -.2 1 /
\plot  .2 1  .5 1 /
\plot -3.3 .3  .5 .3 /
\ellipticalarc axes ratio 2:1 180 degrees from -4.65 1  center 
at -4.65 .65 
\plot -4.65 0.3  -3.3 0.3 /
\setquadratic
\plot  .5 1  .9 1.2  1 1.5 /
\plot  .5  0.3  .9 .2 1 0 /
\setlinear
\plot -3.2 -.5  -4.1 -.5 /
\plot -2.8 -.5  -2.2 -.5 /
\plot -1.8 -.5  -1.5 -.5 /
\plot -3.3 -1  -1.5 -1 /
\ellipticalarc axes ratio 1:1 180 degrees from -4.65 -.5  center 
at -4.65 -.75 
\plot -4.65 -1  -3.3 -1 /
\setquadratic
\plot  -1.5 -.5  -1.1 -.3  -1 0 /
\plot  -1.5  -1  -1.1 -1.2 -1 -1.5 /
\endpicture
=
\beginpicture
\setcoordinatesystem units <.5cm,.5cm>         
\setplotarea x from -5.5 to 3.5, y from -2 to 2    
\put{\scriptsize $j$} at 1 2
\put{\scriptsize $i$} at -1 2 
\put{$\bullet$} at -3 1.5      %
\put{$\bullet$} at -2 1.5      %
\put{$\bullet$} at -1 1.5      %
\put{$\bullet$} at  0 1.5      
\put{$\bullet$} at  1 1.5      %
\put{$\bullet$} at  2 1.5      %
\put{$\bullet$} at  3 1.5      %
\put{$\bullet$} at -3 -1.6          %
\put{$\bullet$} at -2 -1.6          %
\put{$\bullet$} at -1 -1.6          %
\put{$\bullet$} at  0 -1.6          
\put{$\bullet$} at  1 -1.6          %
\put{$\bullet$} at  2 -1.6          %
\put{$\bullet$} at  3 -1.6          %
\plot -4.5 2 -4.5 0.62 /
\plot -4.25 2 -4.25  0.62 /
\plot -4.5 0.35 -4.5 -0.87 /
\plot -4.25 0.35 -4.25  -0.87 /
\plot -4.5 -1.18   -4.5 -2 /
\plot -4.25 -1.18 -4.25 -2 /
\ellipticalarc axes ratio 1:1 360 degrees from -4.5 2 center 
at -4.375 2
\put{$*$} at -4.375 2  
\ellipticalarc axes ratio 1:1 180 degrees from -4.5 -2 center 
at -4.375 -2 
\plot -3 1.5  -3 0.65 /
\plot -3 .3  -3 -.85 /
\plot -3 -1.15  -3 -1.5 /
\plot -2 1.5  -2 0.65 /
\plot -2 .3  -2 -.85 /
\plot -2 -1.15   -2 -1.5 /
\plot -1 0  -1 -.85 /
\plot -1 -1.15   -1 -1.5 /
\plot  0 1.5   0 -.85 /
\plot 0 -1.15   0 -1.5 /
\plot  1 1.5   1 0 /
\plot  2 1.5   2 -1.5 /
\plot  3 1.5   3 -1.5 /
\setlinear
\plot -3.2 1  -4.1 1 /
\plot -2.8 1  -2.2 1 /
\plot -1.8 1  -1.5 1 /
\plot -3.3 .5  -1.5 .5 /
\ellipticalarc axes ratio 1:1 180 degrees from -4.65 1  center 
at -4.65 .75 
\plot -4.65 0.5  -3.3 0.5 /
\setquadratic
\plot  -1.5 1  -1.1 1.2  -1 1.5 /
\plot  -1.5  0.5  -1.1 .3 -1 0 /
\setlinear
\plot -3.2 -.3  -4.1 -.3 /
\plot -2.8 -.3  -2.2 -.3 /
\plot -1.8 -.3  -1.2 -.3 /
\plot -.8 -.3  -.2 -.3 /
\plot  .2 -.3  .5 -.3 /
\plot -3.3 -1  .5 -1 /
\ellipticalarc axes ratio 2:1 180 degrees from -4.65 -.3  center 
at -4.65 -.65 
\plot -4.65 -1  -3.3 -1 /
\setquadratic
\plot  .5 -.3  .9 -.15  1 0 /
\plot  .5  -1  .9 -1.2 1 -1.5 /
\endpicture
=X^{\varepsilon_i}X^{\varepsilon_j}
$$
shows that the $X^{\varepsilon_i}$ pairwise commute. 

Let $\fg$ be a finite-dimensional complex Lie algebra with a symmetric nondegenerate
$\ad$-invariant bilinear form, and let
$U=U_h\fg$ be the Drinfel'd-Jimbo quantum group corresponding to $\fg$. The quantum group $U$ is a ribbon Hopf algebra with invertible $\cR$-matrix 
$$\cR=\sum_{\cR} R_1\otimes R_2
\qquad\hbox{in}\quad U\otimes U, \qquad \text{ and ribbon element } v = e^{-h\rho}u,$$ 
where $u = \sum_{\cR} S(R_2)R_1$ (see \cite[Corollary (2.15)]{LR}).
For $U$-modules $M$ and $N$, the map
\begin{equation}\label{preRMNdefn}
\beginpicture
\setcoordinatesystem units <1cm,.5cm>         
\setplotarea x from -8 to 2, y from 0 to 1.5    
\put{$
\begin{matrix}
\check R_{MN}\colon &M\otimes N &\longrightarrow &N\otimes M\\
&m\otimes n &\longmapsto &\displaystyle{
\sum_{\cR} R_2 n\otimes R_1 m }
\end{matrix}
$} at -5 1
\put{$M\otimes N$} at 0.5 2.4
\put{$N\otimes M$} at 0.5 0.1
\put{$\bullet$} at  0.1 1.9      
\put{$\bullet$} at  0.9 1.9      %
\put{$\bullet$} at  0.1 .6          
\put{$\bullet$} at  0.9 .6          %
\setquadratic
\plot  0.1 .6  .15 .9  .4 1.15 /
\plot  .6 1.35  .85 1.6  0.9 1.9 /
\plot 0.1 1.9  .15 1.6  .5 1.25  .85 .9  0.9 .6 /
\endpicture
\end{equation}
is a $U$-module isomorphism.  The quasitriangularity of a ribbon Hopf algebra provides the 
braid relation (see, for example, \cite[(2.12)]{OR}),
\begin{align*}
\beginpicture
\setcoordinatesystem units <1cm,.5cm>         
\setplotarea x from 0 to 2, y from -3 to 2.6    
\put{$M\otimes N\otimes P$} at 0.9 2.3
\put{$P\otimes N\otimes M$} at 0.9 -2.5
\put{$\bullet$} at  0.15 1.9      
\put{$\bullet$} at  0.9 1.9      %
\put{$\bullet$} at  1.65 1.9      %
\put{$\bullet$} at  0.15 -2          
\put{$\bullet$} at  0.9 -2          %
\put{$\bullet$} at  1.65 -2          %
\plot  1.65 1.9   1.65 .6 /
\plot  0.15 .6   0.15 -0.7 /
\plot  1.65 -0.7   1.65 -2 /
\setquadratic
\plot  0.15 .6  .2 .9  .45 1.15 /
\plot  .65 1.35  .85 1.6  0.9 1.9 /
\plot 0.15 1.9  .2 1.6  .55 1.25  .85 .9  0.9 .6 /
\plot  0.9 -0.7  0.95 -0.4  1.15 -0.15 /
\plot  1.35 .05  1.6 .3  1.65 .6 /
\plot 0.9 .6  0.95 .3  1.25 -0.05  1.6 -0.4  1.65 -0.7 /
\plot  0.15 -2  .2 -1.7  .45 -1.45 /
\plot  .65 -1.25  .85 -1  0.9 -0.7 /
\plot 0.15 -0.7  .2 -1  .55 -1.35  .85 -1.7  0.9 -2 /
\endpicture
~~&=~~
\beginpicture
\setcoordinatesystem units <1cm,.5cm>         
\setplotarea x from 0 to 2, y from -3 to 2.6    
\put{$M\otimes N\otimes P$} at 0.9 2.3
\put{$P\otimes N\otimes M$} at 0.9 -2.5
\put{$\bullet$} at  0.15 1.9      
\put{$\bullet$} at  0.9 1.9      %
\put{$\bullet$} at  1.65 1.9      %
\put{$\bullet$} at  0.15 -2          
\put{$\bullet$} at  0.9 -2          %
\put{$\bullet$} at  1.65 -2          %
\plot  0.15 1.9   0.15 .6 /
\plot  1.65 .6   1.65 -0.7 /
\plot  0.15 -0.7   0.15 -2 /
\setquadratic
\plot  0.9 0.6  0.95 0.9  1.15 1.15 /
\plot  1.35 1.35  1.6 1.6  1.65 1.9 /
\plot 0.9 1.9  0.95 1.6  1.25 1.25  1.6 0.9  1.65 0.6 /
\plot  0.15 -0.7  .2 -0.4  .45 -0.15 /
\plot  .65 .05  .85 .3  0.9 .6 /
\plot 0.15 .6  .2 .3  .55 -0.05  .85 -0.4  0.9 -0.7 /
\plot  0.9 -2  0.95 -1.7  1.15 -1.45 /
\plot  1.35 -1.25  1.6 -1  1.65 -0.7 /
\plot 0.9 -0.7  0.95 -1  1.25 -1.35  1.6 -1.7  1.65 -2 /
\endpicture
\\
(\check R_{MN}\otimes \id_P)
(\id_N\otimes \check R_{MP})
(\check R_{NP}\otimes \id_M)
&=
(\id_M\otimes \check R_{NP})
(\check R_{MP}\otimes \id_N)
(\id_P\otimes \check R_{MN}).
\end{align*}

\begin{thm} \label{affaction} Let $\fg$ be a finite-dimensional complex Lie algebra  with a symmetric nondegenerate $\ad$-invariant bilinear form, 
let $U=U_h\fg$ be the corresponding Drinfeld-Jimbo quantum group and let
$C=Z(U)$ be the center of $U_h\fg$.  Let $M$ and $V$ be $U$-modules.
Then $M\otimes V^{\otimes k}$ is a $C B_k$-module with
action given by
\begin{equation}\label{BraidRep}
\begin{matrix}
\Phi\colon &C B_k &\longrightarrow
&\End_{U}(M\otimes V^{\otimes k}) \\
&T_i &\longmapsto &\check R_i, &\qquad &1\le i\le k-1,\\
&X^{\varepsilon_1} &\longmapsto &\check R_0^2, \\
&z &\longmapsto &z_M,
\end{matrix}
\end{equation}
where $z_M = z\otimes \id_V^{\otimes k}$, 
$$\check R_i = \id_M\otimes \id_V^{\otimes (i-1)}
\otimes \check R_{VV}\otimes \id_V^{\otimes (k-i-1)}
\qquad\hbox{and}\qquad
\check R_0^2 = (\check R_{MV}\check R_{VM})\otimes \id_V^{\otimes (k-1)},
$$
with $\check R_{MV}$ as in \eqref{preRMNdefn}.
The $CB_k$ action commutes with the $U$-action on $M\otimes V^{\otimes k}$.
\end{thm}
\begin{proof}
The relations \eqref{affrel1} and \eqref{affrel4} are consequences of the definition of the 
action of $T_i$ and $X^{\varepsilon_1}$. The relations \eqref{affrel2} and \eqref{affrel3} 
follow from the following computations:
$$
\check R_i\check R_{i+1}\check R_i 
=
\beginpicture
\setcoordinatesystem units <.3cm,.3cm>         
\setplotarea x from 0 to 2, y from -2 to 2    
\plot  1.65  2.6   1.65  1.3 /
\plot  0.15  1.3     0.15  -0.1 /
\plot  1.65  -0.1   1.65  -1.4 /
\setsolid
\setquadratic
\plot  0.15 -1.4  .2 -1.1  .45 -0.85 /
\plot  .65 -0.7  .85 -0.3  0.9 0 /
\plot 0.15 0  .2 -0.3  .55 -.8  .85 -1.1  0.9 -1.4 /
\plot  0.9 0  0.95 0.3  1.15 0.55 /
\plot  1.35 0.75  1.6 1  1.65 1.3 /
\plot 0.9 1.3  0.95 1  1.25 .65  1.6 .3  1.65 0 /
\plot  0.15 1.3  .2 1.6  .45 1.85 /
\plot  .65 2.05  .85 2.3  0.9 2.6 /
\plot 0.15 2.6  .2 2.3  .55 1.95  .85 1.6  0.9 1.3 /
\endpicture
=
\beginpicture
\setcoordinatesystem units <.3cm,.3cm>         
\setplotarea x from 0 to 2, y from -2 to 2    
\plot  0.15  2.6   0.15  1.3 /
\plot  1.65  1.3   1.65  0 /
\plot  0.15  0   0.15  -1.4 /
\setsolid
\setquadratic
\plot  0.9 1.3  0.95 1.6  1.15  1.85 /
\plot  1.35 2.05  1.6 2.3  1.65 2.6 /
\plot 0.9 2.6  0.95 2.3  1.25 1.95  1.6 1.6  1.65 1.3 /
\plot  0.15 0  .2 .3  .45  0.55 /
\plot  .65 .75  .85 1  0.9 1.3 /
\plot 0.15 1.3  .2 1  .55 .65  .85 .3  0.9 0 /
\plot  0.9 -1.4  0.95 -1.1  1.15 -0.85 /
\plot  1.35 -0.7  1.6 -0.3  1.65 0 /
\plot 0.9 -0  0.95 -0.3  1.25 -.8  1.6 -1.1  1.65 -1.4 /
\endpicture
= \check R_{i+1}\check R_i \check R_{i+1}$$
and
$$
\check R_0^2\check R_1\check R_0^2\check R_1 
=
\beginpicture
\setcoordinatesystem units <.3cm,.3cm>         
\setplotarea x from 0 to 2, y from -2 to 2    
\put{$\cdot$} at  0.15 3.9      
\put{$\cdot$} at  0.9 3.9      %
\put{$\cdot$} at  1.65 3.9      %
\put{$\cdot$} at  0.15  -4          
\put{$\cdot$} at  0.9  -4          %
\put{$\cdot$} at  1.65  -4          %
\plot  0.15  -2.7   0.15  -4 /
\plot  1.65  3.9   1.65  2.6 /
\plot  1.65  2.6   1.65  1.3 /
\plot  0.15  1.3     0.15  -0.1 /
\plot  1.65  -0.1   1.65  -1.4 /
\plot  1.65  -1.4   1.65  -2.7 /
\setsolid
\setquadratic
\plot  0.15 -1.4  .2 -1.1  .45 -0.85 /
\plot  .65 -0.7  .85 -0.3  0.9 0 /
\plot 0.15 0  .2 -0.3  .55 -.8  .85 -1.1  0.9 -1.4 /
\plot  0.9 0  0.95 0.3  1.15 0.55 /
\plot  1.35 0.75  1.6 1  1.65 1.3 /
\plot 0.9 1.3  0.95 1  1.25 .65  1.6 .3  1.65 0 /
\plot  0.15 -2.7  .2 -2.4  .45  -2.15 /
\plot  .65 -2.05  .85 -1.7  0.9 -1.3 /
\plot 0.15 -1.4  .2 -1.7  .55 -2.05  .85 -2.4  0.9 -2.7 /
\plot  0.9 -4  0.95 -3.7  1.15 -3.45 /
\plot  1.35 -3.25  1.6 -3  1.65 -2.7 /
\plot 0.9 -2.7  0.95 -3  1.25 -3.35  1.6 -3.7  1.65 -4 /
\plot  0.15 2.6  .2 2.9  .45 3.15 /
\plot  .65 3.35  .85 3.6  0.9 3.9 /
\plot 0.15 3.9  .2 3.6  .55 3.25  .85 2.9  0.9 2.6 /
\plot  0.15 1.3  .2 1.6  .45 1.85 /
\plot  .65 2.05  .85 2.3  0.9 2.6 /
\plot 0.15 2.6  .2 2.3  .55 1.95  .85 1.6  0.9 1.3 /
\endpicture
=
\beginpicture
\setcoordinatesystem units <.3cm,.3cm>         
\setplotarea x from 0 to 2, y from -2 to 2    
\put{$\cdot$} at  0.15 3.9      
\put{$\cdot$} at  0.9 3.9      %
\put{$\cdot$} at  1.65 3.9      %
\put{$\cdot$} at  0.15  -4          
\put{$\cdot$} at  0.9  -4          %
\put{$\cdot$} at  1.65  -4          %
\plot  1.65  3.9   1.65  2.6 /
\plot  0.15  2.6   0.15  1.3 /
\plot  1.65  1.3   1.65  0 /
\plot  0.15  0   0.15  -1.4 /
\plot  1.65  -1.4   1.65  -2.7 /
\plot  0.15  -2.7   0.15  -4 /
\setdashes
\plot  0  2.6   2  2.6 /
\plot  0  -1.4   2  -1.4 /
\setsolid
\setquadratic
\plot  0.15 2.6  .2 2.9  .45 3.15 /
\plot  .65 3.35  .85 3.6  0.9 3.9 /
\plot 0.15 3.9  .2 3.6  .55 3.25  .85 2.9  0.9 2.6 /
\plot  0.9 1.3  0.95 1.6  1.15  1.85 /
\plot  1.35 2.05  1.6 2.3  1.65 2.6 /
\plot 0.9 2.6  0.95 2.3  1.25 1.95  1.6 1.6  1.65 1.3 /
\plot  0.15 0  .2 .3  .45  0.55 /
\plot  .65 .75  .85 1  0.9 1.3 /
\plot 0.15 1.3  .2 1  .55 .65  .85 .3  0.9 0 /
\plot  0.9 -1.4  0.95 -1.1  1.15 -0.85 /
\plot  1.35 -0.7  1.6 -0.3  1.65 0 /
\plot 0.9 -0  0.95 -0.3  1.25 -.8  1.6 -1.1  1.65 -1.4 /
\plot  0.15 -2.7  .2 -2.4  .45 -2.15 /
\plot  .65 -1.95  .85 -1.7  0.9 -1.4 /
\plot 0.15 -1.4  .2 -1.7  .55 -2.05  .85 -2.4  0.9 -2.7 /
\plot  0.9 -4  0.95 -3.7  1.15 -3.45 /
\plot  1.35 -3.25  1.6 -3  1.65 -2.7 /
\plot 0.9 -2.7  0.95 -3  1.25 -3.35  1.6 -3.7  1.65 -4 /
\endpicture
=
\beginpicture
\setcoordinatesystem units <.3cm,.3cm>         
\setplotarea x from 0 to 2, y from -2 to 2    
\put{$\cdot$} at  0.15 3.9      
\put{$\cdot$} at  0.9 3.9      %
\put{$\cdot$} at  1.65 3.9      %
\put{$\cdot$} at  0.15  -4          
\put{$\cdot$} at  0.9  -4          %
\put{$\cdot$} at  1.65  -4          %
\plot  0.15  3.9   0.15  2.6 /
\plot  1.65  2.6   1.65  1.3 /
\plot  0.15  1.3   0.15  0 /
\plot  1.65  0   1.65  -1.4 /
\plot  0.15  -1.4   0.15  -2.7 /
\plot  1.65  -2.7   1.65  -4 /
\setdashes
\plot  0  0   2  0 /
\setsolid
\setquadratic
\plot  0.15 -1.4  .2 -1.1  .45 -0.85 /
\plot  .65 -0.7  .85 -0.3  0.9 0 /
\plot 0.15 0  .2 -0.3  .55 -.8  .85 -1.1  0.9 -1.4 /
\plot  0.9 -2.7  0.95 -2.4  1.15  -2.15 /
\plot  1.35 -1.95  1.6 -1.7  1.65 -1.4 /
\plot 0.9 -1.4  0.95 -1.7  1.25 -2.05  1.6 -2.4  1.65 -2.7 /
\plot  0.15 -4  .2 -3.7  .45  -3.45 /
\plot  .65 -3.25  .85 -3  0.9 -2.7 /
\plot 0.15 -2.7  .2 -3  .55 -3.35  .85 -3.7  0.9 -4 /
\plot  0.9 2.6  0.95 2.9  1.15 3.15 /
\plot  1.35 3.3  1.6 3.7  1.65 4 /
\plot 0.9 4  0.95 3.7  1.25 3.2  1.6 2.9  1.65 2.6 /
\plot  0.15 1.3  .2 1.6  .45 1.85 /
\plot  .65 2.05  .85 2.3  0.9 2.6 /
\plot 0.15 2.6  .2 2.3  .55 1.95  .85 1.6  0.9 1.3 /
\plot  0.9 0  0.95 0.3  1.15 0.55 /
\plot  1.35 0.75  1.6 1  1.65 1.3 /
\plot 0.9 1.3  0.95 1  1.25 .65  1.6 .3  1.65 0 /
\endpicture
=
\beginpicture
\setcoordinatesystem units <.3cm,.3cm>         
\setplotarea x from 0 to 2, y from -2 to 2    
\put{$\cdot$} at  0.15 3.9      
\put{$\cdot$} at  0.9 3.9      %
\put{$\cdot$} at  1.65 3.9      %
\put{$\cdot$} at  0.15  -4          
\put{$\cdot$} at  0.9  -4          %
\put{$\cdot$} at  1.65  -4          %
\plot  0.15  3.9   0.15  2.6 /
\plot  1.65  2.6   1.65  1.3 /
\plot  1.65  1.3   1.65  0 /
\plot  0.15  0     0.15  -1.4 /
\plot  1.65  -1.4   1.65  -2.7 /
\plot  1.65  -2.7   1.65  -4 /
\setdashes
\plot  0  1.3   2  1.3 /
\plot  0  -2.7   2  -2.7 /
\setsolid
\setquadratic
\plot  0.15 -2.7  .2 -2.4  .45 -2.15 /
\plot  .65 -1.95  .85 -1.7  0.9 -1.4 /
\plot 0.15 -1.4  .2 -1.7  .55 -2.05  .85 -2.4  0.9 -2.7 /
\plot  0.9 -1.4  0.95 -1.1  1.15  -0.85 /
\plot  1.35 -0.7  1.6 -0.3  1.65 0 /
\plot 0.9 0  0.95 -0.3  1.25 -0.8  1.6 -1.1  1.65 -1.4 /
\plot  0.15 -4  .2 -3.7  .45  -3.45 /
\plot  .65 -3.25  .85 -3  0.9 -2.7 /
\plot 0.15 -2.7  .2 -3  .55 -3.35  .85 -3.7  0.9 -4 /
\plot  0.9 2.6  0.95 2.9  1.15 3.15 /
\plot  1.35 3.3  1.6 3.7  1.65 4 /
\plot 0.9 4  0.95 3.7  1.25 3.2  1.6 2.9  1.65 2.6 /
\plot  0.15 1.3  .2 1.6  .45 1.85 /
\plot  .65 2.05  .85 2.3  0.9 2.6 /
\plot 0.15 2.6  .2 2.3  .55 1.95  .85 1.6  0.9 1.3 /
\plot  0.15 0  .2 0.3  .45 0.55 /
\plot  .65 .75  .85 1  0.9 1.3 /
\plot 0.15 1.3  .2 1  .55 .65  .85 .3  0.9 0 /
\endpicture
= \check R_1\check R_0^2\check R_1\check R_0^2. 
$$
\end{proof}

Let $v=e^{-h\rho}u$ be the ribbon element in $U=U_h\fg$.  For a $U_h\fg$-module $M$
define
\begin{equation}\label{qcasimir}
\begin{matrix}
C_M\colon &M &\longrightarrow &M \\
&m &\longmapsto &vm 
\end{matrix}
\qquad\hbox{so that}\qquad
C_{M\otimes N} =
(\check R_{MN}\check R_{NM})^{-1}
(C_M\otimes C_N)
\end{equation}
(see \cite[Prop. 3.2]{Dr}).
If $M$ is a $U_h\fg$-module generated by a highest weight vector 
$v^+_\lambda$ of weight $\lambda$, then
\begin{equation}\label{qCasvalue} 
C_M = q^{-\langle \lambda,\lambda+2\rho\rangle} \id_M,
\qquad\hbox{where}\quad q = e^{h/2}
\end{equation}
(see \cite[Prop. 2.14]{LR} or \cite[Prop. 5.1]{Dr}).
From \eqref{qCasvalue} and the relation \eqref{qcasimir} it follows that if $M=L(\mu)$ and 
$N=L(\nu)$ are finite-dimensional irreducible $U_h\fg$-modules of highest weights
$\mu$ and $\nu$ respectively, 
then $\check R_{MN}\check R_{NM}$ acts on the
$L(\lambda)$-isotypic component 
$L(\lambda)^{\oplus c_{\mu\nu}^\lambda}$
of the decomposition
\begin{equation}\label{fulltwist} 
L(\mu)\otimes L(\nu) = \bigoplus_\lambda L(\lambda)^{\oplus c_{\mu\nu}^\lambda}
\qquad\hbox{by the scalar}\qquad
q^{\langle\lambda,\lambda+2\rho\rangle 
-\langle\mu,\mu+2\rho\rangle 
-\langle\nu,\nu+2\rho\rangle}.
\end{equation}
By the definition of $X^{\varepsilon_i}$ in \eqref{BraidMurphy},
\begin{equation}\label{XiasRmatrix}
\Phi(X^{\varepsilon_i}) = 
 \check R_{M\otimes V^{\otimes(i-1)},V}\check R_{V, M\otimes V^{\otimes {(i-1)}}}
= 
\beginpicture
\setcoordinatesystem units <.5cm,1cm>         
\setplotarea x from -5.5 to 4.5, y from -1.25 to 1.25    
\put{${}^i$} at 1 1.2 
\put{$\bullet$} at -3 0.75      %
\put{$\bullet$} at -2 0.75      %
\put{$\bullet$} at -1 0.75      %
\put{$\bullet$} at  0 0.75      
\put{$\bullet$} at  1 0.75      %
\put{$\bullet$} at  2 0.75      %
\put{$\bullet$} at  3 0.75      %
\put{$\bullet$} at -3 -0.75          %
\put{$\bullet$} at -2 -0.75          %
\put{$\bullet$} at -1 -0.75          %
\put{$\bullet$} at  0 -0.75          
\put{$\bullet$} at  1 -0.75          %
\put{$\bullet$} at  2 -0.75          %
\put{$\bullet$} at  3 -0.75          %
\setquadratic
\plot -4.5 1.25  -4.4 .9 -4 .6  / 
\plot  -4 .6 -3.4 .25 -3.25 0 /
\plot -4.25 1.25  -4.15 .9 -3.75 .6  / 
\plot  -3.75 .6 -3.15 .25 -3.0 0 /
\plot -4.5 -1.25  -4.4 -.9 -4 -.6  / 
\plot -4.25 -1.25  -4.15 -.9 -3.75 -.6  / 
\plot -3.25 0 -3.30 -.09  -3.40 -.17 /
\plot -3.00 0 -3.05 -.09 -3.12 -.17 /
\plot  -4.00   -.6 -3.70 -.5 -3.40 -.3 /
\plot  -3.75   -.6 -3.45  -.5 -3.15 -.3 /

\ellipticalarc axes ratio 1:1 360 degrees from -4.5 1.25 center 
at -4.375 1.25
\put{$*$} at -4.375 1.25  
\ellipticalarc axes ratio 1:1 180 degrees from -4.5 -1.25 center 
at -4.375 -1.25 
\setquadratic
\plot -3 0.75  -2.9  .55  -2.5 .375 /
\plot -2.5 .375  -2.1  .25  -2 0 /
\plot -2    0.75     -1.9  .55  -1.5 .375 /
\plot -1.5 0.375  -1.1  .25  -1 0 /
\plot -1    0.75     -0.9  .55  -0.5 .375 /
\plot -0.5 0.375  -0.1  .25  0  0 /
\plot 0    0.75     .1  .55  .5 .375 /
\plot .5  0.375   .9  .25  1  0 /
\plot -3 -0.75  -2.9  -.55  -2.5 -.380 /
\plot -2.3  -.26  -2.1  -.20  -2 0 /
\plot -2    -0.75     -1.9  -.55  -1.7 -.45 /
\plot -1.4 -0.28  -1.1  -.20  -1 0 /
\plot -1    -0.75   -0.9  -.6  -0.7 -.48 /
\plot -0.4 -0.32  -0.1  -.2  0  0 /
\plot 0    -0.75     .1  -.65  .2 -.55 /
\plot .55  -0.38   .9  -.25  1  0 /
\setlinear
\plot 2 .75  2 -.75 /
\plot 3 .75 3 -.75 /
\setquadratic
\plot  1 .75 .9 .6  .4 .47 /
\plot -3.5 .22 -3.9  .15  -4 0 /
\plot -1 -.375  -3.25  -.25  -4 0 /
\plot 1 -.75  .5  -.5 -1 -.375 / 
\setlinear
\plot .1 .45  -.2 .4 /
\plot -.8 .38  -1.1 .37 /
\plot -1.7 .36 -2.1 .35 /
\plot -2.6 .33  -3 .29 /
\setdashes
\plot -5 0 4 0 /
\endpicture,
\end{equation}
so that, by \eqref{qcasimir},
the eigenvalues of $\Phi(X^{\varepsilon_i})$ are functions of the eigenvalues of the Casimir.

\section{Actions of classical type tantalizers} 
\label{sec:classical-actions}

In this section, we define the affine Birman-Murakami-Wenzl (BMW) algebra $W_k$ and its degenerate version $\cW_k$, exactly following our treatment in \cite{DRV}.
Just as the affine BMW algebras $W_k$ and the affine Hecke algebras $H_k$ are quotients
of the group algebra of affine braid group $CB_k$, 
the degenerate affine BMW algebras $\cW_k$ and 
the degenerate affine Hecke algebras $\cH_k$  are quotients of $\cB_k$. 
Moreover, the tensor space actions defined in Theorems \ref{degaction} and Theorem \ref{affaction}
factor through these quotients in important cases.  The affine and degenerate affine BMW algebras arise 
when $\fg$ is $\fso_n$ or $\fsp_n$
and $V$ is the first fundamental representation; similarly, the 
affine and degenerate affine Hecke algebras arise when $\fg$ is $\fgl_n$ or $\fsl_n$ and 
$V$ is the first fundamental representation. In the case when $M$ is the trivial representation and
$\fg$ is $\fso_n$, the Jucys-Murphy elements $y_1, \dots, y_k$ in $\cB_k$ become the
Jucys-Murphy elements for the Brauer algebras used in \cite{Naz}; in the case that 
$\fg = \fsl_n$, these become the classical Jucys-Murphy elements in the group algebra of the symmetric group.

In defining the affine and degenerate affine BMW algebras, we must make a choice of infinite families of parameters,  $Z_0^{(\ell)}$ and $z_0^{(\ell)}$, respectively. In order to avoid choices which yield the zero algebra, we choose parameters in the ground ring $C = Z(U)$ which arise naturally in each of the action theorems below. As we will see in the proofs of Theorem \ref{degBMWaction} and Theorem \ref{BMWaction} (specifically, the calculations in \eqref{determining_z0} and \eqref{determining_Z0}),  the natural actions of $\cB_k$ and $CB_k$ on tensor space in Theorems \ref{degaction} and Theorem \ref{affaction} force the parameters to be 
$$z_0^{(\ell)} = \epsilon (\id\otimes \tr_V)((\half y + \gamma)^\ell)
\qquad\hbox{and}\qquad
Z_0^{(\ell)} = \epsilon (\id\otimes \qtr_V)((z\cR_{21}\cR)^\ell).$$

\paragraph{Preliminaries on classical type combinatorics.}

Let $V = \CC^r$. The Lie algebras $\fg=\fgl_r$ and $\fsl_r$ are given by
$$\fgl_r=\End(V)
\qquad\hbox{and}\qquad \fsl_r=\{ x\in \fgl_r\ |\ \tr(x)=0\},$$
with bracket $[x,y]=xy-yx$.   Then 
$$\fgl_r\quad\hbox{has basis}\quad \{ E_{ij}\ |\ 1\le i,j\le r\},$$
where $E_{ij}$ is the matrix with 1 in the $(i.j)$ entry and 0 elsewhere.  A Cartan subalgebra
of $\fgl_r$ is 
$$\fh_{\fgl} = \{ x\in \fgl_r\ |\ \hbox{$x$ is diagonal}\}
\qquad\hbox{with basis}\qquad \{ E_{11}, E_{22}, \ldots, E_{rr}\},$$
and the dual basis $\{\varepsilon_1,\ldots, \varepsilon_r\}$ of $\fh_{\fgl}^*$ is specified by
$$\varepsilon_i\colon \fh_\fgl\to \CC
\qquad\hbox{given by}\qquad
\varepsilon_i(E_{jj}) = \delta_{ij}.$$
The form 
\begin{equation}\label{glform}
\langle , \rangle\colon \fg\otimes \fg \to \CC
\qquad\hbox{given by}\qquad
\langle x,y\rangle = \tr_V(xy)
\end{equation}
is a nondegenerate $\ad$-invariant symmetric bilinear form on $\fg$ such that
the restriction to $\fh_\fgl$ is a nondegenerate form 
$\langle , \rangle\colon \fh_\fgl\otimes \fh_\fgl \to \CC$. Since $\langle , \rangle$ is nondegenerate,
the map $\nu\colon \fh_{\fgl}\to \fh_{\fgl}^*$ given by $\nu(h)= \langle h, \cdot\rangle$ is a 
vector space isomorphism which induces a nondegenerate form $\langle , \rangle$ on $\fh_{\fgl}^*$.  Further,
$$\{ E_{11}, \ldots, E_{rr}\} \quad \text{ and } \quad \{ \varepsilon_1,\ldots,\varepsilon_r\} \quad \text{ are orthonormal bases of $\fh_\fgl$ and $\fh^*_\fgl$,}$$
respectively.
%
A Cartan subalgebra of $\fsl_r$ is 
$$\fh_{\fsl} = (E_{11}+\cdots+E_{rr})^\perp = \{ x\in \fh_\fgl\ |\ \langle x, E_{11}+\cdots+E_{rr}\rangle=0\},$$
the orthogonal subspace to $\CC(E_{11}+\cdots+E_{rr})$.
The \emph{dominant integral weights for $\fgl_r$},
$$P^+ = \{ \lambda_1\varepsilon_1+\cdots+\lambda_r\varepsilon_r\ |\ 
\lambda_i\in \ZZ, \lambda_1\ge \cdots \ge \lambda_r\},$$
index the irreducible finite-dimensional representations $L(\lambda)$ of $\fgl_r$. The irreducible
finite-dimensional representations of $\fsl_r$ are
$$L(\bar\lambda) = \Res^{\fgl_r}_{\fsl_r}(L(\lambda)),
$$
where $\bar\lambda$ is the orthogonal projection of $\lambda$ to
$\fh_\fsl^* =(\varepsilon_1+\cdots+\varepsilon_r)^\perp$.

The matrix units $\{E_{ij}\ |\ 1\le i,j\le r\}$ form a basis of $\fgl_r$ for which the dual basis
with respect to the form in \eqref{glform} is $\{E_{ji}\ |\ 1\le i,j\le r\}$. So 
\begin{align}
\gamma^\fgl &= \sum_{1\le i,j\le r} E_{ij}\otimes E_{ji}
= \sum_{1\le i,j\le r\atop i\ne j} E_{ij}\otimes E_{ji} + \sum_{i=1} E_{ii}\otimes E_{ii}, 
\qquad\hbox{and} \label{glgamma} \\
\gamma^{\fsl} &= \gamma^{\fgl} - E_+\otimes E_+,
\qquad\hbox{where}\quad E_+=E_{11}+\cdots + E_{rr}. \label{slgamma}
\end{align}
If the Casimir for $\fgl_r$,
$$\kappa^{\fgl} = \sum_{1\le i,j\le r} E_{ij}E_{ji},
\quad\hbox{acts on}\quad L(\lambda)\quad\hbox{by the constant}\quad \kappa^{\fgl}{(\lambda)}
$$
then the Casimir for $\fsl_r$,
\begin{equation}\label{casconv}
\kappa^{\fsl} = \kappa^{\fgl} - E_+E_+,
\qquad\hbox{acts on}\quad
L(\bar\lambda)\quad\hbox{by the constant}\quad
\kappa^{\fgl}{(\lambda)}-\frac{1}{r}|\lambda|^2,
\end{equation}
where $|\lambda|=\lambda_1+\cdots+\lambda_r$.

Let $V = \CC^N$. The Lie algebras $\fg=\fso_{N}$ or $\fsp_{N}$ are given by
$$\fg = \{ x\in \fgl(V)\ |\ ( xv_1,v_2) + ( v_1, xv_2) = 0\ \ 
\hbox{for all $v_1,v_2\in V$} \},
$$
where $(,)\colon V\otimes V\to \CC$ is a nondegenerate bilinear form satisfying
\begin{equation}\label{cltype1}
( v_1, v_2) = \epsilon( v_2, v_1),
\qquad\hbox{where}\quad
\epsilon = \begin{cases}
1, &\hbox{if $\fg = \fso_{2r+1}$,} \\
-1, &\hbox{if $\fg = \fsp_{2r}$,} \\
1, &\hbox{if $\fg = \fso_{2r}$.}
\end{cases}
\end{equation}
Choose 
\begin{equation}\label{Vhat}\hbox{a basis $\{v_i\ |\ i\in \hat V\}$ of $V$},
\qquad\hbox{where}\qquad
\hat{V} = \begin{cases}
\{-r,\ldots, -1, 0, 1, \ldots, r\}, &\hbox{if $\fg = \fso_{2r+1}$,} \\
\{-r,\ldots, -1, 1, \ldots, r\}, &\hbox{if $\fg = \fsp_{2r}$,} \\
\{-r,\ldots, -1, 1, \ldots, r\}, &\hbox{if $\fg = \fso_{2r}$.} 
\end{cases}
\end{equation}
so that the matrix of the bilinear form $(,)\colon V\otimes V\to \CC$ is
$$J = \left(\begin{minipage}{1.7cm}\begin{tikzpicture}[scale=.25]
\node at (6.1,6.1) {\small $1$};
\node at (5.4,5.4) {$\cdot$};
\node at (5,5) {$\cdot$};
\node at (4.6,4.6) {$\cdot$};
\node at (4,4) {\small$1$};
\node at (3,3) {\small $\epsilon$};
\node at (2.4,2.4) {$\cdot$};
\node at (2,2) {$\cdot$};
\node at (1.6,1.6) {$\cdot$};
\node at (1,1) {\small$\epsilon$};
\node at (1.8,5.2) {$0$};
\node at (5.2,1.8) {$0$};
\end{tikzpicture}\end{minipage}\right),
\qquad\hbox{and}\qquad
\fg = \{ x\in \fgl_N\ |\ x^tJ +Jx=0\},
$$
where $N = \dim(V)$ and $x^t$ is the transpose of $x$.
Then, as in Molev \cite[(7.9)]{Mo}  and [Bou, Ch.\ 8 \S 13 2.I, 3.I, 4.I],
\begin{equation}\label{natspanningset}
\fg = \hbox{span}\{ F_{ij}\ |\ i,j\in \hat V\}
\quad\hbox{where}\quad F_{ij} = E_{ij} - \theta_{ij}E_{-j,-i},
\end{equation}
where $E_{ij}$ is the matrix with 1 in the $(i,j)$-entry and $0$ elsewhere and
$$
\theta_{ij} = \begin{cases}
1, &\hbox{if $\fg = \fso_{2r+1}$,} \\
\mathrm{sgn}(i) \cdot \mathrm{sgn}(j), &\hbox{if $\fg=\fsp_{2r}$,} \\
1, &\hbox{if $\fg=\fso_{2r}$.} 
\end{cases}
$$

A Cartan subalgebra of $\fg$ is
\begin{equation}\label{cartan}
\fh = \hbox{span}\{ F_{ii}\ |\ i\in \hat V\}
\qquad\hbox{with basis}\qquad
\{ F_{11}, F_{22}, \ldots, F_{rr} \}.
\end{equation}
The dual basis $\{\varepsilon_1,\ldots , \varepsilon_r\}$ of $\fh^*$ is specified by
\begin{equation}
\varepsilon_i\colon \fh\to \CC
\qquad\hbox{given by}\qquad
\varepsilon_i(F_{jj}) = \delta_{ij}.
\label{epsdefn}
\end{equation}
The form 
\begin{equation}\label{favform}
\langle , \rangle\colon \fg\otimes \fg \to \CC
\qquad\hbox{given by}\qquad
\langle x,y\rangle = \half \tr_V(xy)
\end{equation}
is a nondegenerate $\ad$-invariant symmetric bilinear form on $\fg$ such that
the restriction to $\fh$ is a nondegenerate form 
$\langle , \rangle\colon \fh\otimes \fh \to \CC$  on $\fh$. Since $\langle , \rangle$ is nondegenerate,
the map $\nu\colon \fh\to \fh^*$ given by $\nu(h)= \langle h, \cdot\rangle$ is a 
vector space isomorphism which induces a nondegenerate form $\langle , \rangle$ on $\fh^*$. Let $\langle , \rangle\colon \fh^*\otimes \fh^* \to \CC$ be the form on $\fh^*$ induced by the form on $\fh$ and the vector space isomorphism
$\nu\colon \fh \to \fh^*$ given by $\nu(h)= \langle h, \cdot\rangle$. Further,
$$
\{ F_{11}, \ldots, F_{rr}\} \quad\hbox{and}\quad \{ \varepsilon_1,\ldots,\varepsilon_r\}\ \quad \hbox{are orthonormal bases  of $\fh$ and $\fh^*$.}
$$

With $F_{ij}$ as in \eqref{natspanningset}, 
$\fg$ has basis 
$$\{ F_{i, i} ~|~ 0 < i \in \hat V \} \cup \{ F_{\pm i, \pm j} ~|~ 0 < i < j \in \hat V \} \cup \{ F_{0, \pm i} ~|~ 0 < i\in \hat V \}  \qquad \text{if } \fg = \fso_{2r+1},$$
$$\{ F_{i, i},  F_{-i, i}, F_{i, -i} ~|~ 0 < i \in \hat V \} \cup \{ F_{\pm i, \pm j} ~|~ 0 < i < j \in \hat V \} \qquad \text{if } \fg = \fsp_{2r},$$
$$\{ F_{i, i} ~|~ 0 < i \in \hat V \} \cup \{ F_{\pm i, \pm j} ~|~ 0 < i < j \in \hat V \}  \qquad \text{if } \fg = \fso_{2r}.$$
With respect to the nondegenerate $\ad$-invariant symmetric bilinear form 
$\langle , \rangle\colon \fg \otimes \fg \to \CC$ given in \eqref{favform}, 
$\langle x,y\rangle = \half\tr_V(xy)$,
the dual basis with respect to $\langle, \rangle$ is
$$F_{ij}^* =  F_{ji} \quad \text{ if } i \neq - j, \quad \text{ and }  \quad F_{i,-i}^* = \half F_{-i,i}.$$
The sets 
$$\{ F_{-i, -i} ~|~ 0 < i \in \hat V \} \cup \{ F_{\pm i, \pm j} ~|~ 0 < j < i \in \hat V \} \cup \{ F_{ \pm i, 0} ~|~ 0 < i\in \hat V \}  \qquad \text{if } \fg = \fso_{2r+1},$$
$$\{ F_{-i, -i},  F_{-i, i}, F_{i, -i} ~|~ 0 < i \in \hat V \} \cup \{ F_{\pm i, \pm j} ~|~ 0 < j<i \in \hat V \} \qquad \text{if } \fg = \fsp_{2r},$$
$$\{ F_{-i, -i} ~|~ 0 < i \in \hat V \} \cup \{ F_{\pm i, \pm j} ~|~ 0 < j < i \in \hat V \}  \qquad \text{if } \fg = \fso_{2r},$$
are alternate bases, and $F_{i, -i} = 0$ when $\fg=\fso_{2r+1}$ or $\fg=\fso_{2r}$.  So
\begin{equation}\label{sospgamma}
2 \gamma = \sum_{i,j \in \hat V} F_{ij}\otimes  F^*_{ji} + \sum_{i \in \hat V} F_{i,-i}\otimes  F^*_{i, -i} 
= \sum_{i,j \in \hat V} F_{ij}\otimes F_{ji}.
\end{equation}

%
To compute
the value $\half \langle\lambda, \lambda+2\rho\rangle$ in \eqref{gamma_value}
choose positive roots
\begin{equation}\label{posroots}
R^+ = \begin{cases}
\{ \varepsilon_i\pm \varepsilon_j\ |\ 1\le i<j\le r\} 
\cup \{\varepsilon_i\ |\ 1\le i\le r\}, 
&\hbox{for $\fg=\fso_{2r+1}$,} \\
\{ \varepsilon_i\pm \varepsilon_j\ |\ 1\le i<j\le r\}  
\cup \{2\varepsilon_i\ |\ 1\le i\le r\}, 
&\hbox{for $\fg=\fsp_{2r}$,} \\
\{ \varepsilon_i\pm \varepsilon_j\ |\ 1\le i<j\le r\},
&\hbox{for $\fg=\fso_{2r}$,}
\end{cases}
\end{equation}
Since
\begin{align*}
\sum_{1\le i< j \le r} (\varepsilon_i - \varepsilon_j)
&+ \sum_{1\le i<j\le r} (\varepsilon_i+\varepsilon_j)
+\sum_{i=1}^r \varepsilon_i
+\sum_{i=1}^r \varepsilon_i \\
&= \sum_{i=1}^r (r-2i+1)\varepsilon_i
+ \sum_{i=1}^r (r-1)\varepsilon_i
+\sum_{i=1}^r \varepsilon_i
+\sum_{i=1}^r \varepsilon_i,
\end{align*}
it follows that
\begin{equation}\label{explrho}
2\rho = \sum_{i=1}^r
(y-2i+1)\varepsilon_i,
\qquad\hbox{where}\quad
y = \langle \varepsilon_1, \varepsilon_1+2\rho\rangle = \begin{cases}
2r, &\hbox{if $\fg = \fso_{2r+1}$,} \\
2r+1, &\hbox{if $\fg = \fsp_{2r}$,} \\
2r-1, &\hbox{if $\fg = \fso_{2r}$,} 
\end{cases}
\end{equation}
is the value by which the Casimir $\kappa$ acts on $L(\varepsilon_1)$.
Set $q=e^{h/2}$. The quantum dimension of $V$ is
\begin{equation}
\dim_q(V) = \tr_V(e^{h\rho}) = \epsilon+[y],
\qquad\hbox{where}\quad
[y] = \frac{q^y -q^{-y}}{q-q^{-1}},
\end{equation}
since, with respect to a weight basis of $V$, the eigenvalues of the diagonal matrix $e^{h \rho}$ are $e^{\half h (y - 2i + 1)} = q^{(y - 2i + 1)}$.

Identify a weight $\lambda = \lambda_1\varepsilon_1+\cdots+\lambda_r\varepsilon_r$ with the configuration of boxes with 
$\lambda_i$ boxes in row $i$.   If $b$ is a box in position $(i,j)$ of $\lambda$ 
then the \emph{content} of $b$ is
\begin{equation}
c(b)=j-i = \hbox{the diagonal number of $b$},
\qquad\hbox{so that}\qquad
\begin{minipage}{60pt}\begin{tikzpicture}[scale=.5]
 \foreach \x in {1,...,3} {
 \draw (0,\x) -- (3,\x);
 }
  \draw (0,0) -- (1,0);
   \foreach \x in {0,1} {
 \draw (\x, 0) -- (\x,3);
 }  
  \foreach \x in {2,3} {
 \draw (\x, 1) -- (\x,3);
 }  
 \draw (0+.5, 0+.5) node{\scriptsize$-2$} ++(0,1) node{\scriptsize$-1$} ++(0,1) node{\scriptsize$0$};
 \draw (1+.5, 1+.5) node{\scriptsize$0$} ++(0,1) node{\scriptsize$1$};
 \draw (2+.5, 1+.5)  node{\scriptsize$1$}  ++(0,1) node{\scriptsize$2$};
\end{tikzpicture}\end{minipage}
\end{equation}
are the contents of the boxes of 
$\lambda = 
3\varepsilon_1+3\varepsilon_2+\varepsilon_3$.
If $\lambda=\lambda_1\varepsilon_1+\cdots\lambda_n\varepsilon_n$,
then
$$\langle \lambda,\lambda+2\rho\rangle 
-\langle \lambda-\varepsilon_i,\lambda-\varepsilon_i+2\rho\rangle 
=2\lambda_i+2\rho_i-1=y+2\lambda_i-2i=y+2c(\lambda/\lambda^-),$$
where $\lambda/\lambda^-$ is the box at the end of row $i$ in $\lambda$.
By induction,
\begin{equation}\label{cvaluecl}
\langle \lambda,\lambda+2\rho\rangle
=y|\lambda|+2\sum_{b\in \lambda} c(b),
\end{equation}
for $\lambda = \lambda_1\varepsilon_1+\cdots+\lambda_r\varepsilon_r$ with
$\lambda_i\in \ZZ$.

Let $L(\lambda)$ be the  irreducible highest weight $\fg$-module with highest weight 
$\lambda$, and let $V=L(\varepsilon_1)$.  Then,
for $\fg=\fso_{2r+1}$, $\fsp_{2r}$ or $\fso_{2r}$,  
\begin{equation}\label{decomp1}
V\cong V^*
\quad\hbox{and}\quad
V\otimes V \cong L(0) \oplus L(2\varepsilon_1) \oplus L(\varepsilon_1+\varepsilon_2).
\end{equation}
For each component in the decomposition of $V\otimes V$ the values by which $\gamma = \sum_{b\in B} b \otimes b^*$ acts (see \eqref{gamma_value}) are
\begin{align}
&\langle 0, 0+2\rho\rangle 
- \langle \varepsilon_1,\varepsilon_1+2\rho\rangle
- \langle \varepsilon_1,\varepsilon_1+2\rho\rangle =0-y-y = -2y, 
\nonumber \\
&\langle 2\varepsilon_1, 2\varepsilon_1+2\rho\rangle 
- \langle \varepsilon_1,\varepsilon_1+2\rho\rangle
- \langle \varepsilon_1,\varepsilon_1+2\rho\rangle =4+2(y-1)-y-y = 2, 
\label{tonVtimesV} \\
&\langle \varepsilon_1+\varepsilon_2, \varepsilon_1+\varepsilon_2+2\rho\rangle 
- \langle \varepsilon_1,\varepsilon_1+2\rho\rangle
- \langle \varepsilon_1,\varepsilon_1+2\rho\rangle =2+(y-1)+(y-3)-y-y = -2.
\nonumber
\end{align}
The second symmetric and exterior powers of $V$ are 
\begin{equation}\label{decomp2}
S^2(V) = \begin{cases}
L(2\varepsilon_1)\oplus L(0), &\hbox{if $\fg = \fso_{2r+1}$ or $\fso_{2r}$,} \\
L(2\varepsilon_1), &\hbox{if $\fg = \fsp_{2r}$,}
\end{cases}
\end{equation}
and
\begin{equation}\label{decomp3}
\Lambda^2(V) = \begin{cases}
L(\varepsilon_1+\varepsilon_2), &\hbox{if $\fg = \fso_{2r+1}$ or $\fso_{2r}$,} \\
L(\varepsilon_1+\varepsilon_2)\oplus L(0), &\hbox{if $\fg = \fsp_{2r}$.}
\end{cases}
\end{equation}
For all dominant integral weights $\lambda$ 
\begin{equation}\label{tensoreps1}
L(\lambda)\otimes V=
\begin{cases}
\displaystyle{
L(\lambda)}\bigoplus \left(\bigoplus_{\lambda^\pm} L(\lambda^{\pm})\right),
&\hbox{if $\fg=\fso_{2r+1}$ and  $\lambda_r>0$,} \\
\cr
\displaystyle{\bigoplus_{\lambda^\pm} L(\lambda^{\pm})},
&\hbox{if $\fg=\fsp_{2r}$, $\fg=\fso_{2r}$, or} \\
&\hbox{if $\fg=\fso_{2r+1}$ and $\lambda_r=0$,} 
\end{cases}
\end{equation}
where the 
sum over $\lambda^\pm$ denotes a sum over all
dominant weights obtained by adding or removing a box from $\lambda$ (by a routine check using the product formula for Weyl characters in \cite[VIII\ \S 9 Prop.\ 2]{Bou}).
If $\fg=\fso_{2r}$ then addition and removal of a box should include the possibility
of addition and removal of a box marked with a $-$ sign,
and removal of a box from row $r$ when $\lambda_r=\frac{1}{2}$ changes
$\lambda_r$ to $-\frac{1}{2}$.

\paragraph{Preliminaries on quantum trace.}
This paragraph provides a brief review of quantum traces and quantum dimensions
(see also \cite[4.2.9]{CP}) in the form suitable to our needs for the proofs of 
the main theorems of this section.
If $\fg$ is a finite-dimensional complex Lie algebra with a symmetric nondegenerate $\ad$-invariant bilinear form, and $U_h\fg$ is the Drinfel'd-Jimbo quantum group corresponding to $\fg$,  then both 
\begin{equation}\label{eq:v=1}
\hbox{$U= U\fg$ with $\cR=1\otimes 1$ and $v=1$}
\qquad\hbox{and}\qquad
\hbox{$U = U_h\fg$\quad with\quad $v = e^{-h\rho}u$}
\end{equation}
are ribbon Hopf algebras (\cite[Corollary (2.15)]{LR}).  
For $U = U\fg$ or $U_h \fg$, let $V$ be a finite-dimensional $U$-module, and let $V^*$ be the dual module.  Define
$$\begin{matrix}
\mathrm{ev}\colon &V^*\otimes V &\longrightarrow &\mathbf{1} \\
&\phi\otimes v &\longmapsto &\phi(v)
\end{matrix}
\qquad\hbox{and}\qquad
\begin{matrix}
\mathrm{coev}\colon  &\mathbf{1} &\longrightarrow &V\otimes V^* \\
&1 &\longmapsto &\sum_i v_i\otimes v^i
\end{matrix}
$$
where $\{v_1,\ldots, v_n\}$ is a basis of $V$ and $\{v^1,\ldots, v^n\}$ is the dual
basis in $V^*$.  
Let $E_V$ be the composition
$$
E_V\colon V\otimes V^* \xrightarrow{v^{-1}\otimes 1}
V\otimes V^* \xrightarrow{\check R_{VV^*}} V^*\otimes V \xrightarrow{\mathrm{ev}}
\mathbf{1} \xrightarrow{\mathrm{coev}} V\otimes V^*,
$$
so that $E_V$ is a $U$-module homomorphism with image a submodule of $V\otimes V^*$
isomorphic to the trivial representation of $U$.  

Let $M$ be a $U$-module and let $\psi\in \End_U(M\otimes V)$.
Then, as operators on $M\otimes V\otimes V^*$, 
\begin{equation}\label{qtraceandE}
(\id\otimes E_V)(\psi\otimes \id)(\id\otimes E_V)
= (\id\otimes \qtr_V)(\psi)\otimes E_V,
\end{equation}
where the \emph{quantum trace} $(\id\otimes \qtr_V)(\psi)\colon M\to M$ is  given by 
\begin{equation}\label{qtrdefn}
(\id\otimes \qtr_V)(\psi) = (\id\otimes \tr_V)\big((1\otimes uv^{-1})\psi\big).
\end{equation}
The special case when $M= \mathbf{1}$ and $\psi = \id_V$ is the \emph{quantum 
dimension} of $V$,
\begin{equation}\label{framinganom}\dim_q(V) = \qtr_V(\id_V), \qquad \text{ so that } \quad E_V^2 = \dim_q(V)E_V.
\end{equation}
If $C_V\colon V\to V$ is the map defined in \eqref{qcasimir}, then 
\begin{equation}\label{EVqtraction}
(\id\otimes \qtr_V)(\check R_{VV}) = C_V^{-1}
\end{equation}
(see, for example, \cite[Prop. 3.11]{LR}). 
In the case $U = U\fg$, the ribbon element $v = 1$, so that 
$\qtr_V (\varphi) = \tr_V(\varphi)$ and $C_V = \id_V$.

\begin{remark}
The identity \eqref{qtraceandE} and the second identity in \eqref{framinganom}
are the source of the connection between quantum traces, the Jones basic
construction and conditional expectations (see \cite[Def.\ 2.6.6]{GHJ}).  These
tools are extremely powerful for the study of Temperley-Lieb algebras,
Brauer algebras, BMW algebras, and other algebras which arise as tantalizer algebras
(tensor power centralizer algebras).
\end{remark}

\subsection{The degenerate affine BMW algebra action}
\label{sec:degBMW-action}

Define $e_i$ in the degenerate affine braid algebra $\cB_k$ by 
\begin{equation}\label{rel:e_defn} 
t_{s_i}y_i = y_{i+1}t_{s_i} -(1-e_i),
\qquad \hbox{for $i=1,2,\ldots, k-1$,}
\end{equation}
so that, with $\gamma_{i,i+1}$ as in \eqref{rel:graded_braid5},
\begin{equation}\label{eq:tt-e_defn}\gamma_{i,i+1}t_{s_i} = 1 - e_i.
\end{equation}
By definition, the algebra $\cB_k$ is an algebra over a commutative base ring $C$.
Fix constants
$$\epsilon=\pm1\qquad\hbox{and}\qquad
z_0^{(\ell)}\in C\quad\hbox{for $\ell\in \ZZ_{\ge0}$.}
$$
The \emph{degenerate affine Birman-\!Wenzl-Murakami (BMW) algebra}  $\cW_k$ 
(with parameters $\epsilon$ and $z_0^{(\ell)}$) is the quotient of the degenerate affine 
braid algebra $\cB_k$ by the relations
\begin{equation}\label{rel:untwisting} 
e_i t_{s_i} = t_{s_i} e_i = \epsilon e_i,
\qquad
e_i t_{s_{i-1}}e_i = e_i t_{s_{i+1}}e_i = \epsilon e_i,
\end{equation}
\begin{equation}\label{rel:unwrapping} 
e_1 y_1^{\ell} e_1 = z_0^{(\ell)} e_1,
\qquad
e_i(y_i + y_{i+1}) = 0 = (y_i + y_{i+1})e_i.
\end{equation}
The \emph{degenerate affine Hecke algebra} $\cH_k$ is the quotient of $\cW_k$ by the relations
\begin{equation}\label{rel:gradedhecke}
e_i=0, 
\qquad\hbox{for $i=1,\ldots, k-1$.}
\end{equation}

\begin{thm}\label{degBMWaction} Let $\Phi: \cB_k \to \End_\fg(M \otimes V^{\otimes k})$ be the representation defined in Theorem \ref{degaction}. 
\item{(a)}  Let $\fg$ be $\fso_{2r+1}$, $\fsp_{2r}$ or $\fso_{2r}$ and 
$\gamma = \sum_b b \otimes b^*$ as in \eqref{sospgamma}.  Use notations for
irreducible representations as in \eqref{tensoreps1}.
Let
$$y
= \begin{cases}
2r, &\hbox{if $\fg = \fso_{2r+1}$,} \\
2r+1, &\hbox{if $\fg = \fsp_{2r}$,} \\
2r-1, &\hbox{if $\fg = \fso_{2r}$,} 
\end{cases}
\qquad
\epsilon\, =
	\begin{cases}
		\!~~~1, &\hbox{if $\fg = \fso_{2r+1}$,} \\
		-1, &\hbox{if $\fg = \fsp_{2r}$,} \\
		\!~~~1, &\hbox{if $\fg = \fso_{2r}$,}
	\end{cases} \qquad
 V=L(\varepsilon_1),
$$
and let
$$
z_0^{(\ell)} = \epsilon(\id\otimes \tr_V)((\hbox{$\frac12$} y + \gamma)^\ell),
\qquad\hbox{
for $\ell\in \ZZ_{\ge 0}$.}
$$
Then
$\Phi\colon \cB_k\to \End_{U}(M\otimes V^{\otimes k})$
is a representation of the degenerate affine BMW algebra $\cW_k$.

\item[(b)]  
If $\fg= \fgl_r$, $\gamma  = \sum_b b \otimes b^*$ is as in \eqref{glgamma},
and $V = L(\varepsilon_1)$ then 
$\Phi\colon \cB_k\to \End_{U\fg}(M\otimes V^{\otimes k})$
is a representation of the degenerate affine Hecke algebra. If $\fg=\fsl_r$, $\gamma  = \sum_b b \otimes b^*$ is as in \eqref{slgamma}, and $V=L(\overline{\varepsilon_1})$ then $\Phi'\colon \cB_k\to \End_{U\fg}(M\otimes V^{\otimes k})$ given by
$$\Phi'(t_{s_i})=\Phi(t_{s_i}), \quad \Phi'(\gamma_{\ell,m})=\hbox{$\frac{1}{r}$}+\Phi(\gamma_{\ell,m}), \quad\hbox{and}\quad \Phi'(\kappa_i) = \Phi(\kappa_i) 
$$
extends to a representation of the degenerate affine Hecke algebra. 
\end{thm}

\begin{proof}
\smallskip\noindent 
(a)  
The action of $\gamma$ on the tensor product of two simple modules is given in \eqref{gamma_value}, so the computations in \eqref{tonVtimesV} determine the action of $\gamma$ on the components of $V\otimes V$.  The decompositions of the second symmetric and exterior powers in \eqref{decomp2} and \eqref{decomp3}
determine the action of $t_{s_1}$ on $V\otimes V$.
The operator $\Phi(e_1)$ is determined from $\Phi(t_{s_1})$ and $\Phi(\gamma)$ via \eqref{eq:tt-e_defn},
$$\Phi(\gamma)\Phi(t_{s_1}) = 1 - \Phi(e_1).$$  
In summary, $\Phi(t_{s_1})$, $\Phi(e_1)$, and $\Phi(\gamma)$ act on the components of $V \otimes V$  by
$$
\begin{array}{lcccl}
&L(0) &L(2\varepsilon_1) &L(\varepsilon_1+\varepsilon_2)\qquad  \\
\Phi(\gamma_{1,2}) &-y &1 &-1  \\
\Phi(t_{s_1}) &\epsilon &1 &-1  \\
\Phi(e_1) &1+\epsilon y &0 &0  
\end{array}
$$
where $y$ and $\gamma$ are as in \eqref{explrho} and \eqref{gamma_defn}, respectively.
The first relation in \eqref{rel:untwisting} follows.

Since $\dim(V) = \epsilon+y$,
the first identity in \eqref{framinganom}
gives that 
\begin{equation}\label{e=E}
\Phi(e_1)=  \epsilon E_V.
\end{equation}
By \eqref{eq:v=1}, \eqref{qtraceandE}, and \eqref{EVqtraction},  
\begin{align*}
\Phi(e_it_{s_{i-1}}e_i)
&=\epsilon(1\otimes E_V)(\check R_{VV}\otimes 1)(1\otimes E_V)\epsilon
=(\id\otimes\tr_V)(\check R_{VV})\otimes E_V \nonumber\\
&= C_V^{-1}\otimes E_V
=\id\otimes E_V
=\epsilon\Phi(e_i), 
\end{align*}
which establishes the second relation in \eqref{rel:untwisting}. 
Since $y=\langle \varepsilon_1,\varepsilon_1+2\rho\rangle = \Phi(\kappa_i)$, it follows
from \eqref{gensAtogensB} that 
$\Phi(y_1) = \Phi(\half \kappa_1 + \gamma_{0,1})
= \half y + \gamma$, 
and by \eqref{qtraceandE},
\begin{align}
\Phi(e_1y_1^\ell e_1) 
&= \epsilon(\id\otimes E_V)(\hbox{$\frac{1}{2}$}y +\gamma)^\ell\epsilon(\id\otimes E_V)
= (\id\otimes\tr_V)((\hbox{$\frac{1}{2}$}y +\gamma)^\ell)\otimes E_V \nonumber \\
&=\epsilon\, (\id\otimes\tr_V)((\hbox{$\frac{1}{2}$}y +\gamma)^\ell) \Phi(e_1) = z_0^{(\ell)}\Phi(e_1),
\label{determining_z0}
\end{align}
which gives the first relation in \eqref{rel:unwrapping}.
Since the $y_i$ commute and $t_{s_i} ( y_i + y_{i+1}) = (y_i + y_{i+1})t_{s_i}$, it follows that
$$
 e_i(y_i + y_{i+1}) = (t_{s_i} y_i - y_{i+1} t_{s_i} + 1)(y_i + y_{i+1}) 
		=(y_i + y_{i+1})(t_{s_i} y_i - y_{i+1} t_{s_i} + 1)\\
=(y_i+y_{i+1})e_i.
$$
For $b\in U\fg$ or $U\fg\otimes U\fg$, let
$b_i$ and $b_{i,i+1}$ denote the 
action of an element $b$ on the $i$th, respectively $i$th and $(i+1)$st,
factors of $V$ in $M\otimes V^{\otimes(i+1)}$.  Then, as operators on 
$M\otimes V^{\otimes(i+1)}$,
\begin{align*}
(y_i+y_{i+1})e_i
&= \left(\half\kappa_i + \sum_{r=0}^{i-1} \gamma_{r,i} 
+ \hbox{$\frac12$}\kappa_{i+1}
+\sum_{r=0}^i \gamma_{r, i+1}\right)e_i 
= \left(\half\Delta(\kappa)_{i,i+1} + \sum_{r=0}^{i-1} (\gamma_{r,i} + \gamma_{r, i+1})\right)e_i \\
&=\left(\half\Delta(\kappa)_{i,i+1} 
+ \sum_{r=0}^{i-1} \sum_{b} b \otimes \Delta(b^*)_{i,i+1}\right)
e_i 
=0, 
\end{align*}
because $e_i$ is a projection onto $L(0)$ and the action of $b^*$ and $\kappa$ on $L(0)$ is $0$.

\smallskip\noindent
(b)  
In the case where $\fg = \fgl_{r}$ and $V= L(\varepsilon_1)$, 
$$V \otimes V = L(2\varepsilon_1) \oplus L(\varepsilon_1 + \varepsilon_2),
\quad \hbox{ with }  
\Lambda^2(V) = L(\varepsilon_1+\varepsilon_2)
\quad\hbox{and}\quad
S^2(V) = L(2\varepsilon_1).
$$
So by \eqref{gamma_value},
\begin{equation}\label{glcase}
\begin{array}{lccl}
&L(2\varepsilon_1) &L(\varepsilon_1+\varepsilon_2)\qquad  \\
\Phi(\gamma_{1,2}) &1 &-1  \\
\Phi(t_{s_1}) &1 &-1  
\end{array}
\qquad\hbox{and}\qquad
\Phi(e_1) = \Phi(\gamma) - \Phi(t_{s_1})=0.
\end{equation}
In the case where $\fg = \fsl_r$ and $V=L(\bar\varepsilon_1)$,
$$V \otimes V = L(\overline{2\varepsilon_1}) \oplus L(\overline{\varepsilon_1 + \varepsilon_2}), \quad \text{ with }  
\Lambda^2(V) = L(\overline{\varepsilon_1+\varepsilon_2})\quad\hbox{and}\quad
S^2(V) = L(\overline{2\varepsilon_1}).$$
As the map $\phi: \cB_k \to \cB_k$ given by 
$$t_{s_i} \mapsto t_{s_i}, \qquad \gamma_{i,j} \mapsto \gamma_{i,j} - a, \qquad \kappa_i \mapsto \kappa_i, \qquad \text{ for fixed $a \in C$}$$
is an automorphism, the result follows from \eqref{glcase} and \eqref{casconv}.

\end{proof}

\begin{remark}  Fix $b_1,\ldots, b_r\in C$.  The \emph{degenerate cyclotomic BMW algebra} 
$\cW_{r,k}(b_1,\ldots, b_r)$ 
is the degenerate affine BMW algebra with the additional relation 
\begin{equation}\label{cycrelation}
(y_1-b_1)\cdots (y_1-b_r) = 0.
\end{equation}
The \emph{degenerate cyclotomic Hecke algebra} $\cH_{r,k}(b_1,\ldots, b_r)$ is the
degenerate affine Hecke algebra $\cH_k$ with the additional relation
\eqref{cycrelation}.
In Theorem \ref{degBMWaction}, if $\Phi(y_1)$ has eigenvalues
$u_1,\ldots, u_r$ then $\Phi$ is a representation of $\cW_{r,k}(u_1,\ldots, u_r)$ or $\cH_{r,k}(u_1,\ldots, u_r)$.
\end{remark}

\begin{remark}In general, for any constants $a_0$, $a$, and  $c$,  the map $\phi: \cB_k \to \cB_k$ given by 
$$t_{s_i} \mapsto t_{s_i}, \quad \gamma_{i,j} \mapsto \gamma_{i,j} - c, \quad  \kappa_0 \mapsto \kappa_0 - a_0, \quad \text{ and }  \quad \kappa_j \mapsto \kappa_j - a,\qquad \text{ for $j = 1, \dots, k,$}$$
is an automorphism. So, following the proof of Theorem \ref{degBMWaction}(b), $\Phi': \cB_k \to \End_{U\fg}(M \otimes V^{\otimes k})$ given by
$$\Phi'(t_{s_i})=\Phi(t_{s_i}), \quad \Phi'(\gamma_{\ell,m})=\hbox{$\frac{1}{r}$}+\Phi(\gamma_{\ell,m}),$$
$$\Phi'(\kappa_0) = a_0 + \Phi(\kappa_0), \quad\hbox{and}\quad \Phi'(\kappa_j) = a + \Phi(\kappa_j) \qquad \text{ for $j = 1, \dots, k$,}
$$
also extends to a representation of $\cH_k$ when $\fg = \fsl_r$. When $M = L(\mu)$ is a finite-dimensional highest weight module taking $a_0 = \frac{|\mu|}{r}$ and $a = \frac{1}{r}$ is combinatorially convenient.
\end{remark}

%
%
\subsection{The affine BMW algebra action}
\label{sec:affBMW-action}
Let $C$ be a commutative ring and let $CB_k$ be the group algebra of the affine
braid group.  Fix constants
$$q, z\in C
\qquad\hbox{and}\qquad
Z_0^{(\ell)}\in C,\quad\hbox{for $\ell\in \ZZ$,}
$$
with $q$ and $z$ invertible.  Let $Y_i = zX^{\varepsilon_i}$ so that
\begin{equation}
	\label{rel:Y}
		Y_1 = zX^{\varepsilon_1},\qquad
		Y_i = T_{i-1}Y_{i-1}T_{i-1},
	\qquad\hbox{and}\qquad
		Y_iY_j=Y_jY_i,\ \ \hbox{for $1\le i,j\le k$.}
\end{equation}
In the affine braid group
\begin{equation}
	\label{rel:Antisymmetry_T}
		T_i Y_iY_{i+1}  = Y_iY_{i+1}T_i.
\end{equation}
Assume $q-q^{-1}$ is invertible in $C$. Define $E_i \in CB_k$ by\begin{equation}\label{rel:E_Defn} 
T_i Y_i = Y_{i+1} T_i -(q-q^{-1})Y_{i+1}(1-E_i).
\end{equation}
The \emph{affine BMW algebra} $W_k$ 
is the quotient of the group algebra $CB_k$ by the relations
\begin{equation}\label{rel:Untwisting} 
E_i T^{\pm1}_i = T^{\pm1}_i E_i = z^{\mp1}E_i,
\qquad
E_i T_{i-1}^{\pm1}E_i = E_i T_{i+1}^{\pm1}E_i = z^{\pm1}E_i,
\end{equation}
\begin{equation}\label{rel:Unwrapping} 
E_1 Y_1^{\ell} E_1 = Z_0^{(\ell)} E_1,
\qquad
E_iY_iY_{i+1} = E_i = Y_iY_{i+1}E_i.
\end{equation}
Left multiplying \eqref{rel:E_Defn} by $Y_{i+1}^{-1}$ and using the second identity in 
\eqref{rel:Y}
shows that \eqref{rel:E_Defn} is equivalent to
$T_i - T_i^{-1} = (q - q^{-1})(1 - E_i)$. So 
\begin{equation}\label{Skein} 
E_i = 1-\frac{T_i-T_i^{-1}}{q-q^{-1}},
\qquad\hbox{and}\qquad
E_i^2 = \left(1 + \frac{z-z^{-1}}{q-q^{-1}}\right)E_i
\end{equation}
follows by multiplying the first equation in \eqref{Skein} by $E_i$ and using 
\eqref{rel:Untwisting}.

The \emph{affine Hecke algebra} $H_k$ is the affine BMW algebra $W_k$ 
with the additional relations
\begin{equation}\label{rel:hecke}
	E_i=0, 
	\qquad\hbox{for $i=1,\ldots, k-1$.}
\end{equation}


%


\begin{thm}\label{BMWaction}   Let $\Phi: CB_k \to \End_{U_h\fg}(M \otimes V^{\otimes k})$ be the representation defined in Theorem \ref{affaction}. 
\item[(a)]  Let $\fg$ be $\fso_{2r+1}$, $\fsp_{2r}$ or $\fso_{2r}$ and let
$\gamma  = \sum_b b \otimes b^*$ is as in \eqref{sospgamma}.
Let
$$y
= \begin{cases}
2r, &\hbox{if $\fg = \fso_{2r+1}$,} \\
2r+1, &\hbox{if $\fg = \fsp_{2r}$,} \\
2r-1, &\hbox{if $\fg = \fso_{2r}$,} 
\end{cases}
\qquad
\epsilon\, =
	\begin{cases}
		\!~~~1, &\hbox{if $\fg = \fso_{2r+1}$,} \\
		-1, &\hbox{if $\fg = \fsp_{2r}$,} \\
		\!~~~1, &\hbox{if $\fg = \fso_{2r}$,}
	\end{cases} \qquad
 V=L(\varepsilon_1),
$$
$z=\epsilon\, q^y$,
and
$$
Z_0^{(\ell)}=\epsilon\, (\id\otimes\qtr_V)\big((z\cR_{21}\cR)^{\ell}\big), \qquad \text{for $\ell \in \ZZ$.}
$$
 Then $\Phi: C B_k\rightarrow \End_U(M\otimes V^{\otimes k})$ is a representation of the affine BMW algebra $W_k$.  

\item[(b)]  
If $\fg= \fgl_r$, 
$\gamma  = \sum_b b \otimes b^*$ is as in \eqref{glgamma},
and $V = L(\varepsilon_1)$, then 
$\Phi\colon C B_k\to \End_U(M\otimes V^{\otimes k})$
is a representation of the affine Hecke algebra. 
If $\fg=\fsl_r$,
$\gamma  = \sum_b b \otimes b^*$ is as in \eqref{slgamma},
and $V=L(\overline{\varepsilon_1})$, then 
$$\Phi'\colon C B_k\to \End_{U}(M\otimes V^{\otimes k})
\quad\hbox{given by}\quad \Phi'(T_i)=q^{1/r}\Phi(T_i)\quad\hbox{and}\quad \Phi'(X^{\varepsilon_i})=\Phi(X^{\varepsilon_i}),$$
extends to a representation of the affine Hecke algebra.
\end{thm}

\begin{proof}
(a) 	
By \eqref{fulltwist}, the computations in \eqref{tonVtimesV}
determine the action of $\check R^2_{VV}$ on the components of $V\otimes V$.  The operator
$\Phi(T_1) = \check R_{VV}$ is the square root of $\check R^2_{VV}$ and, at $q=1$ ,
specializes to $t_{s_1}$, the operator that switches the factors in $V\otimes V$.  Thus equations \eqref{decomp2} and \eqref{decomp3} determine the sign of $\Phi(T_1)$ on each component. The operator
$\Phi(E_1)$ is determined from $\Phi(T_1)$ via the first identity in  \eqref{Skein},
	$$\Phi(E_i) = 1 - \frac{\Phi(T_i)-\Phi(T_i^{-1})}{q-q^{-1}}.$$
Then 
$\check R_{VV}^2, \Phi(T_1)$ and $\Phi(E_1)$ act on the components of $V\otimes V$  by
$$
\begin{array}{cccc}
	&L(0)&L(2\varepsilon_1)&L(\varepsilon_1+\varepsilon_2)\\
	\check R_{VV}^2 & q^{-2y} & q^2 & q^{-2} \\
		\Phi(T_1)&\epsilon q^{-y}&q&-q^{-1} \\
	\Phi(E_1)& 1+\epsilon[y]&0&0
	\end{array}\qquad
\hbox{where}\quad [y] = \frac{q^y- q^{-y}}{q-q^{-1}}.
$$
The first relation in \eqref{rel:Untwisting} follows from
$$\Phi(E_1T_1)= \epsilon q^{-y} \Phi(E_1) = z^{-1}\Phi(E_1).$$
Since $\dim_q(V) = \epsilon+[y]$, \eqref{framinganom}
gives 
\begin{equation}\label{E=E}
\Phi(E_1)=\epsilon E_V.
\end{equation}
By \eqref{qtraceandE}, \eqref{EVqtraction}, \eqref{qCasvalue}, and \eqref{explrho},
\begin{align*}
	\Phi(E_iT_{i-1}E_i)
		&=\epsilon(1\otimes E_V)(\check R_{VV}\otimes 1)(1\otimes E_V)\epsilon
			=(\id\otimes\qtr_V)(\check R_{VV})\otimes E_V \\
		&= C_V^{-1}\otimes E_V
			= q^{\langle \varepsilon_1,\varepsilon_1+2\rho\rangle}(\id\otimes E_V)
			=q^y\epsilon\Phi(E_i) = z\Phi(E_i).
\end{align*}
This establishes the second relation in \eqref{rel:Untwisting}. 
By \eqref{qtraceandE},
\begin{align}
	\Phi(E_1Y_1^\ell E_1) 
		&= \epsilon(\id\otimes E_V)(z\cR_{21}\cR)^\ell\epsilon(\id\otimes E_V) 
		= (\id\otimes\qtr_V)((z\cR_{21}\cR)^\ell)\otimes E_V 
		\nonumber
		\\
		&=\epsilon\, (\id\otimes\qtr_V)((z\cR_{21}\cR)^\ell) \Phi(E_1) = Z_0^{(\ell)}\Phi(E_1),
		\label{determining_Z0}
\end{align}
which gives the first relation in \eqref{rel:Unwrapping}.
Since the $Y_i$ commute and $T_iY_iY_{i+1}=Y_iY_{i+1}T_i$,
$$
E_iY_iY_{i+1}
	=\left(1-\frac{T_i-T_i^{-1}}{q-q^{-1}}\right)Y_iY_{i+1}
	=Y_iY_{i+1}\left(1-\frac{T_i-T_i^{-1}}{q-q^{-1}}\right)
	=Y_iY_{i+1}E_i.
$$
The proof that $E_iY_iY_{i+1}=E_i$
is exactly as in the proof of \cite[Thm. 6.1(c)]{OR}:
Since $\Phi(E_1)=\epsilon E_V$, using $E_1T_1 = z^{-1}E_1$ and the pictorial equalities
$$
\epsilon\, z^2\cdot
\beginpicture
	\setcoordinatesystem units <.5cm,.5cm>         
	\setplotarea x from -2.5 to 1.5, y from -2 to 4    
		\plot -1.5 3.7 -1.5 1.12 /
		\plot -1.5 0.88 -1.5 -0.88 /
		\plot -1.5 -1.12 -1.5 -1.75 /
		\plot -1.25 3.7 -1.25 1.12 /
		\plot -1.25 0.88 -1.25 -0.88 /
		\plot -1.25 -1.12 -1.25 -1.75 /
		\ellipticalarc axes ratio 1:1 360 degrees from -1.5 3.7 center 
		at -1.375 3.7
		\put{$*$} at -1.375 3.7  
		\ellipticalarc axes ratio 1:1 180 degrees from -1.5 -1.75 center 
		at -1.375 -1.75 
		\plot  1 2   1 0.5 /
		\plot  1 -0.5   1 -1.5 /
	\setlinear
		\plot -0.3 1.5  -1.1 1.5 /
		\ellipticalarc axes ratio 2:1 180 degrees from -1.65 1.5  center 
		at -1.65 1.25 
		\plot -1.65 1  -0.3 1 /
		\plot -0.3 -0.5  -1.1 -0.5 /
		\ellipticalarc axes ratio 2:1 180 degrees from -1.65 -.5  center 
		at -1.65 -.75 
		\plot -1.65 -1  -0.3 -1 /
		\setquadratic
		\plot  -0.3 1  -0.05 0.8  -0 0.5 /
		\plot   0 0.5   0.15 0.2   0.7 0 /
		\plot   1 -0.5   0.9 -0.15   0.7 0 /
		\plot  -0.3 1.5  -0.05 1.7  -0 2 /
		\plot   -0.3 -0.5     -.1 -.425  -0.05 -0.325 /
		\plot   -0.05 -0.325   0.15 -0.1   0.35 -0.05 /
		\plot   0.65 0.15   0.9 0.25   1 0.5 /
		\plot  -0.3 -1  -0.05 -1.2  -0 -1.5 /
		\ellipticalarc axes ratio 1:1 180 degrees from 1 2 center 
		at 0.5 2
		\ellipticalarc axes ratio 1:1 180 degrees from 0 3.5 center 
		at 0.5 3.5
	\endpicture
= \epsilon\, z^2\cdot
	\beginpicture
	\setcoordinatesystem units <.5cm,.5cm>         
	\setplotarea x from -2 to 2.5, y from -2 to 4    
		\plot -1.5 3.7 -1.5 1 /
		\plot -1.25 3.7 -1.25 1 /
		\ellipticalarc axes ratio 1:1 360 degrees from -1.5 3.7 center 
			at -1.375 3.7
		\put{$*$} at -1.375 3.7  
		\ellipticalarc axes ratio 1:1 180 degrees from -1.5 -1.75 center 
			at -1.375 -1.75 
	\setquadratic
		\plot   0 2   0.15 1.7   0.7 1.5 /
		\plot   1 1   0.9 1.35   0.7 1.5 /
		\plot   0 1   0.1 1.3   0.35 1.45 /
		\plot   0.65 1.65   0.9 1.75   1 2 /
		\plot   -1.5 1   -1.15 0.55   0.215 0.295 /
		\plot   1.65 -0.2    1.25 0.2   0.215 0.295 /
		\plot   -1.25 1   -0.9 0.7   0.205 0.515 /
		\plot    1.9 -0.2    1.5 0.35   0.205 0.515 /
		\plot   0.65 0.65   0.9 0.75   1 1 /
		\plot   -0.35 0.65  -0.1 0.75   0 1 /
		\plot   -1 -0.2   -0.9 0.1   -0.65 0.25 /
		\plot   -0 -0.2    0.1 0.1    0.25 0.2 /
		\plot   0 -0.2   0.1 -0.6   0.5 -0.95 /
		\plot   1 -1.7   0.9 -1.3   0.5 -0.95 /
		\plot   -1 -0.2   -0.9 -0.6   -0.5 -0.95 /
		\plot   0 -1.7   -0.1 -1.3   -0.5 -0.95 /
		\plot  1.9 -0.2    1.65 -0.8    0.8 -1 /
		\plot  -1.25 -1.75  -1.05 -1.35  -0.45 -1.1 /   
		\plot  -0.35 -0.9    -0.15 -0.875   0.2  -0.85 /
		\plot  -0.2 -1.05    0.05 -1.05   0.4  -1.05 /
		\plot  1.65 -0.2    1.45 -0.7   0.55  -0.85 /
		\plot  -1.5 -1.75  -1.35 -1.35  -0.7 -1 /   
		\ellipticalarc axes ratio 1:1 180 degrees from 1 2 center 
			at 0.5 2
		\ellipticalarc axes ratio 1:1 180 degrees from 0 3.5 center 
			at 0.5 3.5
	\endpicture
= ~\epsilon\, z^2 z^{-1}\cdot
	\beginpicture
	\setcoordinatesystem units <.5cm,.5cm>         
	\setplotarea x from -2 to 2, y from -2 to 4    
		\plot -1.5 3.7 -1.5 1 /
		\plot -1.25 3.7 -1.25 1 /
		\ellipticalarc axes ratio 1:1 360 degrees from -1.5 3.7 center 
			at -1.375 3.7
		\put{$*$} at -1.375 3.7  
		\ellipticalarc axes ratio 1:1 180 degrees from -1.5 -1.75 center 
			at -1.375 -1.75 
		\plot  0 2   0 1 /
		\plot  1 2   1 1 /
	\setquadratic
		\plot   -1.5 1   -1.15 0.55   0.215 0.295 /
		\plot   1.65 -0.2    1.25 0.2   0.215 0.295 /
		\plot   -1.25 1   -0.9 0.7   0.205 0.515 /
		\plot    1.9 -0.2    1.5 0.35   0.205 0.515 /
		\plot   0.65 0.65   0.9 0.75   1 1 /
		\plot   -0.35 0.65  -0.1 0.75   0 1 /
		\plot   -1 -0.2   -0.9 0.1   -0.65 0.25 /
		\plot   -0 -0.2    0.1 0.1    0.25 0.2 /
		\plot   0 -0.2   0.1 -0.6   0.5 -0.95 /
		\plot   1 -1.7   0.9 -1.3   0.5 -0.95 /
		\plot   -1 -0.2   -0.9 -0.6   -0.5 -0.95 /
		\plot   0 -1.7   -0.1 -1.3   -0.5 -0.95 /
		\plot  1.9 -0.2    1.65 -0.8    0.8 -1 /
		\plot  -1.25 -1.75  -1.05 -1.35  -0.45 -1.1 /   
		\plot  -0.35 -0.9    -0.15 -0.875   0.2  -0.85 /
		\plot  -0.2 -1.05    0.05 -1.05   0.4  -1.05 /
		\plot  1.65 -0.2    1.45 -0.7   0.55  -0.85 /
		\plot  -1.5 -1.75  -1.35 -1.35  -0.7 -1 /   
		\ellipticalarc axes ratio 1:1 180 degrees from 1 2 center 
			at 0.5 2
		\ellipticalarc axes ratio 1:1 180 degrees from 0 3.5 center 
			at 0.5 3.5
	\endpicture
$$
it follows that 
$\Phi(E_1Y_1Y_2T_1^{-1}) 
	= \epsilon (1\otimes E_V)\Phi(zX^{\varepsilon_1})
	\Phi(zT_1X^{\varepsilon_1})$
acts as 
$\epsilon z^2 z^{-1}\cdot \check R_{L(0),M}
	\check R_{M,L(0)}(\id_M\otimes E_V)$.
By \eqref{qcasimir}, this is equal to
\begin{align*}
	\epsilon z (C_{M}\otimes C_{L(0)})
	C_{M\otimes L(0)}^{-1} (\id_M\otimes E_V)
	=\epsilon z\cdot C_MC_M^{-1}(\id_M\otimes E_V)
	=z \cdot \Phi(E_1) = \Phi(E_1T_1^{-1}),
\end{align*}
so that $\Phi(E_1Y_1Y_2T_1^{-1}) = \Phi(E_1T_1^{-1})$.  This
establishes the second relation in  \eqref{rel:Unwrapping}.

\smallskip\noindent
(b) 
In the case where $\fg = \fgl_{r}$ and $V= L(\varepsilon_1)$, 
$$V\otimes V=L(2\varepsilon_1)\oplus L(\varepsilon_1+\varepsilon_2)\quad \hbox{ with }  
S^2(V) = L(2\varepsilon_1)
\quad\hbox{and}\quad
\Lambda^2(V) = L(\varepsilon_1+\varepsilon_2).$$
So by \eqref{fulltwist},
\begin{equation}\label{qglcase}
		 \begin{array}{lccl}
			 &L(2\varepsilon_1) &L(\varepsilon_1+\varepsilon_2)  \\
			 \Phi(\check R^2_{VV}) &q^2 &q^{-2}  \\
			 \Phi(T_1) &q &-q^{-1}
		 \end{array}
\qquad\hbox{so that}\qquad
\Phi(E_1) = 1-\frac{\Phi(T_1)-\Phi(T_1)^{-1}}{q-q^{-1}}=0.
\end{equation}
In the case where $\fg = \fsl_r$ and $V=L(\bar\varepsilon_1)$,
$$V \otimes V = L(\overline{2\varepsilon_1}) \oplus L(\overline{\varepsilon_1 + \varepsilon_2}), \quad \text{ with }  
\Lambda^2(V) = L(\overline{\varepsilon_1+\varepsilon_2})\quad\hbox{and}\quad
S^2(V) = L(\overline{2\varepsilon_1}).$$
Since the map $\phi: B_k \to B_k$ given by 
$$T_{i} \mapsto aT_{i}, \qquad X^{\varepsilon_i} \mapsto X^{\varepsilon_i}, \qquad \text{ for invertible $a \in C$}$$
is an automorphism, the result then follows from \eqref{qglcase} and \eqref{casconv} (also see \cite[Prop.\ 4.4]{LR}).
\end{proof}

\begin{remark}  Fix $b_1,\ldots, b_r\in C$.  
The \emph{cyclotomic BMW algebra} $W_{r,k}(b_1,\ldots, b_r)$ 
is the affine BMW algebra $W_k$ with the additional relation 
\begin{equation}\label{cycrelationA}
	(Y_1-b_1)\cdots (Y_1-b_r) = 0.
\end{equation}
The \emph{cyclotomic Hecke algebra} $H_{r,k}(b_1,\ldots, b_r)$ is the affine Hecke algebra
$H_k$ with 
the additional relation \eqref{cycrelationA}.
 In Theorem \ref{BMWaction}, if $\Phi(Y_1)$ has eigenvalues
$u_1,\ldots, u_r$, then $\Phi$ is a representation of $W_{r,k}(u_1,\ldots, u_r)$ or $H_{r,k}(u_1,\ldots, u_r)$. 
\end{remark}

\section{Central element transfer via Schur-Weyl duality}\label{sec:Zs}
In Theorem \ref{degBMWaction} and Theorem \ref{BMWaction}, 
the parameters 
$$z_0^{(\ell)} = \epsilon(\id\otimes \tr_V)((\hbox{$\frac12$} y + \gamma)^\ell)
\qquad\hbox{and}\qquad
Z_0^{(\ell)} = \epsilon\, (\id\otimes\qtr_V)\big((z\cR_{21}\cR)^{\ell}\big)
$$
of the degenerate affine BMW algebra and 
affine BMW algebra, respectively, arise naturally from the action on tensor space.
It is a consequence of \cite[Prop. 1.2]{Dr} that these
are central elements of the enveloping algebra $U\fg$ and the quantum group $U_h\fg$, respectively:
$$z_0^{(\ell)}\in Z(U\fg)
\qquad\hbox{and}\qquad Z_0^{(\ell)}\in Z(U_h\fg).$$
 The Harish-Chandra isomorphism provides isomorphisms between the centers
 $Z(U\fg)$ or $Z(U_h\fg)$ and rings of symmetric functions. In this section
we show how to use the recursive formulas 
of \cite{Naz} and \cite{BB} for the central
elements $z_k^{(\ell)}$ and $Z_k^{(\ell)}$ in the degenerate affine and affine BMW algebras 
(formulas \eqref{grproduct} and \eqref{Z+product})
to determine the Harish-Chandra images of $z_0^{(\ell)}$ and $Z_0^{(\ell)}$.

\paragraph{Preliminaries on the Harish-Chandra isomorphisms.}
\label{sec:HC}
Let $\fg$ be a finite-dimensional complex Lie algebra with a symmetric nondegenerate $\ad$-invariant bilinear form.
The triangular decomposition $\fg = \fn^- \oplus \fh\oplus \fn^+$ (see \cite[VII \S 8 no.\ 3 Prop. 9]{Bou}) yields triangular decompositions of both the enveloping algebra $U= U\fg$ and the quantum group $U= U_h\fg$ in the form
$U = U^{-}U_{0}U^+$. If $U=U_h\fg$ then
$U_0 = \hbox{span}\{K^{\lambda^\vee}\ |\ \lambda^\vee\in \fh_\ZZ\}$ with 
$K^{\lambda^\vee}K^{\nu^\vee} = K^{\lambda^\vee+\nu^\vee}$, where
$\fh_\ZZ$ is a lattice in $\fh$.  Alternatively,
$$U_{0} = U\fh 
= \CC[h_1,\ldots, h_r]  \quad \text{if $U = U\fg$} 
\qquad \text{ and } \qquad 
U_{0} = \CC[L_1^{\pm1},\ldots, L_r^{\pm1}]  \quad \text{if $U = U_h\fg$},$$ 
where $h_1,\ldots, h_r$ is a basis of $\fh_\ZZ$, and $L_i = K^{h_i} = q^{h_i}$.

For $\mu\in \fh^*$, define the ring homomorphisms $\ev_\mu: U_{0} \to \CC$ by
\begin{equation}\label{sigandevdefn1}
\mathrm{ev}_\mu(h)= \langle \mu,h\rangle \qquad\hbox{and}\qquad
\mathrm{ev}_\mu(K^{\lambda^\vee})=q^{\langle \mu,\lambda^\vee\rangle}
\end{equation}
for $h\in \fh$ and $K^{\lambda^\vee}$ with $\lambda^\vee\in \fh_\ZZ$. For $\rho = \half \sum_{\alpha \in R^+} \alpha$ as in \eqref{Casvalue}, let $\sigma_\rho$ be the algebra automorphism given by
\begin{equation}
\sigma_\rho(h_i) = h +  \langle \rho, h_i\rangle \qquad \text{ and } \qquad 
\sigma_\rho(L_i) = q^{\langle \rho, h_i\rangle}L_i.
\end{equation} 

Define a vector space homomorphism by 
\begin{equation}\label{pi0defn1}
\pi_0\colon U \longrightarrow U_{0}
\qquad\hbox{by}\qquad
\pi_0 = \varepsilon^-\otimes \id\otimes \varepsilon^+\colon U_{-} \otimes U_{0} \otimes U_+\longrightarrow U_{0},
\end{equation}
where $\varepsilon^-\colon U_{-}\to \CC$ and $\varepsilon^+\colon U_+\to \CC$ are
the algebra homomorphisms determined by 
$$
\varepsilon^-(y) = 0
\qquad\hbox{and}\qquad
\varepsilon^+(x)=0,
\qquad\hbox{for $x\in \fn^+$ and $y\in \fn^-$, or}
$$
$$
\varepsilon^-(F_i) = 0
\qquad\hbox{and}\qquad
\varepsilon^+(E_i)=0,
\qquad\hbox{for $i=1,\ldots, n$.}
$$

The following important theorem says that both the center of $U\fg$ and the center of $U_h\fg$
are isomorphic to rings of {symmetric functions}. 

\begin{thm}[{Harish-Chandra/Chevalley isomorphism, \cite[VII \S 8 no.\ 5 Thm. 2]{Bou} and \cite[Thm. 9.1.6]{CP}}]
\label{HCIso}
Let $U = U\fg$ or $U_h\fg$, so that
$$U_{0} = \CC[h_1,\ldots, h_r] \ \text{ if $U = U\fg$} \qquad \text{ and } \qquad
U_{0} = \CC[L_1^{\pm1},\ldots, L_r^{\pm1}] \quad \text{ if $U = U_h\fg$}.$$  
Let $L(\mu)$ denote the irreducible $U$-module of highest weight $\mu$.
Then the restriction of 
$\pi_0$ to the center of $U$,
$$
\begin{matrix}
\pi_0\colon Z(U) &\longrightarrow &\sigma_\rho(U_{0}^{W_0}),\\
z &\longmapsto &\sigma_\rho(s)
\end{matrix}
\qquad\hbox{is an algebra isomorphism},
$$
where $s \in U_0^{W_0}$
is the symmetric function determined by
	$$z\quad \text{ acts on $L(\mu)$ by }\quad
	 \mathrm{ev}_\mu (\sigma_\rho(s)) = \mathrm{ev}_{\mu+\rho}(s),
	 \qquad\hbox{for $\mu\in \fh^*$.}$$
\end{thm}

\subsection{Central elements $z_V^{(\ell)}$}
\label{sec:Nazarov-to-HC}

Let $z_0^{(\ell)}$ and $\epsilon$ be the parameters of the degenerate affine BMW algebra
$\cW_k$.  Let $u$ be a variable and define $z_i^{(\ell)}\in \cW_k$ for $i=1,\ldots, k-1$ by
\begin{equation}\label{grproduct}
z_i(u) + \epsilon u- \half
=(z_0(u) + \epsilon u- \half)
\prod_{j=1}^{i} \frac{(u+y_j-1)(u+y_j+1)(u-y_j)^2}{(u+y_j)^2(u-y_j+1)(u-y_j-1)},
\end{equation}
where
$$z_i(u) = \sum_{\ell\in \ZZ_{\ge 0}} z_i^{(\ell)}u^{-\ell},
\qquad\hbox{for $i=0,1,\ldots, k-1$.}$$
The following proposition from \cite[Lemma 3.8]{Naz} is  proved also in 
\cite[Theorem 3.2 and Remark 3.4]{DRV}

\begin{prop}\label{ziprop}  In the degenerate affine BMW algebra $\cW_k$,
$$ e_{i+1}y_{i+1}^\ell e_{i+1}  = z_i^{(\ell)} e_{i+1},
\qquad\hbox{for $i=0,\ldots, k-1$ and $\ell\in \ZZ_{\ge 0}$.}
$$
\end{prop}

The following theorem uses the identity \eqref{grproduct} and the action of the 
degenerate affine BMW algebra on tensor space to provide a formula for
the Harish-Chandra images of the central elements 
$z_V^{(\ell)} = \epsilon \, (\id\otimes \tr_V)\left( \left(\half y + \gamma\right)^\ell\right)$
in the enveloping algebra $U\fg$ for orthogonal and symplectic Lie algebras $\fg$.
By Theorem \ref{degBMWaction} these particular central elements are natural parameters for
the degenerate affine BMW algebras.  The concept of the proof of Theorem \ref{z1action}
is, at its core, the same as the pattern taken by Nazarov for the proof of \cite[Theorem 3.9]{Naz}.

\begin{thm} 
\label{z1action}
Let $\fg = \fso_{2r+1}$, $\fsp_{2r}$ or $\fso_{2r}$, use notations for $\fh^*$ as in 
\eqref{cltype1}-\eqref{cvaluecl}
and let $h_1,\ldots, h_r$ be the basis of $\fh$ dual to the orthonormal basis
$\varepsilon_1, \ldots, \varepsilon_r$ of $\fh^*$ (so that $h_i=F_{ii}$, where
$F_{ii}$ is as in \eqref{cartan}).  With respect to the form $\langle,\rangle$
in \eqref{favform}, let $\gamma = \sum_b b \otimes b^*$ as in \eqref{gamma_defn}.
Let
$$y
= \begin{cases}
2r, &\hbox{if $\fg = \fso_{2r+1}$,} \\
2r+1, &\hbox{if $\fg = \fsp_{2r}$,} \\
2r-1, &\hbox{if $\fg = \fso_{2r}$,} 
\end{cases}
\qquad
\epsilon\, =
	\begin{cases}
		\!~~~1, &\hbox{if $\fg = \fso_{2r+1}$,} \\
		-1, &\hbox{if $\fg = \fsp_{2r}$,} \\
		\!~~~1, &\hbox{if $\fg = \fso_{2r}$,}
	\end{cases} \qquad
 V=L(\varepsilon_1),
$$
and let $z_V^{(\ell)}$
be the central elements in $U\fg$ defined by 
$$z_V^{(\ell)} = \epsilon \, (\id\otimes \tr_V)\left( \left(\half y + \gamma\right)^\ell\right),
\qquad\hbox{and write}\quad
z_V(u) = \sum_{i\in \ZZ_{\ge0}} z_V^{(\ell)}u^{-\ell}.$$
Then
$$\pi_0(z_V(u)+\epsilon u-\hbox{$\frac12$}) = 
(\epsilon\,u+\half)\frac{(u+\half y - r)}{(u-\half y + r)}
\sigma_{\rho}\left( \prod_{i =1}^r 
\frac{(u+h_i + \half)}
{ (u +h_i - \half)}
\frac{ (u -h_i + \half )}
{(u -h_i - \half)}\right),$$
where $\sigma_\rho$ is the 
algebra automorphism given by $\sigma_\rho(h_i) = h_i +\langle \rho,\varepsilon_i\rangle$ and 
$\pi_0$ is the isomorphism in Theorem \ref{HCIso}.
\end{thm}

\begin{proof} 
In the definition of the action of the degenerate affine BMW algebra in 
Theorem \ref{degBMWaction},
$y_1$ acts on $M\otimes V$ as $\half y + \gamma$,
and
$$\hbox{$e_1y_1^\ell e_1$ acts on $M\otimes V^{\otimes 2}$ as 
$z_V^{(\ell)}e_1$.}
$$
Also
$$
\hbox{$e_1$ and $y_1$ in $\cW_2$ act on $M\otimes V^{\otimes 2}$ 
with $M=L(0)\otimes V^{\otimes (k-1)}$}
$$
in the same way that
$$\hbox{$e_k$ and $y_k$ in $\cW_{k+1}$ act on $M\otimes V^{\otimes(k+1)}$
with $M=L(0)$.}$$
By Proposition \ref{ziprop}, $z_{k-1}^{(\ell)}e_k = e_k y_k^\ell e_k$. Hence, as operators on
$L(0)\otimes V^{\otimes (k-1)}$,
\begin{align}
z_V(u)+\epsilon\,u&-\hbox{$\frac12$}
= z_{k-1}(u)+\epsilon\,u-\hbox{$\frac12$}
\label{opgrprod}
\end{align}
We will use \eqref{grproduct} to compute the action of this operator on the 
$L(\mu)\otimes \cW_{k-1}^\mu$ isotypic component in the
$U\fg\otimes \cW_{k-1}$-module decomposition
\begin{equation}\label{tensordecompA}
L(0)\otimes V^{\otimes (k-1)} \cong \bigoplus_\mu L(\mu)\otimes \cW_{k-1}^\mu.
\end{equation}

As an operator on $L(0)\otimes V$,
$$\gamma
=\hbox{$\frac12$}\big(\langle \varepsilon_1,\varepsilon_1+2\rho\rangle
-\langle \varepsilon_1,\varepsilon_1+2\rho\rangle
+\langle 0,0+2\rho\rangle\big)
=0 \quad \text{by \eqref{gamma_value},
}$$
and so 
$$z_0^{(\ell)} = \epsilon \, (\id \otimes \tr_V)((\half y + \gamma)^\ell) =  \epsilon \, (\id \otimes \tr_V)((\half y )^\ell)
 = \epsilon \, \dim(V)(\half y )^\ell. $$
Therefore, since $\dim(V) = \epsilon+y$,
$$z_0(u) 
= \sum_{\ell\in \ZZ_{\ge0}} z_V^{(\ell)}u^{-\ell}
= \sum_{\ell\in \ZZ_{\ge0}} \epsilon\, \dim(V)(\hbox{$\frac12$}y)^\ell u^{-\ell}
= \epsilon \, \dim(V) \frac{1}{1-\half yu^{-1}}
= \frac{1 + \epsilon\, y}{1-\half yu^{-1}}.
$$
Thus, as an operator on $L(0)\otimes V$,
\begin{align}
z_0(u) + \epsilon \, u - \half  &= \frac{1 + \epsilon \, y}{1 - \half y u^{-1}} + \epsilon u - \half  
=  \frac{(\epsilon\,u+\half)(u + \half y)}{u - \half y }.\label{L0action}
 \end{align}


By the first identity in \eqref{gensAtogensB} and the definition of $\Phi$ in Theorem \ref{degaction}, 
	$$y_k\in \cW_k
\quad\hbox{acts on}\quad L(0)\otimes V^{\otimes k}
=(L(0)\otimes V^{\otimes(k-1)})\otimes V
\quad\hbox{as}\quad
\half y + \gamma.$$
If $L(\mu)$ is an irreducible $U\fg$-module in $L(0) \otimes V^{\otimes (k-1)}$, 
then \eqref{gamma_value}, \eqref{explrho}, and \eqref{cvaluecl} give that $y_k$ acts on the 
$L(\lambda)$ component of 
$L(\mu)\otimes V$ by the constant $c(\lambda,\mu)=0$ when $\lambda=\mu$,
and by the constant
\begin{align}
c(\lambda,\mu) 
&=
\half y 
+ 
\half(\langle \mu\pm \varepsilon_i,\mu\pm \varepsilon_i+2\rho\rangle 
-\langle \mu,\mu+2\rho\rangle 
-\langle \varepsilon_1, \varepsilon_1+2\rho\rangle) 
\nonumber \\
&= \begin{cases}
\half y + c(\lambda/\mu), &\hbox{if $\mu\subseteq \lambda$,} \\
-\half y - c(\mu/\lambda), &\hbox{if $\mu\supseteq \lambda$,} 
\end{cases}
\qquad\hbox{where $\lambda = \mu\pm \varepsilon_i$.}
\label{onebox}
\end{align}
As in \cite[Theorem 2.6]{Naz}, the irreducible $\cW_k$-module  $\cW_k^{\mu/0}=\cW_k^\mu$
has a basis $\{v_T\}$ indexed by \emph{up-down tableaux} $T= (T^{(0)},T^{(1)},
\cdots, T^{(k)})$, where $T^{(0)} = \emptyset$, $T^{(k)} = \mu$, and 
$T^{(i)}$ is a partition obtained from $T^{(i-1)}$ by adding or removing a box
(or, in some cases when $\fg=\fso_{2r+1}$ leaving the partition the same; see
\eqref{tensoreps1})
and
$$
y_iv_T = \begin{cases}
(\half y + c(b)) v_T, &\hbox{if $b = T^{(i)}/T^{(i-1)}$}, \\
(-\half y - c(b)) v_T, &\hbox{if $b = T^{(i-1)}/T^{(i)}$}, \\
0, &\hbox{if $T^{(i-1)}=T^{(i)}$.}
\end{cases}
$$

\noindent Thus the product on the right hand side of \eqref{grproduct}
$$\prod_{i=1}^{k-1} \frac{(u+y_i-1)(u+y_i+1)(u-y_i)^2}{(u+y_i)^2(u-y_i+1)(u-y_i-1)}
\qquad\hbox{acts on $L(\mu)\otimes \cW_{k-1}^\mu$ in 
\eqref{tensordecompA}}
$$
by
\begin{equation}
\label{udtabform}
\prod_{i=1}^{k-1} 
\frac{(u+c(T^{(i)},T^{i-1})-1)(u+c(T^{(i)},T^{(i-1)})+1)(u-c(T^{(i)},T^{(i-1)}))^2}
{ (u + c(T^{(i)},T^{i-1}))^2( u - c(T^{(i)},T^{i-1})+1)  (u-c(T^{(i)},T^{i-1})-1)}
\end{equation}
for any up-down tableau $T$ of length $k$ and shape $\mu$.  If a box is added (or removed) at step $i$ and then removed (or added) at step $j$, then the $i$ and $j$ factors of this product cancel.
Therefore
\eqref{udtabform} is equal to
\begin{equation}\label{allboxesform}
\prod_{b\in \mu} 
\frac{(u+\half y + c(b) -1)(u+\half y + c(b) +1)(u-\half y - c(b) )^2}
{ (u + \half y + c(b) )^2( u - \half y - c(b) +1)  (u-\half y - c(b) -1)}
\end{equation}
(see  \cite[Lemma 3.8]{Naz}).
If $\mu = (\mu_1, \dots , \mu_r)$, simplifying one row at a time,
\begin{align}
\prod_{b\in \mu} 
\frac{(u+\half y + c(b) -1)(u+\half y + c(b) +1)}
{ (u + \half y + c(b) )^2} &= 
\prod_{i =1}^r 
\frac{(u+\half y - i)(u+\half y +\mu_i - i+1)}
{ (u + \half y + 1-i )(u + \half y + \mu_i -i )} \notag \\
&= \frac{u+\half y - r}{u + \half y}\prod_{i =1}^r 
\frac{(u+\half y +\mu_i - i+1)}
{ (u + \half y + \mu_i -i )}, \label{onerowatatime} 
\end{align}
 (see the example following this proof). It follows that \eqref{allboxesform} is equal to 
\begin{align}
 \frac{(u+\half y - r)}{(u + \half y)} &\frac{(u - \half y )}{(u-\half y + r)}
 \prod_{i =1}^r 
\frac{(u+\half y +\mu_i - i+1)}
{ (u + \half y + \mu_i -i )}
\frac{ (u - \half y -( \mu_i -i) )}
{(u-\half y -(\mu_i - i+1))}\notag
\\
&=  \ \frac{(u+\half y - r)}{(u + \half y)} \frac{(u - \half y)}{(u-\half y + r)}
\mathrm{ev}_{\mu+\rho}\left( \prod_{i =1}^r 
\frac{(u+h_i + \half)}
{ (u +h_i - \half)}
\frac{ (u -h_i + \half )}
{(u -h_i - \half)}\right),\label{rightrowbyrowfact}
\end{align}
since 
$
\mathrm{ev}_{\mu+\rho}(h_i)
= \mu_i+\rho_i
= \mu_i + \half (y-2i+1)
= \half y + \half + \mu_i-i.
$

Combining \eqref{L0action} and \eqref{rightrowbyrowfact},
the identity  \eqref{opgrprod} gives that, as operators on
$L(\mu)\otimes \cW_{k-1}^\mu$ in \eqref{tensordecompA},
$$
z_V(u)+\epsilon\,u-\hbox{$\frac12$} = 
(\epsilon\,u+\half)\frac{(u+\half y - r)}{(u-\half y + r)}
\mathrm{ev}_{\mu+\rho}\left( \prod_{i =1}^r 
\frac{(u+h_i + \half)}
{ (u +h_i - \half)}
\frac{ (u -h_i + \half )}
{(u -h_i - \half)}\right).$$
By Theorem \ref{HCIso}, the desired result follows. 
\end{proof}
 
 \begin{example}
 To help illuminate the cancellation done in \eqref{onerowatatime}, let
 $\mu=(5,5,3,3,1,1)$, where the contents of boxes are 
\begin{equation}
\beginpicture
\setcoordinatesystem units <0.5cm,0.5cm>         
\setplotarea x from 0 to 4, y from -3 to 3    
\linethickness=0.5pt                          
\putrule from 0 3 to 5 3          %
\putrule from 0 2 to 5 2          
\putrule from 0 1 to 5 1          %
\putrule from 0 0 to 3 0          %
\putrule from 0 -1 to 3 -1          %
\putrule from 0 -2 to 1 -2          %
\putrule from 0 -3 to 1 -3          %
\putrule from 0 -3 to 0 3        %
\putrule from 1 -3 to 1 3        %
\putrule from 2 -1 to 2 3        %
\putrule from 3 -1 to 3 3        
\putrule from 4 1 to 4 3        %
\putrule from 5 1 to 5 3        %
\put{$\scriptstyle{0}$} at .5 2.5 
\put{$\scriptstyle{1}$} at 1.5 2.5 
\put{$\scriptstyle{2}$} at 2.5 2.5 
\put{$\scriptstyle{3}$} at 3.5 2.5 
\put{$\scriptstyle{4}$} at 4.5 2.5 
\put{$\scriptstyle{-1}$} at .5 1.5 
\put{$\scriptstyle{0}$} at 1.5 1.5 
\put{$\scriptstyle{1}$} at 2.5 1.5 
\put{$\scriptstyle{2}$} at 3.5 1.5 
\put{$\scriptstyle{3}$} at 4.5 1.5 
\put{$\scriptstyle{-2}$} at .5 .5 
\put{$\scriptstyle{-1}$} at 1.5 .5 
\put{$\scriptstyle{0}$} at 2.5 .5 
\put{$\scriptstyle{-3}$} at .5 -.5 
\put{$\scriptstyle{-2}$} at 1.5 -.5 
\put{$\scriptstyle{-1}$} at 2.5 -.5 
\put{$\scriptstyle{-4}$} at .5 -1.5 
\put{$\scriptstyle{-5}$} at .5 -2.5 
\endpicture.
\end{equation}
In this example, the product over the boxes in the first row of the diagram is
\begin{align*}
\prod_{b\mathrm{\ in\ first\ row\ of\ }\mu}&
\frac{(x+c(b)-1)(x+c(b)+1)}{(x+c(b))(x+c(b))} \\
&=
\frac{(x-1)(x+1)}{(x+0)(x+0)}
\frac{(x+0)(x+2)}{(x+1)(x +1)}
\frac{(x+1)(x+3)}{(x+2)(x+2)}
\frac{(x+2)(x+4)}{(x+3)(x+3)}
\frac{(x+3)(x+5)}{(x+4)(x+4)} \\
&=
\frac{(x-1)}{(x+0)}
\frac{(x+5)}{(x+4)},
\qquad\hbox{where $x = u+\frac12 y$.}
\end{align*}
Thus, simplifying the product one row at a time, 
\begin{align*}
\prod_{b\in\mu}&
\frac{(x+c(b)-1)(x+c(b)+1)}{(x+c(b))(x+c(b))} \\
&=
\frac{(x-1)(x+5)}{(x+0)(x+4)}
\frac{(x-2)(x+4)}{(x-1)(x+3)}
\frac{(x-3)(x+1)}{(x-2)(x+0)}
\frac{(x-4)(x+0)}{(x-3)(x-1)}
\frac{(x-5)(x-3)}{(x-4)(x-4)}
\frac{(x-6)(x-4)}{(x-5)(x-5)}\\
&=
\frac{(x-6)}{(x+0)}
\frac{(x+5)}{(x+4)}
\cdot
\frac{(x+4)}{(x+3)}
\cdot
\frac{(x+1)}{(x+0)}
\cdot
\frac{(x+0)}{(x-1)}
\cdot
\frac{(x-3)}{(x-4)}
\cdot
\frac{(x-4)}{(x-5)} \\
\end{align*}
leads to the identity
\begin{align*}
\prod_{b\in\mu}
\frac{(x+c(b)-1)(x+c(b)+1)}{(x+c(b))(x+c(b))} 
= \frac{x-r}{x+0} \prod_{i=1}^r \frac{x+\mu_i-i+1}{x+\mu_i-i},
\qquad \text{where $\mu = (\mu_1,\ldots, \mu_r)$.}
\end{align*}
\end{example}


\subsection{Central elements $Z_V^{(\ell)}$}
\label{sec:BB-to-HC}

Let $Z_0^{(\ell)}$, $z$ and $q$ be the parameters of the affine BMW algebra $W_k$. Let $u$ be a variable and define $Z_i^{(\ell)}, Z_i^{(-\ell)} \in W_k$ for $i = 1, \dots, k-1$ by 
\begin{align}
Z^+_i(u) + &\frac{z^{-1}}{q - q^{-1}} - \frac{u^2}{u^2 - 1}  \nonumber \\
&= \left(Z_0^+ +  \frac{z^{-1}}{q - q^{-1}} - \frac{u^2}{u^2 - 1}\right)
\prod_{j=1}^{i}\frac{(u-Y_j)^2(u-q^{-2}Y_j^{-1})(u-q^2 Y_j^{-1})}
{(u-Y_j^{-1})^2(u - q^2 Y_j)(u-q^{-2}Y_j)}
, \label{Z+product}\\
Z^-_i(u) - &\frac{z}{q - q^{-1}} - \frac{u^2}{u^2 - 1} \nonumber \\
&= \left(Z_0^- -  \frac{z}{q - q^{-1}} - \frac{u^2}{u^2 - 1}\right)
\prod_{j=1}^{i}\frac{(u-Y_j^{-1})^2(u - q^2 Y_j)(u-q^{-2}Y_j)}
{(u-Y_j)^2(u-q^{-2}Y_j^{-1})(u-q^2 Y_j^{-1})} \label{Z-product}
\end{align}
where
$$Z_i^+(u) = \sum_{\ell \in \ZZ_{\geq 0}} Z_i^{(\ell)} u^{-\ell}
\quad \text{ and } \quad
Z_i^-(u) = \sum_{\ell \in \ZZ_{\geq 0}} Z_i^{(-\ell)} u^{-\ell} \quad \qquad \text{ for } i=0,\dots, k-1.$$ 
The following proposition is equivalent to  \cite[Lemma 7.4]{BB} and is also proved in \cite[Theorem 3.6 and Remark 3.8]{DRV}.

\begin{prop}  \label{Ziprop}
In the affine BMW algebra $W_k$,
$$E_{i+1}Y_i^\ell E_{i+1} = Z_i^{(\ell)}E_{i+1},
\qquad\hbox{for $i=0,1,\ldots, k-2$ and $\ell\in \ZZ$.}
$$
\end{prop}

The following theorem uses the identity \eqref{Z+product} and the action of the 
affine BMW algebra on tensor space to provide a formula for
the Harish-Chandra images of the central elements 
$Z_V^{(\ell)} = \epsilon (\id\otimes \qtr_V)\left( (z\cR_{21}\cR)^\ell)\right)$
in the Drinfeld-Jimbo quantum group $U_h\fg$ for orthogonal and symplectic Lie algebras $\fg$.
By Theorem \ref{BMWaction} these central elements are natural parameters for
the affine BMW algebras.  

\begin{thm} \label{Z1action}  Let $U=U_h\fg$ be the Drinfeld-Jimbo quantum
group corresponding to $\fg = \fso_{2r+1}$, $\fsp_{2r}$ or $\fso_{2r}$ and
use notations for $\fh^*$ as in 
\eqref{cltype1}-\eqref{cvaluecl}.
Identify $U_0$ as a subalgebra of $\CC[L_1^{\pm1},\ldots, L_r^{\pm1}]$
where $\mathrm{ev}_{\varepsilon_i}(L_j) = q^{\langle \varepsilon_i, \varepsilon_j\rangle} = q^{\delta_{ij}}$ (so that $L_i=e^{\half hF_{ii}}$, where
$F_{ii}$ is as in \eqref{cartan}).
Let
$$y
= \begin{cases}
2r, &\hbox{if $\fg = \fso_{2r+1}$,} \\
2r+1, &\hbox{if $\fg = \fsp_{2r}$,} \\
2r-1, &\hbox{if $\fg = \fso_{2r}$,} 
\end{cases}
\qquad
\epsilon\, =
	\begin{cases}
		\!~~~1, &\hbox{if $\fg = \fso_{2r+1}$,} \\
		-1, &\hbox{if $\fg = \fsp_{2r}$,} \\
		\!~~~1, &\hbox{if $\fg = \fso_{2r}$,}
	\end{cases} \qquad
 V=L(\varepsilon_1),
$$
and $z = \epsilon q^y$. Let $Z_V^{(\ell)}$
be the central elements in $U_h\fg$ defined by 
$$Z_V^{(\ell)} = \epsilon (\id\otimes \qtr_V)\left( (z\cR_{21}\cR)^\ell)\right)
$$
and write
$$Z_V^+(u) = \sum_{\ell\in \ZZ_{\ge0}} Z_V^{(\ell)}u^{-\ell}
\qquad\hbox{and}\qquad
Z_V^-(u) = \sum_{\ell\in \ZZ_{\ge0}} Z_V^{(-\ell)}u^{-\ell}.
$$
Then
\begin{align*}
\pi_0&\left(Z_V^+(u) + \frac{z^{-1}}{q-q^{-1}} - \frac{u^2}{u^2-1}\right) \\
&= \left(\frac{z}{q-q^{-1}}\right) \frac{(u+q)(u-q^{-1})}{(u+1)(u-1)}
\frac{(u-\epsilon\,q^{2r - y})}{(u-\epsilon\,q^{y-2r})} \ 
\sigma_{\rho}\left(
\prod_{i=1}^r
\frac{(u-\epsilon\, L_i^{-2}q^{-1}) (u-\epsilon\, L_i^2 q^{-1})}
{(u-\epsilon L_i^{-2}q )(u-\epsilon L_i^2 q)}
\right)
\end{align*}
and
\begin{align*}
\pi_0&\left(Z_V^-(u) - \frac{z}{q-q^{-1}} - \frac{u^2}{u^2-1}\right) \\
&=
-\frac{z^{-1}}{(q-q^{-1})}\frac{(u - q)(u + q^{-1})}{(u+1)(u-1)}
\frac{(u-\epsilon\,q^{y-2r})}{(u-\epsilon\,q^{2r-y})}
\sigma_\rho\left(
\prod_{i=1}^r 
\frac{(u-\epsilon L_i^{-2}q )(u-\epsilon L_i^2 q)}
{(u-\epsilon\, L_i^{-2}q^{-1}) (u-\epsilon\, L_i^2 q^{-1})}
\right),
\end{align*}
where $\sigma_\rho$  is the 
algebra automorphism given by $\sigma_\rho(L_i) = q^{\langle \rho, \varepsilon_i\rangle}L_i$ and 
$\pi_0$ is the isomorphism in Theorem \ref{HCIso}.\end{thm}

\begin{proof}
In the definition of the action of the affine BMW algebra in Theorem \ref{affaction}, $Y_1$ acts on $M \otimes V$ as $z \cR_{21} \cR$, and \\
\smallskip
\centerline{$E_1 Y_1^\ell E_1$ acts on  $M \otimes V^{\otimes 2}$ as $Z_V^{(\ell)} E_1$.}
\smallskip
Also\\
\smallskip
\centerline{$E_1$ and $Y_1$ in $W_2$ act on $M \otimes V^{\otimes 2}$ with $M = L(0) \otimes V^{\otimes (k-1)}$}
\smallskip
in the same way that \\
\smallskip
\centerline{$E_k$ and $Y_k$ in $W_{k+1}$ act on $M \otimes V^{\otimes (k+1)}$ with $M = L(0)$.}
\smallskip

By Proposition \ref{Ziprop}, $Z_{k-1}^{(\ell)}E_k = E_k Y_k^\ell E_k$ and so it follows that, as operators on $L(0) \otimes V^{\otimes (k-1)}$, 
\begin{align}
Z_V^+(u) + &\frac{z^{-1}}{q-q^{-1}} - \frac{u^2}{u^2-1} 
= Z_{k-1}^+(u) + \frac{z^{-1}}{q-q^{-1}} - \frac{u^2}{u^2-1} 
\end{align}
and 
\begin{align}
Z_V^-(u) - &\frac{z}{q-q^{-1}} - \frac{u^2}{u^2-1}
= Z_{k-1}^-(u) - \frac{z}{q-q^{-1}} - \frac{u^2}{u^2-1} 
\end{align}
We will use  \eqref{Z+product} and  \eqref{Z-product} to compute the action of these
operators on the $L(\mu) \otimes W_{k-1}^\mu$ isotypic component in the
$U_h\fg\otimes W_{k-1}$-module decomposition
\begin{equation}\label{tensordecomp}
L(0)\otimes V^{\otimes (k-1)} \cong \bigoplus_\mu L(\mu)\otimes W_{k-1}^\mu.
\end{equation}

As an operator on $L(0)\otimes V$,
$z(\cR_{21}\cR) 
=z
q^{\< \varepsilon_1,\varepsilon_1+2\rho\>
-\<\varepsilon_1,\varepsilon_1+2\rho\>
+\< 0,0+2\rho\>}
=z$.
Hence
$$Z_V^{(\ell)} = \epsilon\, (\id\otimes\qtr_V)((z\cR_{21}\cR)^\ell) = z^\ell \epsilon\, \dim_q(V).
$$
Therefore, since $\displaystyle{\epsilon\, \dim_q(V) = \frac{z-z^{-1}}{q-q^{-1}}+1}$,
\begin{align*}
Z^+_V(u) 
= \sum_{\ell\in \ZZ_{\ge0}}  \epsilon\, \dim_q(V)z^\ell u^{-\ell} 
&= \epsilon\, \dim_q(V) \frac{1}{1-zu^{-1}}
=\frac{z-z^{-1}+(q-q^{-1})}{(q-q^{-1})(1-zu^{-1})}.
\end{align*}
A similar computation of $Z_V^-$ yields
$$
Z^-_V(u) 
=\frac{z-z^{-1}+q-q^{-1}}{(q-q^{-1})(1-z^{-1}u^{-1})}.
$$
Thus,  as operators on $L(0)\otimes V$, 
\begin{equation}\label{L0action+}
Z_V^+ + \frac{z^{-1}}{q-q^{-1}} - \frac{u^2}{u^2-1}
= \frac{z}{(q-q^{-1})}\frac{(1-z^{-1}u^{-1})}{(1-zu^{-1})}\frac{(u+q)(u-q^{-1})}{(u+1)(u-1)}
\end{equation}
and
\begin{equation}\label{L0action-}
Z_V^{-} - \frac{z}{q-q^{-1}} - \frac{u^2}{u^2-1}
= \frac{-z^{-1}}{(q-q^{-1})}\frac{(1-zu^{-1})}{(1-z^{-1}u^{-1})}\frac{(u-q)(u+q^{-1})}{(u+1)(u-1)}.
\end{equation}

By \eqref{XiasRmatrix} and the definition of $\Phi$ in Theorem \ref{affaction},
$$Y_k\in W_k
\quad\hbox{acts on}\quad L(0)\otimes V^{\otimes k}
=(L(0)\otimes V^{\otimes(k-1)})\otimes V
\quad\hbox{as}\quad
z\cR_{21}\cR.$$
If $L(\mu)$ is an irreducible $U_h\fg$-module in $L(0) \otimes V^{\otimes (k-1)}$, 
then \eqref{fulltwist} and \eqref{cvaluecl} give that $Y_k$ acts on the 
$L(\lambda)$ component of 
$L(\mu)\otimes V$ by the constant $\epsilon q^{2c(\lambda,\mu)}=1\cdot q^0=1$, when
$\fg=\fso_{2r+1}$ and $\lambda=\mu$, and by the constant
$$
\epsilon q^{2c(\lambda,\mu)}
= \begin{cases}
\epsilon q^{y + 2c(\lambda/\mu)}, &\hbox{if $\mu\subseteq \lambda$,} \\
\epsilon q^{- y -2c(\mu/\lambda)}, &\hbox{if $\mu\supseteq \lambda$,} \\
\end{cases}
= \begin{cases}
zq^{2c(\lambda/\mu)}, &\hbox{if $\mu\subseteq \lambda$,} \\
z^{-1}q^{-2c(\mu/\lambda)}, &\hbox{if $\mu\supseteq \lambda$,} \\
\end{cases},
\qquad\hbox{where $\lambda=\mu\pm\varepsilon_i$.}
$$
and $c(\lambda,\mu)$ is as computed in \eqref{onebox}.
As in \cite[Theorem 6.3(b)]{OR},
 the irreducible $W_k$-module $W_k^{\mu/0}=W_k^\mu$
 has a basis $\{v_T\}$ indexed by \emph{up-down tableaux} $T= (T^{(0)},T^{(1)},
 \cdots, T^{(k)})$, where $T^{(0)} = \emptyset$, $T^{(k)} = \mu$, and 
 $T^{(i)}$ is a partition obtained from $T^{(i-1)}$ by adding or removing a box,
 and
$$Y_iv_T = \begin{cases}
zq^{2c(b)} v_T, &\hbox{if $b = T^{(i)}/T^{(i-1)}$}, \\
z^{-1}q^{-2c(b)} v_T, &\hbox{if $b = T^{(i-1)}/T^{(i)}$},\\
v_T,& \hbox{if $T^{(i-1)}=T^{(i)}$.}
\end{cases}
$$
Thus
\begin{align*}
\prod_{i=1}^{k-1}
\frac
{(u-Y_i)^2(u-q^{-2}Y_i^{-1})(u-q^2Y_i^{-1})}
{(u-Y_i^{-1})^2(u-q^2Y_i)(u-q^{-2}Y_i)}
\qquad\hbox{acts on $L(\mu)\otimes W_{k-1}^\mu$ in \eqref{tensordecomp}}
\end{align*}
by
\begin{equation}\label{UDTABform}
\prod_{i=1}^{k-1} 
\frac{(u-\epsilon\, q^{2c(T^{(i)},T^{(i-1)})})^2(u-\epsilon\, q^{-2}q^{-2c(T^{(i)},T^{(i-1)})})(u-\epsilon\, q^2q^{-2c(T^{(i)},T^{(i-1)})})
}
{(u - \epsilon\, q^{-2c(T^{(i)},T^{i-1})})^2( u - \epsilon\, q^2q^{2c(T^{(i)},T^{i-1})})(u-\epsilon\, q^{-2}q^{2c(T^{(i)},T^{i-1})}) 
 }
\end{equation}
for any up-down tableau $T$ of length $k$ and shape $\mu$. 
If a box is added (or removed) at step $i$ and then removed (or added) at step $j$, then the $i$ and $j$ factors of this product cancel. Therefore
\eqref{UDTABform} is equal to
\begin{equation}\label{AllBoxesform}
\prod_{b\in \mu} 
\frac{(u-zq^{2c(b)})^2(u-z^{-1}q^{-2(c(b)+1)})(u-z^{-1}q^{-2(c(b) -1)})}
{(u -z^{-1}q^{-2c(b)})^2( u - zq^{2(c(b)+1)}) (u-zq^{2(c(b)-1)})  }.
\end{equation}
Simplifying one row at a time,
\begin{align*}
\prod_{b\in\mu}
\frac{(u-z^{-1}q^{-2(c(b)-1)})(u-z^{-1}q^{-2(c(b)+1)})}{(u-z^{-1} q^{-2c(b)})(u-z^{-1}q^{-2c(b)})}
&= \prod_{i=1}^r \frac{(u-z^{-1}q^{-2(-i)})(u-z^{-1}q^{-2(\mu_i-i+1)})}{(u-z^{-1}q^{-2(-(i-1))})(u-z^{-1}q^{-2(\mu_i-i)})}\\
&= \frac{u - z^{-1}q^{2r}}{u- z^{-1}q^{2\cdot 0}} \prod_{i=1}^r  \frac{u-z^{-1}q^{-2(\mu_i-i+1)}}{u-z^{-1}q^{-2(\mu_i-i)}}\
\end{align*}
if $\mu = (\mu_1, \dots , \mu_r)$. It follows that 
 \eqref{AllBoxesform} is equal to
\begin{align}
\frac{(u-z^{-1}q^{2r})}{(u-z^{-1})}
&\frac{(u-z)}{(u-zq^{-2r})}
\prod_{i=1}^r \frac{(u-z^{-1}q^{-2(\mu_i-i+1)})}
{(u-z^{-1}q^{-2(\mu_i-i)})}
\cdot \frac{(u-zq^{2(\mu_i-i)})}
{(u-zq^{2(\mu_i-i+1)})}
\nonumber \\
&=\frac{(u-\epsilon \, q^{-(y-2r)})}{(u-z^{-1})}
\frac{(u-z)}{(u-\epsilon \, q^{y-2r})}
\mathrm{ev}_{\mu+\rho}\left(
\prod_{i=1}^r 
\frac{(u-\epsilon\, L_i^{-2}q^{-1}) (u-\epsilon\, L_i^2 q^{-1})}
{(u-\epsilon L_i^{-2}q )(u-\epsilon L_i^2 q)}
\right), \label{RightRowbyRowFact}
\end{align}
since $z^{-1}q^{2r} = \epsilon q^{-y} q^{2r} = \epsilon q^{2r-y}$ and
$$
\mathrm{ev}_{\mu+\rho}(L_i^2)
=q^{\langle \mu+\rho, 2\varepsilon_i\rangle}
= q^{2\mu_i + (y-2i+1)}
= q^{y + 1+ 2(\mu_i-i)}=\epsilon zq^{2(\mu_i-i)+1}.
$$
Combining \eqref{L0action+} and \eqref{RightRowbyRowFact},
the identity  \eqref{Z+product} gives that, as operators on $L(\mu)\otimes W_{k-1}^\mu$
in \eqref{tensordecomp},
\begin{align}
&Z_k^+ +\frac{z^{-1}}{q-q^{-1}}-\frac{u^2}{u^2-1}  \\
&=
\frac{z}{(q-q^{-1})}\frac{(u+q)(u-q^{-1})}{(u+1)(u-1)}
\frac{(u-\epsilon\,q^{2r-y})}{(u-\epsilon\,q^{y-2r})}
\mathrm{ev}_{\mu+\rho}\left(
\prod_{i=1}^r 
\frac{(u-\epsilon\, L_i^{-2}q^{-1}) (u-\epsilon\, L_i^2 q^{-1})}
{(u-\epsilon L_i^{-2}q )(u-\epsilon L_i^2 q)}
\right).
\label{Z+RHS}\notag
\end{align}
Similarly,
$Z_k^- - \frac{z}{q-q^{-1}} + \frac{1}{u^2-1}$ 
acts on the 
$L(\mu)\otimes W_{k-1}^\mu$ isotypic component in the
$U_h\fg\otimes W_{k-1}$-module decomposition in \eqref{tensordecomp}
by
\begin{align*} 
&-\frac{z^{-1}}{(q-q^{-1})}\frac{(u - q)(u + q^{-1})}{(u+1)(u-1)}
\frac{(u-\epsilon\,q^{2r-y})}{(u-\epsilon\,q^{y-2r})}
\mathrm{ev}_{\mu+\rho}\left(
\prod_{i=1}^r 
\frac{(u-\epsilon L_i^{-2}q )(u-\epsilon L_i^2 q)}
{(u-\epsilon\, L_i^{-2}q^{-1}) (u-\epsilon\, L_i^2 q^{-1})}
\right).
\end{align*}
By Theorem \ref{HCIso}, the desired results follow. 

\end{proof}

In the following corollary, we shall repackage Theorem \ref{Z1action} to give a formula
for the Harish-Chandra image of $Z_V^{(\ell)}$ in terms of ``Weyl characters''.  To do this
we will use the universal characters of \cite{KT} following the notation in 
\cite[\S 6]{HR}.  For a formal alphabet $Y$ let 
$sa_\lambda(Y)$ be the universal Weyl character for $\fgl_r$, 
$sp_\lambda(Y)$ the universal Weyl character for $\fsp_{2r}$, and 
$so_\lambda(Y)$ the universal Weyl character for the orthogonal cases.

The Cauchy-Littlewood identities (see \cite[Lemma 1.5.1]{KT},
\cite[Theorems 7.8FG and 7.9C]{We}, and \cite[(6.4) and (6.5)]{HR}) are
$$\prod_{i,j}\frac{1}{1-x_iy_j}
=\Omega(XY) = \sum_\lambda sa_\lambda(X)sa_\lambda(Y),$$
$$\prod_{i\le j}\frac{1}{1-y_iy_j}\prod_{i,j}\frac{1}{1-x_iy_j}
=\Omega(XY - sa_{(2)}(Y)) = \sum_\lambda sa_\lambda({Y})so_\lambda({X}),$$
$$\prod_{i< j}\frac{1}{1-y_iy_j}\prod_{i,j}\frac{1}{1-x_iy_j}
=\Omega(XY - sa_{(1^2)}(Y)) = \sum_\lambda sa_\lambda({Y})sp_\lambda({X}),$$
where $\Omega$ is the Cauchy kernel (see \cite[(6.3)]{HR}) and the first equality in each line is for the 
formal alphabets $X=\sum_i x_i$ and $Y = \sum_j y_j$.
The identity \cite[Lemma 6.7(a)]{HR} states
\begin{equation}\label{HRidentity}
sa_\lambda ((q-q^{-1})u^{-1}) = \begin{cases}
(q-q^{-1})u^{-\ell}(-q^{-1})^{\ell-m}q^{m-1}, & \text{ if } \lambda = (m, 1^{\ell-m}),\\
0, & \text{ otherwise.}
\end{cases}
\end{equation}

\begin{cor}\label{cor:ZactionToBaumann}
In the same setting as in Theorem \ref{Z1action}, let
$$y
= \begin{cases}
2r, &\hbox{if $\fg = \fso_{2r+1}$,} \\
2r+1, &\hbox{if $\fg = \fsp_{2r}$,} \\
2r-1, &\hbox{if $\fg = \fso_{2r}$,} 
\end{cases}
\qquad
\epsilon\, =
	\begin{cases}
		\!~~~1, &\hbox{if $\fg = \fso_{2r+1}$,} \\
		-1, &\hbox{if $\fg = \fsp_{2r}$,} \\
		\!~~~1, &\hbox{if $\fg = \fso_{2r}$,}
	\end{cases} \qquad
 V=L(\varepsilon_1),
$$
$z=\epsilon q^y$, and let $Z_V^{(\ell)}$ be the 
central elements in the Drinfeld-Jimbo quantum group $U_h\fg$ which are 
given by $Z_V^{(\ell)}= 
\epsilon (\id\otimes \qtr_{L(\varepsilon_1)})((z\cR_{21}\cR)^\ell)$.
Let $X$
be the formal alphabet given by
$X= \sum_{i \in \hat V} L_i^2$ and fix $c=1$ if $\ell$ is even and $c=0$ if $\ell$ is odd.
%
%
Then for $\ell \geq 1$,
$$\pi_0(Z^{(\ell)}_V) = \sigma_\rho\left( c+  z\epsilon^{\ell}
\sum_{m=1}^\ell(q-q^{-1})(-1)^{\ell-m}q^{-(\ell-2m+1)}s_{(m,1^{\ell-m})}(X)
\right)
$$
where $s_{(m,1^{\ell-m})}(X) = so_{(m,1^{\ell-m})}(X)$ in the orthogonal cases and
$s_{(m,1^{\ell-m})}(X) = sp_{(m,1^{\ell-m})}(X)$ in the symplectic case.  
\end{cor}
\begin{proof} Let $\hat V$ as in \eqref{Vhat}, $L_{-i}=L_i^{-1}$ where $L_i$ is as in the
statement of Theorem \ref{Z1action}, and let $L_{\varepsilon_0} = 1$.
The identity in Theorem \ref{Z1action} can be rewritten as
\begin{align}
\pi_0\left( Z_V^+(u) + \frac{z^{-1}}{q-q^{-1}} - \frac{u^2}{u^2-1} \right)
       &= \sigma_\rho\left(
               \frac{z}{q-q^{-1}} \frac{u^2-q^{2\epsilon}}{u^2-1}
               \prod_{j\in \hat V}
                       \frac{(1- L_{\wt(v_j)}^2q^{-1}(\epsilon u)^{-1})}
                       {(1- L_{\wt(v_j)}^2q(\epsilon u)^{-1})}\right). \label{Z1actionRe}
\end{align}
By \eqref{HRidentity},
$$sa_{(2)}((q-q^{-1})(\epsilon u)^{-1}) 
=(q^2-1)(\epsilon u)^{-2}
\quad\hbox{and}\quad
sa_{(1^2)}((q-q^{-1})(\epsilon u)^{-1}) 
= (q^{-2}-1)(\epsilon u)^{-2}.
$$
So the Cauchy-Littlewood identities give
\begin{align*}
&\frac{(1-q^2(\epsilon u)^{-2})}{1-(\epsilon u)^{-2}} \prod_{i \in \hat V}
\frac{(1-L_i^{2}q^{-1}(\epsilon u)^{-1})}{(1- L_i^{2}q(\epsilon u)^{-1})} \\
&\qquad=\Omega \left(X(q- q^{-1})(\epsilon u)^{-1} - sa_{(2)}((q-q^{-1})(\epsilon u)^{-1})\right)
=\sum_\lambda sa_\lambda((q- q^{-1})(\epsilon u)^{-1})so_\lambda(X) \\
&\qquad= \sum_{\ell \in \ZZ_{\geq 0}} \left(
\sum_{m=1}^\ell (q- q^{-1})(-q^{-1})^{\ell-m}q^{m-1}so_{(m,1^{\ell-m})}(X)
\right)(\epsilon u)^{-\ell}\\
&\qquad= \sum_{\ell \in \ZZ_{\geq 0}} \epsilon^{\ell}(q- q^{-1})\left(
\sum_{m=1}^\ell(q-q^{-1})(-1)^{\ell-m}q^{-(\ell-2m+1)}so_{(m,1^{\ell-m})}(X)
\right) u^{-\ell}
\end{align*}
in the orthogonal case, and 
\begin{align*}
&\frac{(1-q^{-2}(\epsilon u)^{-2})}{1-(\epsilon u)^{-2}} \prod_{i \in \hat V}
\frac{(1-L_i^{2}q^{-1}(\epsilon u)^{-1})}{(1- L_i^{2}q(\epsilon u)^{-1})} \\
&\qquad =\Omega \left(X(q- q^{-1})({\epsilon}u)^{-1} - sa_{(1^2)}((q-q^{-1})({\epsilon}u)^{-1})\right)
=\sum_\lambda sa_\lambda((q- q^{-1})(\epsilon u)^{-1})sp_\lambda(X) \\
&\qquad = \sum_{\ell \in \ZZ_{\geq 0}} \left(
\sum_{m=1}^\ell (q- q^{-1})(-q^{-1})^{\ell-m}q^{m-1}sp_{(m,1^{\ell-m})}(X)
\right)(\epsilon u)^{-\ell}
\\
&\qquad = \sum_{\ell \in \ZZ_{\geq 0}}  \epsilon^{\ell}(q- q^{-1})\left(
\sum_{m=1}^\ell (-1)^{\ell-m}q^{-(\ell-2m+1)}sp_{(m,1^{\ell-m})}(X)
\right)u^{-\ell}
\end{align*}
in the symplectic case.  The statement now follows by noting that 
$u^2/(u^2-1)=1/(1-u^{-2})=\sum_{k\in \ZZ_{\ge 0}} u^{-2k}$ and taking the coefficient of 
$u^{-\ell}$ on each side of \eqref{Z1actionRe}.
\end{proof}

\section{Symplectic and orthogonal higher Casimir elements}
\label{sec:connections}
Our final goal in this paper will be to connect the central elements appearing naturally 
as parameters of the affine and degenerate affine BMW algebras (see Theorems \ref{degBMWaction} and \ref{BMWaction}) 
to higher Casimir elements for orthogonal and symplectic Lie 
algebras and quantum groups. In the degenerate case, we explain how the generating function
for $z_V^{(\ell)}$ derived in Theorem \ref{z1action} can be matched up with the generating functions for
central elements given by Perelomov-Popov in \cite{PP1, PP2}. Expositions of the Perelomov-Popov results are also  in \cite[\S 7.1]{Mo} and \cite[\S 127]{Zh}.
In the affine case we show how the formula for $Z_V^{(\ell)}$ in Corollary \ref{cor:ZactionToBaumann}
can be derived as a special case of a remarkable identity for central elements in quantum groups
discovered by Baumann \cite[Thm.\ 1]{Bau}. 

\subsection{The central elements $z_V^{(\ell)}$ as higher Casimir elements}
\label{sec:Nazarov-to-PP}

Returning to the notation developed in the preliminaries of Section \ref{sec:classical-actions}, let $\fg = \fgl_r$ with nondegenerate $\ad$-invariant form $\<,\>$ as in \eqref{glform}  and operator
$\gamma = \gamma^{\fgl}$ as in \eqref{glgamma}. Then

\begin{align*}
(\id\otimes \tr_V)(\gamma^\ell) 
&= \sum_{i_1,i_2,\ldots, i_\ell} E_{i_1i_2}E_{i_2i_3}\cdots E_{i_\ell i_1} 
\end{align*}
are the central elements of $U\fgl_n$ found, for example, in Gelfand 
\cite[(3)]{Ge}.
Perelomov-Popov \cite{PP1,PP2} generalized this construction to
$\fg=\fso_{2r+1}, \fsp_{2r}$, and $\fso_{2r}$, 
by letting $F_{ij}$ be the natural spanning set
for $\fg$ given in \eqref{natspanningset}, viewing $F = (F_{ij})_{i,j\in \hat V}$ 
as a matrix with entries in $\fg$, and writing
\begin{equation}\label{PPelts}
\tr F^k = \sum_{i_1,i_2,\ldots, i_k \in \hat V} F_{i_1i_2}F_{i_2i_3}\cdots F_{i_ki_1}
\qquad\hbox{for $k\in \ZZ_{>0}$},
\end{equation}
as an element of the enveloping algebra $U\fg$ (see \cite[Thm.\ 7.1.7]{Mo}). These elements are central in $U\fg$ and Perelomov-Popov gave the following generating function formula for their Harish-Chandra images (see \cite[\S127]{Zh}).  The proof we give below shows that the result of Perelemov-Popov is equivalent to Theorem \ref{z1action} (which we obtained from the degenerate affine BMW algebra and Schur-Weyl duality). A proof of Theorem \ref{PP} using the theory of twisted Yangians is given in \cite[\S 7.1]{Mo}.  

\begin{thm}\label{PP} (Perelomov-Popov) \cite[Cor.\ 7.1.8]{Mo}  Let $\fg = \fso_{2r+1}$ or $\fsp_{2r}$ or $\fso_{2r}$, use notations for $\fh^*$ as in 
Section \ref{sec:classical-actions} and let $h_1,\ldots, h_r$ be the basis of $\fh$ dual to the orthonormal basis
$\varepsilon_1, \ldots, \varepsilon_r$ of $\fh^*$.
Let $\rho'_r = \half - \half y$, $l_0=0$ in the case that $\fg=\fso_{2r+1}$ and let
$$l_i = - l_{-i} = h_i + \rho_i,
\ \ \hbox{for $i=1,2,\ldots, r$,}
\qquad\hbox{where $\rho_i = \half(y-2i+1)$.}
$$
Then
\begin{align}\label{PPformula}
\pi_0\left(1+\frac{x+\half}{x+\half-\half\epsilon }
\Big(
\sum_{\ell\in \ZZ_{\ge 0}} \frac{(-1)^\ell \tr(F^k)}{(x+\rho'_r)^{\ell+1}}
\Big)\right) 
= \prod_{i\in \hat V} \frac{x+l_i+1}{x+l_i}. 
\end{align}
\end{thm}
\begin{proof}
By \eqref{sospgamma},
$$\gamma = \half \sum_{i,j \in \hat V} F_{ij}\otimes F_{ji}.$$
Let $\eta\colon \fg \to \End(V)$ be the defining representation. Since $F_{-j,-i} = -\theta_{ij}F_{ij}$,
\begin{align*}(\id\otimes \eta)(\gamma)
&= \half\sum_{i,j \in \hat V} F_{ij}\otimes (E_{ji} - \theta_{ji}E_{-i,-j}) 
= \half\sum_{i,j \in \hat V}( F_{ij}\otimes E_{ji} - \theta_{ji}F_{-j,-i}\otimes E_{ji})\\
&=  \half\sum_{i,j \in \hat V}(F_{ij} - \theta_{ji}F_{-j,-i})\otimes E_{ji} 
=  \half\sum_{i,j \in \hat V}(F_{ij} +F_{ij})\otimes E_{ji} \\
&= \sum_{i,j \in \hat V}F_{ij}\otimes E_{ji} = F^t=-\theta F \theta,
\qquad\hbox{where}\quad \theta= \begin{pmatrix} \epsilon\, \id & 0 \\ 0 & \id \end{pmatrix}.
\end{align*}
Thus
\begin{equation}\label{gammaktoFk}
(\id\otimes \tr_V)(\gamma^k) =\tr( (F^t)^k )
=\tr((-\theta F \theta)^k) = (-1)^k \tr(\theta^2 F^k) = (-1)^k \tr(F^k),
\end{equation}
which provides the connection of the elements $z_0^{(\ell)}$ appearing in Theorems \ref{degBMWaction} and
\ref{z1action} to the elements in \eqref{PPelts}.

In order to transform the generating function for the elements $(\id\otimes \tr_V)(\gamma^\ell)$ into the generating function for the elements $(\id\otimes \tr_V)((\half y + \gamma)^\ell)$, notice
\begin{align*}
\sum_{\ell\in \ZZ_{\ge0}} &(\id\otimes \tr_V)
\Big((\half y + \gamma)^\ell\Big)u^{-\ell}
=(\id\otimes \tr_V)\left(\frac{1}{1-(\half y+\gamma)u^{-1}}\right)\\
&=(\id\otimes \tr_V)\left(\frac{1}{1-\half yu^{-1}-\gamma u^{-1}}\right) 
=(\id\otimes \tr_V)\!\left(\Big(\frac{1}{1-\half y u^{-1}}\Big)
\Bigg(\frac{1}{1-\gamma \frac{u^{-1}}{1-\half y u^{-1}} }\Bigg)\right) \\
&=(\id\otimes \tr_V)\!\left(\frac{1}{1-\half y u^{-1}}
\sum_{\ell\in \ZZ_{\ge 0}} \frac{\gamma^\ell u^{-\ell}}{(1-\half y u^{-1})^{\ell}}
\right) 
=u\left(
\sum_{\ell\in \ZZ_{\ge 0}} \frac{(\id\otimes \tr_V)(\gamma^\ell) }{(u-\half y)^{\ell+1}}
\right) \\
&=u\left(
\sum_{\ell\in \ZZ_{\ge 0}} \frac{(\id\otimes \tr_V)(\gamma^\ell) }{(u-\half+\rho'_r)^{\ell+1}}
\right),
\qquad\hbox{where $\rho'_r = \half - \half y$.}
\end{align*}
Then Theorem \ref{z1action} is equivalent to
\begin{align*}
\pi_0&\left(1+\frac{\epsilon u}{\epsilon u-\hbox{$\frac12$}}
\Big(
\sum_{\ell\in \ZZ_{\ge 0}} \frac{(\id\otimes \tr_V)(\gamma^\ell) }{(u-\half+\rho'_r)^{\ell+1}}
\Big)\right) 
=
\pi_0\Big(1+\frac{\epsilon}{\epsilon u-\hbox{$\frac12$}}
\sum_{\ell\in \ZZ_{\ge0}} (\id\otimes \tr_V)\left( \left(\half y + \gamma\right)^\ell\right)u^{-\ell}\Big) 
\\
&= 
\frac{(\epsilon\,u+\half)}{(\epsilon\,u-\half)}
\frac{(u+\half y - r)}{(u-\half y + r)}
\sigma_\rho\left( \prod_{i =1}^r 
\frac{(u+h_i + \half)}
{ (u +h_i - \half)}
\frac{ (u -h_i + \half )}
{(u -h_i - \half)}\right) \\
&= 
\frac{(\epsilon\,u+\half)}{(\epsilon\,u-\half)}
\frac{(u+\half y - r)}{(u-\half y + r)}
\left( \prod_{i =1}^r 
\frac{(u+h_i +\rho_i + \half)}
{ (u +h_i +\rho_i - \half)}
\frac{ (u -h_i - \rho_i + \half )}
{(u -h_i -\rho_i - \half)}\right) \\
&= 
\frac{(u-\half+\half\epsilon+\half)}{(u-\half-\half\epsilon+\half)}
\frac{(u-\half+\half y - r+\half)}{(u-\half-\half y + r+\half))}\!
\left( \prod_{i =1}^r 
\frac{(u-\half+h_i +\rho_i + 1)}
{ (u-\half +h_i +\rho_i)}
\frac{ (u-\half -h_i - \rho_i + 1 )}
{(u-\half -h_i -\rho_i)}\right)
.\end{align*}
Replacing $x= u-\half$,
\begin{align*}
&\pi_0\left(1+\frac{x+\half}{x+\half-\half\epsilon }
\Big(
\sum_{\ell\in \ZZ_{\ge 0}} \frac{(\id\otimes \tr_V)(\gamma^\ell) }{(x+\rho'_r)^{\ell+1}}
\Big)\right) \\
&\qquad= 
\frac{(x+\half\epsilon+\half)}{(x-\half\epsilon+\half)}
\frac{(x+\half y - r+\half)}{(x-\half y + r+\half))}
\left( \prod_{i =1}^r 
\frac{(x+h_i +\rho_i + 1)}
{ (x +h_i +\rho_i)}
\frac{ (x -h_i - \rho_i + 1 )}
{(x -h_i -\rho_i)}\right)
\end{align*}
Since
$$
\frac{(x+\half\epsilon+\half)}{(x-\half\epsilon+\half)}
\frac{(x+\half y - r+\half)}{(x-\half y + r+\half))}
=\begin{cases}
\frac{x+l_0+1}{x+l_0}, &\hbox{if $\fg=\fso_{2r+1}$,} \\
1, &\hbox{if $\fg = \fsp_{2r}$ or $\fso_{2r}$,}
\end{cases}
$$
it follows that
\begin{align*}
\pi_0\left(1+\frac{x+\half}{x+\half-\half\epsilon }
\Big(
\sum_{\ell\in \ZZ_{\ge 0}} \frac{(\id\otimes \tr_V)(\gamma^\ell) }{(x+\rho'_r)^{\ell+1}}
\Big)\right) 
= \prod_{i\in \hat V} \frac{x+l_i+1}{x+l_i}. 
\end{align*}
In combination with \eqref{gammaktoFk}, this demonstrates the equivalence of the Perelomov-Popov
theorem and Theorem \ref{z1action}.
\end{proof}

\begin{remark} Since $\pi_0$ and 
$\rho=\rho_1\varepsilon_1+\cdots+\rho_n\varepsilon_n$ are defined via a specific choice of
positive roots, each side of \eqref{PPformula} depends on that choice,
though the identity does not.  The preferred
choice of positive roots in \cite[p. 139]{Mo} differs from our preferred choice in
\eqref{posroots} by the action of the Weyl group element
$w = (1,-r)(2, -(r-1))\cdots (r-1,-2)(r,-1)$, in cycle notation.
\end{remark}

\subsection{The central elements $Z_V^{(\ell)}$ as quantum higher Casimir elements}
\label{sec:BB-to-Bau}
In this section, we show how the formula for the central elements
$Z_V^{(\ell)}$ in Corollary \ref{cor:ZactionToBaumann}
is related to an  identity for central elements in quantum groups
discovered by Baumann in \cite[Thm.\ 1]{Bau}.  To do this we rewrite the 
Baumann identity for $\fg=\fsp_{2r}$, $\fso_{2r+1}$ and $\fso_{2r}$ and
$\lambda=\varepsilon_1$ in terms of Weyl characters indexed by partitions.
Then a theorem of Turaev and Wenzl computing $(\id \otimes \qtr_{L(\nu)}) (\cR_{21}\cR)$
provides a conversion between the expansion in Corollary \ref{cor:ZactionToBaumann}
and the expansion obtained from Baumann's identity.

For $\lambda\in \fh^*$ define the \emph{Weyl character}
$$s_\lambda = \frac{a_{\lambda+\rho}}{a_\rho},
\qquad\hbox{where}\quad
a_\mu = \sum_{w\in W_0} \det(w) e^{w\mu}.
$$
The expressions $s_\lambda$ and $a_\mu$ are elements of the group algebra
of $\fh^*$,  $\CC[\fh^*] = \CC\text{-span}\{ e^\nu\ |\ \nu\in \fh^*\}$ with $e^\mu e^\nu = e^{\mu + \nu}$.
If $w\in W_0$ then
\begin{equation}\label{slinkyrule}
a_{w\mu} = \det(w)a_\mu
\qquad\hbox{and}\qquad
s_{w\circ\mu} = \det(w)s_\mu,
\end{equation}
where the \emph{dot action} of $W_0$ on $\fh^*$ is given by 
\begin{equation}\label{dotaction}
w\circ \mu = w(\mu+\rho) - \rho,
\qquad\hbox{for $w\in W_0$, $\mu\in \fh^*$.}
\end{equation}
The $W_0$-invariants in $\CC[\fh^*]$ are 
$\CC[\fh^*]^{W_0} =\CC\text{-span}\{
s_\lambda\ |\ \lambda\ \hbox{dominant integral}\}.$
For $\ell\in \ZZ_{\ge 0}$ let
\begin{equation}\label{Psielldefn}
\begin{matrix}
\Psi_\ell\colon &\CC[\fh^*]^{W_0} &\longrightarrow &Z(U_h\fg) \\
&s_\nu &\longmapsto &(\id\otimes \qtr_{L(\nu)})((\cR_{21}\cR)^\ell)
\end{matrix}.
\end{equation}
By \cite[Prop.\ 1.2]{Dr}, the map $\Psi_\ell$ is a vector space isomorphism.
%
%

\begin{thm}\label{Baumann}  \cite[Thm.\ 1]{Bau}
For $\ell\in \ZZ_{\ge 0}$ define
$$m_\lambda^{(\ell)} = \sum_{w\in W_0} q^{2\ell\langle w\lambda, \rho\rangle}
s_{w\lambda}.
\qquad\hbox{Then}\quad
\Psi_\ell(m_\lambda^{(\ell)}) = \Psi_1(m_{\ell\lambda}^{(1/\ell)}).$$
\end{thm}

\begin{cor} \label{Bmneps1} In the same setting as in Theorem \ref{Z1action}, let
$$y
= \begin{cases}
2r, &\hbox{if $\fg = \fso_{2r+1}$,} \\
2r+1, &\hbox{if $\fg = \fsp_{2r}$,} \\
2r-1, &\hbox{if $\fg = \fso_{2r}$,} 
\end{cases}
\qquad
\epsilon\, =
	\begin{cases}
		\!~~~1, &\hbox{if $\fg = \fso_{2r+1}$,} \\
		-1, &\hbox{if $\fg = \fsp_{2r}$,} \\
		\!~~~1, &\hbox{if $\fg = \fso_{2r}$,}
	\end{cases} \qquad
 V=L(\varepsilon_1),
$$
$z=\epsilon q^y$, and let $Z_V^{(\ell)}$ be the 
central elements in the Drinfeld-Jimbo quantum group $U_h\fg$ which are 
given by $Z_V^{(\ell)}= 
\epsilon (\id\otimes \qtr_{L(\varepsilon_1)})((z\cR_{21}\cR)^\ell)$ Then, for $\ell\ge 1$,
$$
Z_V^{(\ell)} 
=\epsilon^{\ell}\Psi_1\left( \begin{array}{l}
	\displaystyle{ c ~+~  z\hspace{-20pt}\sum_{m={\max(\ell-r+1,1)}}^\ell \hspace{-20pt} (-1)^{\ell-m}q^{{-(\ell-2m+1)}}s_{(m,1^{\ell-m+1-1})}}\\
 \displaystyle{\hspace{95pt} 
+ z\hspace{-25pt}\sum_{{m=\max(y-\ell+1,1)}}^{{\ell-y+r}}  \hspace{-25pt}  q^{{-(\ell-2m+1)}}(-1)^{m-\ell+y}s_{{(m,1^{m-\ell+y-1})}}}
\end{array}
\right),
$$
where $c$ is given by
$$ c =  \begin{cases} 1 & \text{ if $\ell$ is even and $\ell < y$,}\\
&\text{ or  $\ell \geq y$ and $\fg=\fso_{2r+1}$,}\\
		0 & \text{otherwise.}\end{cases}$$
\end{cor}
\begin{proof}

The Weyl group $W_0$ for $\fg=\fso_{2r+1}$ or $\fg = \fsp_{2r}$ is 
the group of signed permutations. With positive roots as in \eqref{posroots},
the simple reflections are  $s_i = (i, i+1)$ (the transposition switching $\varepsilon_i$ and
$\varepsilon_{i+1}$, for $i=1,2,\ldots, r-1$) and $s_r = (r,-r)$ (the transposition switching
 the sign of $\varepsilon_r$).  For $\fg=\fso_{2r}$, the Weyl group $W_0$ consists of
signed permutations with an even number of signs, with simple reflections $s_i=(i,i+1)$ for $i=1,2,\ldots, r-1$, and
$s_r = (r-1, -r)(r,-(r-1))$.

To prove the desired identity, we will use the second identity in \eqref{slinkyrule} to relabel the Weyl characters
$s_{w\lambda}$ appearing in 
$$m_{\varepsilon_1}^{(\ell)} 
= \sum_{w\in W_0} q^{2\ell\langle w\varepsilon_1, \rho\rangle}
s_{w\varepsilon_1}
\qquad\hbox{and}\qquad
m_{\ell\varepsilon_1}^{(1/\ell)} 
= \sum_{w\in W_0} q^{2\langle w\varepsilon_1, \rho\rangle}
s_{w\ell\varepsilon_1}
$$
by dominant integral weights.
By \eqref{explrho},
$\rho_{i+1} = \rho_i-1$ and $\rho_r = \half(y-2r+1)$. So
if $\mu= \mu_1\varepsilon_1+\cdots+\mu_r\varepsilon_r$ then
$$s_i\circ \mu = \mu_1\varepsilon_1+\cdots+ \mu_{i-1}\varepsilon_{i-1}
+(\mu_{i+1}-1)\varepsilon_i+(\mu_i+1)\varepsilon_{i+1}
+\mu_{i+2}\varepsilon_{i+2}+\cdots+\mu_r\varepsilon_r$$
for $i=1,2,\ldots, r-1$, and
$$s_r\circ \mu 
=\mu_1\varepsilon_1+\cdots\mu_{r-1}\varepsilon_{r-1}+(-\mu_r-(y-2r+1))\varepsilon_r, \qquad \text{ if $\fg= \fso_{2r+1}$ or $\fsp_{2r}$, and }$$
$$s_r\circ \mu 
=\mu_1\varepsilon_1+\cdots\mu_{r-1}\varepsilon_{r-2}+(-\mu_r-1)\varepsilon_{r-1} + (-\mu_{r-1}-1)\varepsilon_{r}, \qquad \text{ if $\fg= \fso_{2r}$.}\qquad\quad$$

\noindent In particular, 
$s_{\ell\varepsilon_i} = 0$ if $0<\ell<i$, and
\begin{equation}\label{slinky-pos}
s_{\ell\varepsilon_i} = 
s_{s_1s_2\cdots s_{i-1}\circ \ell\varepsilon_i} = (-1)^{i-1}s_{(\ell-i+1)\varepsilon_1
+\varepsilon_2+\cdots+\varepsilon_i}
=(-1)^{i-1}s_{(\ell-i+1,1^{i-1})}
\qquad\hbox{if $ 0< i \le \ell$.}\end{equation}

\noindent Furthermore, if $\fg= \fso_{2r+1}$ or $\fsp_{2r}$, then
\begin{align}
s_i\cdots &s_{r-1}s_rs_{r-1}\cdots s_i\circ(-\ell\varepsilon_i)
= s_i\cdots s_{r-1}s_r\circ ( - (\varepsilon_i+\cdots+\varepsilon_{r-1}) 
- (\ell-(r-i))\varepsilon_r) \nonumber \\
&=s_i\cdots s_{r-1}\circ ( - (\varepsilon_i+\cdots+\varepsilon_{r-1}) 
+ (\ell-(r-i)-(y-2r+1))\varepsilon_r) \nonumber \\
&= (\ell-(r-i)-(y-2r+1)-(r-i))\varepsilon_i = (\ell+2i-y-1)\varepsilon_i.
\label{negterm}
\end{align}
Similarly, if $\fg=\fso_{2r}$, then
\begin{align}
s_i\cdots &s_{r-2}s_rs_{r-1}\cdots s_i\circ(-\ell\varepsilon_i)
= s_i\cdots s_{r-2}s_r\circ ( - (\varepsilon_i+\cdots+\varepsilon_{r-1}) 
- (\ell-(r-i))\varepsilon_r) \nonumber \\
&=s_i\cdots s_{r-2}\circ ( - (\varepsilon_i+\cdots+\varepsilon_{r-2}) 
+ (\ell-(r-i)-1)\varepsilon_{r-1}) \nonumber \\
&= (\ell-(r-i)-1-(r-i-1))\varepsilon_i = (\ell+2i-y-1)\varepsilon_i.
\label{negtermD}
\end{align}

\noindent So, letting $W_{\varepsilon_1}$ be the stabilizer of $\varepsilon_1$,
$|W_{\varepsilon_1}|=2^{r-1}(r-1)!$,
and combining \eqref{slinky-pos} and \eqref{negterm} gives
\begin{align}\label{BLHSeps1}
\frac{1}{|W_{\varepsilon_1}|}m_{\varepsilon_1}^{(\ell)}
&= \sum_{i=1}^r q^{2\ell\langle \varepsilon_i, \rho\rangle} s_{\varepsilon_i}
+ q^{2\ell\langle -\varepsilon_i, \rho\rangle} s_{-\varepsilon_i} 
= \begin{cases}
q^{2 \ell \langle \varepsilon_1, \rho \rangle}s_{\varepsilon_1} + q^{-2 \ell \langle \varepsilon_r, \rho \rangle}s_{\varepsilon_{-r}} & \text{ if $\fg = \fso_{2r + 1}$,}\\
q^{2 \ell \langle \varepsilon_1, \rho \rangle}s_{\varepsilon_1} &  \text{ if $\fg = \fsp_{2r}$ or $\fg = \fso_{2r}$.}
\end{cases}
\nonumber \\
&=
q^{\ell(y-1)}s_{\varepsilon_1} - a
 \qquad \text{ where } a = \begin{cases}q^{-\ell}s_0  & \text{ if } \fg=\fso_{2r+1}\\
			0 & \text{otherwise,}
			\end{cases}
\end{align}
since $1+2i-y-1=0$ has a solution exactly if $i=r$ and $\fg=\fso_{2r+1}$.  If $\ell>0$ then using 
$$\det(s_i\cdots s_{r-1}s_rs_{r-1}\cdots s_i) = -1 = -\det(s_i\cdots s_{r-2}s_rs_{r-1}\cdots s_i) ,$$
 $\ell+2i-y-1-(i-1)=\ell-y+i$, and \eqref{slinky-pos}, equations \eqref{negterm} and  \eqref{negtermD} give
\begin{align*}
q^{2\langle -\varepsilon_i, \rho\rangle}s_{-\ell\varepsilon_i}
&= \begin{cases}
-q^{-\ell}s_0, &\hbox{if $\ell=y+1-2i$ and $\fg=\fso_{2r+1}$ or $\fsp_{2r}$}, \\
q^{-\ell}s_0, &\hbox{if $\ell=y+1-2i$ and $\fg=\fso_{2r}$}, \\
(-1)^{i}q^{-(y-2i+1)}s_{(\ell-y+i, 1^{i-1})},
&\hbox{if $\ell-y+i\geq 1$}, \\
0, &\hbox{otherwise}.
\end{cases}
\end{align*}
Since $\ell-y+i\geq 1$ when $i\ge y-\ell+1$,
\begin{align*}
&\frac{1}{|W_{\varepsilon_1}|}m_{\ell \varepsilon_1}^{(1/\ell)}
= \sum_{i=1}^r
q^{2\langle \varepsilon_i, \rho\rangle} s_{\ell\varepsilon_i}
+q^{2\langle -\varepsilon_i, \rho\rangle} s_{-\ell\varepsilon_i} \\
&=\left(\sum_{i=1}^{\min(\ell,r)} q^{y-2i+1} (-1)^{i-1} s_{(\ell-i+1, 1^{i-1})}\right)
 -b+
\left(\sum_{i=\max(y-\ell+1,1)}^r q^{-(y-2i+1)} (-1)^{i} s_{(\ell-y+i, 1^{i-1})}\right)
\end{align*}
$$\text{where $b=0$ if $\ell \geq y$, and if $\ell<y$ then $b$ is given by }\quad 
\begin{array}{c|c|c}
&\ell \text{ is odd}& \ell \text{ is even}\\ \hline
\fg = \fso_{2r+1}& q^{-\ell}s_0& 0\phantom{\Big|} \\ \hline
\fg=\fsp_{2r} &0  & q^{-\ell}s_0\phantom{\Big|}\\ \hline
\fg=\fso_{2r} &0  & -q^{-\ell}s_0\phantom{\Big|}
\end{array}
$$
(so that $b$ is nonzero exactly when $\ell=y+1-2i$ has a solution with $1\le i\le r$).
Then reindexing (with $m = \ell - i +1$ in the first sum and $m=\ell - y +i$ in the second sum) gives
\begin{align*}
\frac{1}{|W_{\varepsilon_1}|}m_{\ell \varepsilon_1}^{(1/\ell)}
&=-b+ \sum_{m={\max(\ell-r+1,1)}}^\ell (-1)^{\ell-m}q^{{y-2(\ell-m) - 1}}s_{(m,1^{\ell-m+1-1})}\\
&\qquad \qquad \qquad \qquad
+ \sum_{{m=\max(y-\ell+1,1)}}^{{\ell-y+r}} q^{{y-2(\ell-m) - 1}}(-1)^{m-\ell+y}s_{{(m,1^{m-\ell+y-1})} } .
\end{align*}
Notice that the last sum appears only if $\ell>y-r$.

Theorem \ref{Baumann} applied in the case that $\lambda = \varepsilon_1$ gives
\begin{align*}
&Z_V^{(\ell)}=\epsilon(\id\otimes \qtr_V)((z\cR_{21}\cR)^\ell)
=\epsilon z^\ell \Psi_{\ell}(s_{\varepsilon_1})=\epsilon (\epsilon q^y)^\ell  \Psi_\ell\left(q^{-\ell(y-1)}\left(
\frac{1}{|W_{\varepsilon_1}|}m_{\varepsilon_1}^{(\ell)} 
+ a \right)\right)
\nonumber \\
&=\epsilon^{\ell+1}q^\ell \left(\frac{1}{|W_{\varepsilon_1}|}\Psi_\ell\left(
m_{\varepsilon_1}^{(\ell)}\right) + a\right)
=\epsilon^{\ell+1}q^\ell \left(\frac{1}{|W_{\varepsilon_1}|}\Psi_1\left(
m_{\ell\varepsilon_1}^{(1/\ell)}\right) +  a\right)\nonumber\\
&= \epsilon^{\ell}\Psi_1\left( \epsilon q^\ell \left(\frac{m_{\ell\varepsilon_1}^{(1/\ell)}}{|W_{\varepsilon_1}|}
+  a\right)\right)\\
&=\epsilon^{\ell}\Psi_1\left( 
	\epsilon q^\ell (a-b) +  z\hspace{-25pt}\sum_{m={\max(\ell-r+1,1)}}^\ell \hspace{-25pt} (-1)^{\ell-m}q^{{-(\ell-2m+1)}}s_{(m,1^{\ell-m+1-1})}\right.\\
& \hspace{95pt} 
+ \left.z\hspace{-25pt}\sum_{{m=\max(y-\ell+1,1)}}^{{\ell-y+r}}  \hspace{-25pt}  q^{{-(\ell-2m+1)}}(-1)^{m-\ell+y}s_{{(m,1^{m-\ell+y-1})}}
\right),
\end{align*}
since $z = \epsilon q^y$. The result follows since $c=\epsilon q^\ell (a-b)$. 
\end{proof}

We would like to connect Corollary \ref{Bmneps1} to the Harish-Chandra images of the parameters $Z_V^{(\ell)}$ computed in Corollary \ref{cor:ZactionToBaumann}. In order to do so, we will use
the following result from \cite[Lemma 3.5.1]{TW} (also see  \cite[Prop.\ 5.3]{Dr}).

\begin{thm}\label{thm:TW}  Let $\fg$ be a finite-dimensional complex Lie algebra with a symmetric nondegenerate $\ad$-invariant bilinear form and let $U= U_h\fg$
be the corresponding Drinfeld-Jimbo quantum group with $R$-matrix $\cR$.
Let $\nu$ be a dominant integral weight so that the irreducible module $L(\nu)$ of highest 
weight $\nu$ is finite-dimensional and let $s_\nu$ be the Weyl character of $L(\nu)$.  Then
$$\hbox{
$(\id \otimes \qtr_{L(\nu)}) (\cR_{21}\cR)$\quad acts on $L(\mu)$ by\quad 
$\mathrm{ev}_{2(\mu+\rho)} (s_\nu)\id_{L(\mu)}$},
$$
where $\mathrm{ev}_\gamma\colon \CC[\fh^*] \to \CC$ are the algebra homomorphisms given by
$\mathrm{ev}_\gamma(e^{\tau})=q^{\langle \gamma,\tau\rangle}$ for $\gamma,\tau\in \fh^*$.
\end{thm}

For $\fg=\fso_{2r+1}$, $\fsp_{2r}$ or $\fso_{2r}$,
the Turaev-Wenzl identity almost provides an inverse to the Harish-Chandra
homomorphism.  With $\varepsilon_1,\ldots, \varepsilon_r$ as in \eqref{epsdefn},
converting variable alphabets from 
$$Y = \sum_{i \in \hat V} e^{\varepsilon_i} \quad \text{ to } \quad
X = \sum_{i \in \hat V} L_i^2,
\qquad\hbox{then}\qquad
\ev_{2\lambda}(s_{\mu}(Y)) = \ev_{\lambda}(s_\mu(X)).$$
Thus, Theorem \ref{thm:TW} in combination with the Harish-Chandra isomorphism in Theorem \ref{HCIso} says that
$\mathrm{ev}_\mu (\pi_0(\Psi_1(s_\nu)) ) 
= \mathrm{ev}_{2(\mu + \rho)} s_\nu(Y) 
= \mathrm{ev}_{\mu+\rho} s_\nu(X) 
= \mathrm{ev}_{\mu}(\sigma_\rho(s_\nu(X))).$
Hence
\begin{equation}\label{HCinv}
\pi_0(\Psi_1(s_\nu)) = \sigma_\rho(s_\nu(X)).
\end{equation}
The modification rules of \cite[\S 2.4]{KT} are used to convert the
universal Weyl characters appearing in Corollary \ref{cor:ZactionToBaumann} to 
actual Weyl characters $s_\lambda$.  In general, either $sp_\lambda(X)=0$ or there 
is a unique dominant weight $\mu$ and 
a uniquely determined sign such that $sp_\lambda(X) = \pm s_\mu$,
and similarly for the orthogonal cases. 
In particular,  if $\ell(\lambda)<r$ then $sp_\lambda(X) = s_\lambda$ in the symplectic case 
and $so_\lambda(X) = s_\lambda$ in the orthogonal case \cite[Prop. 2.2.1]{KT}.  
In view of \eqref{HCinv}, the conversion from universal Weyl characters to actual Weyl characters 
provides the equivalence between Corollary \ref{Bmneps1} and Corollary \ref{cor:ZactionToBaumann}.

%
%
%
%
%

%
\end{document}